\documentclass[11pt,a4paper]{article}
\usepackage{amssymb, shapepar}
\usepackage{amsfonts}
\usepackage{mathrsfs}
\usepackage{latexsym,bm}
\usepackage{amsthm,amsmath}
\usepackage{indentfirst}
\usepackage{color}
\usepackage{tabularx}

\theoremstyle{plain}\newtheorem{thm}{Theorem}[section]
\theoremstyle{plain}\newtheorem{defn}{Definition}[section]
\theoremstyle{plain}\newtheorem{prop}{Proposition}[section]
\theoremstyle{plain}\newtheorem{cor}{Corollary}[section]
\theoremstyle{plain}\newtheorem{lem}{Lemma}[section]
\theoremstyle{plain}

\numberwithin{equation}{section}

\def\rr{\mathbb{R}}
\def\cc{\mathbb{C}}
\def\rn{\mathbb{R}^{n}}

\def\zz{\mathbb{Z}}
\def\nn{\mathbb{N}}
\def\cdd{\mathcal {D}}
\def\cd'{\mathcal{D}'}
\def\cm{\mathcal {M}}

\newcommand{\cp}{{\mathcal P}}

\def\az{\alpha}
\def\bz{\beta}

\def\dz{\delta}

\def\ez{\epsilon}
\def\fz{\infty}
\def\gz{\gamma}
\def\lz{\lambda}
\def\oz{\Omega}
\def\wz{\omega}
\def\pz{\partial}
\def\rz{\rho}
\def\tz{\theta}
\def\sz{\sigma}
\def\vz{\varphi}
\def\vez{\varepsilon}

\def\ls{\lesssim}

\def\v{$V$}

\def\L{\mathcal{L}}

\def\T{\mathcal{T}}

\def\wha{$H_{\L}^{1}(\omega)$}

\def\whp{$h_{\rho,N}^p(\wz)$}

\def\whaf{\|f\|_{H_{\L}^{1}(\omega)}}

\def\l{\left}
\def\r{\right}

\def\supp{{\rm supp}}
\def\loc{{\rm loc}}
\def\wt{\widetilde}

\def\dsum{\displaystyle\sum}

\def\dint{\displaystyle\int}
\def\dfrac{\displaystyle\frac}
\def\dsup{\displaystyle\sup}

\def\dinf{\displaystyle\inf}

\def\fin{{\mathop\mathrm{fin}}}

\newcommand{\bea}{\begin{align}}
\newcommand{\eea}{\end{align}}
\newcommand{\beq}{\begin{equation}}
\newcommand{\eeq}{\end{equation}}
\newcommand{\beqs}{\begin{eqnarray*}}
\newcommand{\eeqs}{\end{eqnarray*}}
\newcommand{\beqn}{\begin{eqnarray}}
\newcommand{\eeqn}{\end{eqnarray}}
\newcommand{\beqa}{\begin{array}}
\newcommand{\eeqa}{\end{array}}
\newcommand{\noi}{\noindent}

\begin{document}

\title{\bf Weighted local Hardy spaces associated to Schr\"{o}dinger operators
}

\author{  Hua Zhu\footnote{Corresponding author}
\ and Lin Tang}

\date{}
\maketitle

\noi {\bf Abstract.}\quad
In this paper, we characterize the weighted local Hardy spaces $h^p_\rho(\wz)$
related to the critical radius function $\rho$ and weights $\wz\in A_{\fz}^{\rz,\,\fz}(\rn)$
which locally behave as Muckenhoupt's weights and actually include them,
by the local vertical maximal function, the local nontangential maximal function
and the atomic decomposition.
By the  atomic characterization,  we also prove the existence of finite
atomic decompositions associated with
$h^{p}_{\rho}(\wz)$.
Furthermore, we establish boundedness in $h^p_\rho(\wz)$ of
quasi- Banach-valued sublinear operators. As their applications,  we
establish the equivalence  of the  weighted local
Hardy space $h^1_\rho(\wz)$  and the weighted Hardy space $H^1_{\cal L}(\wz)$ associated to Schr\"{o}dinger operators $\cal L$  with $\wz \in A_1^{\rz,\fz}(\rn)$.

\bigskip

\section{Introduction} \label{s1}

\noi The theory of classical local Hardy spaces, originally introduced
by Goldberg \cite{Go},
plays an important role in various field of analysis
and partial differential equations;
see \cite{Bu, Sc, St, Tay, Tr1, Tr2} and their references.
In particular, pseudo-difference operators are bounded on
local Hardy spaces $h^p(\rn)$ for $p\in (0,1]$,
but they are not bounded on Hardy spaces $H^p(\rn)$ for $p\in (0,1]$;
see \cite{Go} (also \cite{Tr1, Tr2}).
In \cite{Bu}, Bui studied the weighted local Hardy space $h_{\wz}^p(\rn)$
with $\wz\in A_{\infty}(\rn)$, where and in what follows,
$A_p(\rn)$ for $p\in[1,\infty]$ denotes the
class of Muckenhoupt's weights; see \cite{Du,GR,Gr,St}
for their definition and properties.

In \cite{Ry}, Rychkov introduced and studied some properties of
the weighted Besov-Lipschitz spaces and
Triebel-Lizorkin spaces with weights that are locally in $A_p(\rn)$
but may grow or decrease exponentially,
which contain Hardy spaces. In particular,  Rychkov \cite{Ry}
generalized some of theories
of weighted local Hardy spaces developed by Bui \cite{Bu} to
$A_{\infty}^{loc}(\rn)$ weights,
where $A_{\infty}^{loc}(\rn)$ weights denote local $A_{\infty}(\rn)$
weights which are non-doubling weights,
and $A_{\infty}^{loc}(\rn)$ weights include $A_{\infty}(\rn)$ weights.
Recently, Tang \cite{Ta1} established the weighted
atomic decomposition characterization of the weighted
local Hardy space $h_{\wz}^p(\rn)$ with $\wz\in A_{\infty}^{loc}(\rn)$
via the local grand maximal function,  and gave some criterions about boundedness of $\mathcal{B}_{\bz}-$sublinear
operators on  $h_{\wz}^p(\rn)$ which was first introduced in \cite{YZ}; meanwhile, Tang \cite{Ta1} also proved that pseudo-difference operators are bounded on
local Hardy spaces $h^p_\wz(\rn)$ for $p\in (0,1]$ by using above criterions and main results in \cite{Ta2}.
Furthermore, Yang-Yang \cite{YY} extended the main results in \cite{Ta1} to  the weighted
local Orlicz-Hardy space $h^{\Phi}_{\omega}(\mathbb{R}^n)$ case by applying similar methods in \cite{Ta1}.

On the other hand, the study of schr\"odinger operator $L=-\triangle+V$ recently
attracted much attention; see \cite{BHS1,BHS2,DZ1,DZ2,Sh,Ta3,Ta4,YZ,YZ1,Zh,ZL}.
In particularly,  J. Dziuba\'{n}ski and J. Zienkiewicz \cite{DZ1,DZ2} studied
Hardy space $H^{1}_{\cal L} $ associated to Schr\"{o}dinger operators $\cal L$ with
potential satisfying reverse H\"{o}lder inequality. Recently, Bongioanni, etc. \cite{BHS1} introduced new classes of weights,
 related to Schr\"odinger operators ${\cal L}$, that is,
 $A_p^{\rz,\fz}(\rn)$ weight
 which are in general larger than Muckenhoupt's
 (see Section 2 for notions of $A_p^{\rz,\fz}(\rn)$ weight). Nature, it is a very interesting problem that whether we can give a atomic characterization for
weighted Hardy space $H^{1}_{\cal L}(\wz)$  with $\wz\in A_1^{\rz,\fz}(\rn)$.

The purpose of  this paper is to give a positive answer. More precisely, we first introduce the  weighted local
Hardy spaces $h^p_\rho(\wz)$ with  $A_q^{\rz,\fz}(\rn)$ weights,
and establish
the  atomic characterization of the  weighted local
Hardy spaces $h^p_\rho(\wz)$ with  $\wz\in A_q^{\rz,\fz}(\rn)$ weights. Then, we  establish the equivalence
between the  weighted local
Hardy spaces $h^1_\rho(\wz)$ and the  weighted Hardy space $H^{1}_{\cal L}(\wz)$ associated to Schr\"{o}dinger operator ${\cal L}$ with $\wz\in A_1^{\rz,\fz}(\rn)$.
In particular, it should be pointed out that we can not directly obtain the atomic characterization of  $H^{1}_{\cal L}(\wz)$  with $A_1^{\rz,\fz}(\rn)$ weights by using the methods in \cite{DZ1,DZ2,DZ3}, which forces us to use the  above weighted local
Hardy spaces $h^1_\rho(\wz)$ theory to overcome the difficulty.

The paper is organized as follows. In Section 2,
we review some notions and notations
concerning the weight classes $A_p^{\rz,\tz}(\rn)$
introduced in \cite{BHS1,Ta3,Ta4}.
In Section \ref{s3}, we first introduce the weighted local
Hardy space $h^{p}_{\rho,\,N}(\wz)$ via the local grand
maximal function, and then the weighted atomic local Hardy
space $h^{p,\,q,\,s}_{\rho}(\wz)$ for any admissible triplet
$(p,\,q,\,s)_{\wz}$ (see Definition \ref{d3.4} below), furthermore, we establish the local vertical and the local nontangential maximal
function characterizations of $h^{p}_{\rho,\,N}(\wz)$ via a local
Calder\'on reproducing formula and some useful estimates established
by Rychkov \cite{Ry}.
In Section 4, we establish the Calder\'on-Zygmund
decomposition associated with the grand maximal function.
In Section 5, we prove that for any given admissible  triplet
$(p,\,q,\,s)_\wz$, $h^p_{\rho,\, N}(\wz)=h_{\rho}^{p,q,s}(\wz)$ with
equivalent norms. It is worth pointing out that we obtain Theorem
\ref{t5.1} by a way different from the methods in \cite{Go,Bu},
but close to those in \cite{Bo,Ta1,YY}. For simplicity, in the
rest of this introduction, we denote by $h^{p}_{\rho}(\wz)$ the
 weighted local Hardy space $h^{p}_{\rho,\,N}(\wz)$.
In Section 6, we prove that
$\|\cdot\|_{h_{\rho,\fin}^{p,q,s}(\wz)}$ and
$\|\cdot\|_{h_\rho^p(\wz)}$ are equivalent quasi-norms on
$h^{p,q,s}_{\rho,\fin}(\wz)$ with $q<\fz$, and we obtain criterions
for boundedness of ${\cal B}_\bz$-sublinear operators in $h^p_\rho(\wz)$.
We remark that this extends both the results of
Meda-Sj\"ogren-Vallarino \cite{MSV} and Yang-Zhou \cite{YZ} to
the setting of weighted local Hardy spaces.
In Section 7, we
apply the  atomic characterization of the  weighted local
Hardy spaces $h_{\rho}^1(\wz)$   to establish atomic characterization of weighted Hardy space $H^{1}_{\cal L}(\wz)$ associated to Schr\"{o}dinger operator $\cal L$ with $A_1^{\rz,\fz}(\rn)$ weights.
% Finally, in Section 8, we will give a characterization of  weighted
% Hardy spaces $H^{1}_{\cal L}(\wz)$ via Riesz tranforms
%  associated to Schr\"{o}dinger operator $\cal L$.

Throughout this paper, we let $C$ denote  constants that are
independent of the main parameters involved but whose value may
differ from line to line. By $A\sim B$, we mean that there exists a
constant $C>1$ such that $1/C\le A/B\le C$.
The symbol $A\ls B$ means that $A\le CB$.
%If $A\ls B$ and $B\ls A$, then we write $A\sim B$.
The  symbol $[s]$ for $s\in\rr$ denotes the maximal integer not more than $s$.
We also set $\nn\equiv\{1,\,2,\, \cdots\}$ and $\zz_+\equiv\nn\cup\{0\}$.
The multi-index notation is usual:
for $\az=(\az_1,\cdots,\az_n)$ and
$\pz^\az=(\pz/\pz_{x_1})^{\az_1}\cdots(\pz/\pz_{x_n})^{\az_n}$.
Given a function $g$ on $\rn$, we let $L_g\in\zz_+$ denote the maximal
number such that $g$ has vanishing moments up to  the order $L_g$, i.e.,
$\int x^{\az}g(x)\,dx=0 $ for all multi-indices $\az$ with $|\az|\le L_g$.
If no vanishing moments of $g$, then we put $L_g=-1$.

\section{Preliminaries}\label{s2}

\noi In this section, we review some notions and notations
concerning the weight classes $A_p^{\rz,\tz}(\rn)$
introduced in \cite{BHS1,Ta3,Ta4}.
Given $B=B(x,r)$ and $\lz>0$, we will write $\lz B$ for
the $\lz$-dilate ball, which is the ball with
the same center $x$ and with radius $\lz r$. Similarly,
$Q(x,r)$ denotes the cube centered at $x$ with
side length $r$ (here and below only cubes with sides
parallel to the axes are considered),
and $\lz Q(x,r)=Q(x,\lz r)$.
Especially, we will denote $2 B$ by $B^*$, and $2Q$ by $Q^*$.

Let $\L=-\Delta+V$ be a  Schr\"{o}dinger operator on $\rn,\ n\geq 3$,
where $V \not\equiv 0$ is a fixed non-negative potential.
We assume that $V$ belongs to the
reverse H\"{o}lder class $RH_s(\rn)$ for some $s \geq n/2$;
that is, there
exists $C = C(s,V) > 0$ such that
\beqs
\l(\dfrac{1}{|B|}\dint_B V(x)^{s}\,dx \r)^{\frac{1}{s}}\leq C \l(\dfrac{1}{|B|}\int_B V(x)\,dx\r),
\eeqs
for every ball $B\subset\mathbb{R}^{n}$.
Trivially, $RH_q(\rn)\subset RH_p(\rn)$ provided $1<p \leq q < \infty$.
It is well known that, if $V\in RH_q(\rn)$ for some $q>1$,
then there exists $\varepsilon>0$, which
depends only on $d$ and the constant $C$ in  above inequality,
such that $V\in RH_{q+\varepsilon}(\rn)$ (see \cite{Ge}).
Moreover, the measure $V(x)\,dx$ satisfies the doubling condition:
\beqs
\dint_{B(y,2r)}V(x)\,dx\leq C\dint_{B(y,r)}V(x)\,dx.
\eeqs

With regard to the Schr\"odinger operator $\L$, we
know that the operators derived from $\L$ behave "locally" quite similar to
those corresponding to the Laplacian (see \cite{DG,Sh}).
The notion of locality is given by the critical radius function
\beq \label{2.1}
\rho(x)=\dfrac 1{m_V(x)}=\dsup_{r>0}\l\{r:\ \dfrac
{1}{r^{n-2}}\dint_{B(x,r)}V(y)dy\le 1\r\}.
\eeq
Throughout the paper we assume that $V\not\equiv0$,
so that $0<\rz(x)<\fz$ (see \cite{Sh}).
In particular, $m_V(x)=1$ with $V=1$
and $m_V(x)\sim (1+|x|)$ with $V=|x|^2$.

\begin{lem} \label{l2.1} {\bf (see \cite{Sh})} \
 There exist $C_0\ge 1$ and $k_0 \geq 1$ so that for all $x,y\in \rn$
\beq \label{2.2}
C_0^{-1}\rz(x)\l(1+\frac{|x-y|}{\rz(x)}\r)^{-k_0}\leq \rz(y)\leq C_0\rz(x)\l(1+\frac{|x-y|}{\rz(x)}\r)^{\frac{k_0}{k_0+1}}.
\eeq
 In particular, $\rz(x)\sim\rz(y)$ when $y\in B(x,r)$
 and $r\leq C\rz(x)$, where $C$ is a positive constant.
\end{lem}

A ball of the form $B(x,\rz(x))$ is called critical, and in what follows we will
call critical radius function to any positive continuous function $\rz$ that
satisfies \eqref{2.1}, not necessarily coming from a potential $V$.
Clearly, if $\rz$ is such a function, so it is $\bz \rz$ for any $\bz>0$.
As the consequence of the above lemma we acquire the following result:

\begin{lem}\label{l2.2} {\bf (see \cite{DZ1})} \
 There exists a sequence of points $x_j\in \rn,\ j\geq 1,$ such that the family
 $B_j=B(x_j,\rz(x_j)),\ j\geq 1$ satisfies:
 \begin{enumerate}
 \item[\rm(a)] \quad $\bigcup_j B_j=\rn$.
 \item[\rm(b)] \quad For every $\sigma\geq 1$ there exist constants $C$ and $N_1$
 such that $\Sigma_j \chi_{\sz B_j}\le C\sz^{N_1}$.
\end{enumerate}
\end{lem}

In this paper, we write $\Psi_\tz(B)=(1+r/\rho(x_0))^\tz$, where
 $\tz \ge 0$, $x_0$ and $r$ denotes the center and radius of $B$ respectively.

A weight  always refers to a positive function which is locally
integrable.
As in \cite{BHS1}, we say that a weight $\wz$ belongs to the class
$A_p^{\rho,\tz}(\rn)$ for $1<p<\fz$, if there is a constant $C$ such that for
all balls
  $B$
$$\l(\dfrac 1{\Psi_\tz(B)|B|}\dint_B\wz(y)\,dy\r)
\l(\dfrac 1{\Psi_\tz(B)|B|}\dint_B\wz^{-\frac 1{p
-1}}(y)\,dy\r)^{p-1}\le C.$$
We also
say that a  nonnegative function $\wz$ satisfies the $A_1^{\rho,\tz}(\rn)$
condition if there exists a constant $C$ such that
$$M_{V,\tz}(\wz)(x)\le C \wz(x), \ a.e.\ x\in\rn.$$
 where
$$M_{V,\tz}f(x)\equiv\dsup_{x\in B}\dfrac 1{\Psi_\tz(B)|B|}\dint_B|f(y)|\,dy.$$
When $V=0$, we
denote $M_{0}f(x)$ by $Mf(x)$ (the standard Hardy-Littlewood
maximal function). It is easy to see that $|f(x)|\le M_{V,\tz} f(x)\le
Mf(x)$ for $a.e.\ x\in\rn$ and any $\tz\ge 0$.

Clearly, the classes $A_p^{\rho,\tz}$ are increasing with $\tz$, and
we denote $A_p^{\rz,\fz}=\bigcup_{\tz\ge 0}A^{\rz,\tz}_{p}$.
By H\"older's inequality, we see that
$A^{\rz,\tz}_{p_1} \subset A^{\rz,\tz}_{p_2}$, if $1\le p_1<p_2<\fz$,
and we also denote $A_{\fz}^{\rz,\fz}=\bigcup_{p\ge 1}A^{\rz,\fz}_{p}$.
In addition, for $1\leq p\leq \fz$,
denote by $p'$ the adjoint number of $p$, i.e. $1/p+1/p'=1$.

Since $\Psi_\tz(B)\ge 1$ with $\tz\ge 0$, then
$A_p\subset A_p^{\rho,\tz}$ for $1\le p<\fz$, where $A_p$ denotes
the classical Muckenhoupt weights; see \cite{GR} and \cite{Mu}.
Moreover, the inclusions are proper.
In fact, as the example given in \cite{Ta3},
let $\tz>0$ and $0\le\gz\le\tz$,
it is easy to check that
$\wz(x)=(1+|x|)^{-(n+\gz)}\not\in A_\fz=\bigcup_{p\ge 1}A_p$ and
$\wz(x)dx$ is not a doubling measure, but
$\wz(x)=(1+|x|)^{-(n+\gz)}\in A_1^{\rho,\tz}$
provided that  $V=1$ and
$\Psi_\tz (B(x_0,r))=(1+r)^\tz$.

In what follows, given a Lebesgue measurable set $E$ and a weight $\wz$,
$|E|$ will denote the Lebesgue measure of $E$ and $\wz(E):=\int_E \wz(x)\,dx$.
For any
$\wz\in A^{\rz,\fz}_{\fz}$, the space $L^p_{\wz}(\rn)$
with $p\in(0,\fz)$ denotes the set of all measurable functions $f$
such that
$$\|f\|_{L^p_{\wz}(\rn)}\equiv\left(\int_{\rn}|f(x)|^p
\wz(x)\,dx\r)^{1/p}<\fz,$$
and $L^{\fz}_{\wz} (\rn)\equiv L^{\fz}(\rn)$.
The  symbol $L^{1,\,\fz}_{\wz}(\rn)$ denotes the
set of all measurable functions $f$ such that
$$\|f\|_{L^{1,\,\fz}_{\wz}(\rn)}\equiv\sup_{\lz>0}\left\{\lz
\wz(\{x\in\rn:\,|f(x)|>\lz\})\r\}<\fz.$$
We  define the local Hardy-Littlewood maximal operator by
\beq \label{Mloc}
M^{loc}f(x)\equiv \sup _{\substack {x\in B(x_0,r) \\ r\le \rz(x_0)}}
\frac{1}{|B|}\int_{B}|f(y)|\,dy.
\eeq

 We remark that  balls can be replaced by cubes  in definition of
$A_p^{\rho,\tz}$ and $M_{V,\tz}$, since
$\Psi(B)\le \Psi(2B)\le 2^\tz \Psi(B).$
In fact, for the cube $Q=Q(x_0,r)$, we can also define
$\Psi_\tz(Q)=(1+r/\rho(x_0))^\tz$.
Then  we give the weighted boundedness of $M_{V,\tz}$.

\begin{lem} \label{l2.3} {\bf (see \cite{Ta3})} \
Let $1<p<\fz$, $p'=p/(p-1)$ and assume that $\wz\in A_p^{\rho,\tz}$.
 There
exists a constant $C>0$ such that
$$\|M_{V,p'\tz}f\|_{L^p_{\wz}(\rn)}\le C\|f\|_{L^p_{\wz}(\rn)}.$$
\end{lem}

Next, we give some properties of weights class $A^{\rz,\tz}_{p}$ for $p\ge1$.

\begin{lem}\label{l2.4}
Let $\wz\in A^{\rho,\fz}_p=\bigcup_{\tz\ge 0}A_p^{\rho,\tz}$ for $p\ge 1$.
Then
\begin{enumerate}
 \item[\rm(i)] If $ 1\le p_1<p_2<\fz$, then $A_{p_1}^{\rho,\tz}\subset
A_{p_2}^{\rho,\tz}$.

 \item[\rm(ii)] $\wz\in A_p^{\rho,\tz}$ if and only
if\ $\wz^{-\frac 1{p-1}}\in A_{p'}^{\rho,\tz}$, where $1/p+1/p'=1.$

 \item[\rm(iii)] If $\wz\in A_p^{\rho,\fz},\ 1<p<\fz$, then there exists
$\ez>0$ such that $\wz\in A_{p-\ez}^{\rho,\fz}.$

 \item[\rm(iv)]Let $f\in L_{loc}(\rz)$, $0<\dz<1$, then
$(M_{V,\tz}f)^\dz\in A_1^{\rho,\tz}$.

 \item[\rm(v)] Let $1<p<\fz$, then $\wz\in A_p^{\rho,\fz}$ if and only
if $\wz=\wz_1\wz_2^{1-p}$, where $\wz_1,\wz_2\in A_1^{\rho,\fz}$.

 \item[\rm(vi)] For $\wz\in A_p^{\rho,\tz}$,  $Q=Q(x,r)$ and $\lz>1$,  there
 exists a positive constant $C$ such that
 $$\wz(\lz Q)\le C (\Psi_{\tz}(\lz Q))^p\, \lz^{np}\, \wz(Q).$$

\item[\rm(vii)] If $p\in(1,\fz)$ and $\wz\in A^{\rz,\tz}_{p}(\rn)$,
then the local Hardy-Littlewood maximal operator $M^{\loc}$ is
bounded on $L^p_{\wz}(\rn)$.

\item[\rm(viii)] If $\wz\in A^{\rz,\tz}_{1}(\rn)$, then $M^{\loc}$ is
bounded from $L^1_{\wz}(\rn)$ to $L^{1,\,\fz}_{\wz}(\rn)$.

\end{enumerate}

\end{lem}

\begin{proof}
(i)-(viii) have been proved in \cite{BHS1, Ta4}.
\end{proof}

For any $\wz\in A^{\rz,\fz}_{\fz}(\rn)$, define the  critical
index of $\wz$ by
\beq \label{qw}
q_{\wz}\equiv\inf\left\{p\in[1,\fz):\,\wz\in A^{\rz,\fz}_p(\rn)\r\}.
\eeq
Obviously, $q_\wz\in [1,\fz)$. If $q_\wz\in(1,\fz)$, then
$\wz\not\in A_{q_\wz}^{\rz,\fz}$.

The symbols ${\cal D}(\rn)=C_0^\fz(\rn), {\cal D}'(\rn)$ is the
dual space of ${\cal D}(\rn)$, and for ${\cal D}(\rn),\,{\cal D}'(\rn)$
and $L^p_{\wz}(\rn)$,
we have the following conclusions.

\begin{lem}\label{l2.5}
Let $\wz\in A^{\rz,\fz}_{\fz}(\rn)$, $q_{\wz}$ be as in (2.4) and
$p\in(q_{\wz},\fz]$.
\begin{enumerate}
 \item[\rm(i)] If $\frac{1}{p}+\frac{1}{p'}=1$, then
${\cal {D}}(\rn)\subset L^{p'}_{\wz^{-1/(p-1)}}(\rn)$.

\item[\rm(ii)] $L^{p}_{\wz}(\rn)\subset\cd' (\rn)$ and the
inclusion is continuous.
\end{enumerate}
\end{lem}

By the same method as the proof of Lemma 2.2 in \cite{Ta1}, we can
get the Lemma 2.5, and we omit the details here.

For any $\vz\in {\cal D}(\rn)$, let
$\vz_t(x)=t^{-n}\vz\l(x/t\r)$ for $t>0$ and
$\vz_j(x)=2^{jn}\vz\l(2^jx\r)$ for $j\in \zz$.
It is easy to see that we have the following results.

\begin{lem}\label{l2.6} {\rm\bf (see \cite{Ta1})}\
Let $\vz\in {\cal D}(\rn)$ and $\dint_{\rn} \vz(x)dx=1$.
\begin{enumerate}
\item[\rm(i)] For any $\Phi\in {\cal D}(\rn)$ and $f\in {\cal D}'(\rn)$,
$\Phi*\vz_t\to\Phi$ in ${\cal D}(\rn)$ as $t\to0$, and
$f*\vz_t\to f$ in ${\cal D}'(\rn)$ as $t\to0$.

\item[\rm(ii)] Let
$\wz\in A_\fz^{\rz,\fz}$ and $q_\wz$ be as in (2.3). If $q\in
(q_\wz,\fz)$, then for any $f\in L^q_\wz(\rn)$, $f*\vz_t\to f$ in
$L^q_\wz(\rn)$ as $t\to0$.
\end{enumerate}
\end{lem}

\section{Weighted local Hardy spaces and their maximal
function characterizations}\label{s3}

\noi In this section, we introduce the weighted local Hardy spaces \whp\,
via the local grand
maximal function and establish its local vertical and nontangential
maximal function
characterizations via a local Calder\'on reproducing formula.
We also introduce the weighted atomic local Hardy space
$h^{p,q,s}_{\rho}(\wz)$ and give some basic properties of these
spaces.

We first introduce some local maximal functions. For $N\in\zz_+$
and $R\in(0,\fz)$, let
$$\begin{array}{cl}
\cdd_{N,\,R}(\rn)\equiv \Bigg\{\vz\in\cdd(\rn):&\,\supp(\vz)
 \subset B(0,R), \\
&\left.\|\vz\|_{\cdd_{N}(\rn)}\equiv\dsup_{x\in\rn}
\dsup_{{\az\in\zz^n_+},\,{|\az|\le N}}
|\partial^{\az}\vz(x)|\le1\r\}.
\end{array}$$

\begin{defn}\label{d3.1}
Let $N\in\zz_+$ and $R\in(0,\fz)$. For any $f\in\cd'(\rn)$, the
local nontangential grand maximal function $\wt{\cm}_{N,\,R} (f)$
of $f$ is defined by setting, for all $x\in\rn$,
\beq \label{tgnr}
\wt{\cm}_{N,\,R} (f)(x)\equiv\sup\left\{|\vz_l \ast
f(z)|:\,|x-z|<2^{-l}<\rz(x),\,\vz\in\cdd_{N,\,R}(\rn)\r\},
\eeq
and the  local vertical grand maximal function $\cm_{N,\,R}(f)$
of $f$ is defined by setting, for all $x\in\rn$,
\beq \label{gnr}
{\cm}_{N,\,R} (f)(x)\equiv\sup\left\{|\vz_l \ast
f(x)|:\,0<2^{-l}< \rz(x),\,\vz\in\cdd_{N,\,R}(\rn)\r\}.
\eeq

\end{defn}

For convenience's sake, when $R=1$, we denote $\cdd_{N,\,R}(\rn)$,
$\wt{\cm}_{N,\,R} (f)$ and $\cm_{N,\,R}(f)$ simply by
$\cdd^0_{N}(\rn)$, $\wt{\cm}^0_{N}(f)$ and $\cm^0_{N}(f)$,
respectively;
when $R=\max\{ R_1,\ R_2,\ R_3 \}>1$
(in which $R_1$, $R_2$ and $R_3$  are defined as in Lemma 4.2, 4.4 and 4.8),
we denote $\cdd_{N,\,R}(\rn)$,
$\wt{\cm}_{N,\,R} (f)$ and $\cm_{N,\,R}(f)$ simply by
$\cdd_{N}(\rn)$, $\wt{\cm}_{N}(f)$ and $\cm_{N}(f)$, respectively.
For any $N\in\zz_+$ and $x\in\rn$, obviously,
$$\cm^{0}_N(f)(x)\le\cm_N (f)(x)\le\wt{\cm}_N (f)(x).$$

For the local grand maximal function $\cm^0_N (f)$, we have the
following Proposition \ref{p3.1}, which can be proved by the same method
as in \cite[Proposition 2.2]{Ta1}. Here and in what follows,
the space $L^1_{\loc}(\rn)$ denotes the set of all locally
integrable functions on $\rn$.

\begin{prop} \label{p3.1}
Let $N\ge2$. Then
\begin{enumerate}
\item[\rm(i)] There exists a positive constant $C$ such that
for all $f\in L^1_{\loc}(\rn)\cap\cd'(\rn)$ and almost every $x\in\rn$,
$$|f(x)|\le \cm^0_N (f)(x)\le C M^{\loc}(f)(x).$$

\item[\rm(ii)] If $\wz\in A_p^{\rz,\tz}(\rn)$ with $p\in(1,\fz)$, then
$f\in L^p_{\wz}(\rn)$ if and only if $f\in \cd'(\rn)$ and $\cm^0_N
(f)\in L^p_{\wz}(\rn)$; moreover,
$$\|f\|_{L^p_{\wz}(\rn)}\sim\|\cm^0_N (f)\|_{L^p_{\wz}(\rn)}.$$

\item[\rm(iii)] If $\wz\in A_1^{\rz,\tz}(\rn)$, then $\cm^0_N$ is
bounded from $L^1_{\wz}(\rn)$ to $L^{1,\fz}_{\wz}(\rn)$.
\end{enumerate}
\end{prop}

Now we introduce the weighted local Hardy space via the local
grand maximal function as follows.

\begin{defn}\label{d3.2}
Let $\wz\in A^{\rz,\fz}_{\fz}(\rn)$, $q_{\wz}$ be as in \eqref{qw}, $p\in(0,1]$
and $\wt{N}_{p,\wz}\equiv[n(\frac {q_\wz}p-1)]+2.$
For each $N\in\nn$ with $N\ge\wt{N}_{p,\,\wz}$,
the  weighted local Hardy space is defined by
$$h^{p}_{\rho,\,N}(\wz)\equiv\left\{f\in\cd'(\rn):\ \cm_N (f)\in
L^{p}_{\wz}(\rn)\r\}.$$
Moreover, let
$\|f\|_{h^{p}_{\rho,\,N}(\wz)}\equiv\|\cm_N(f)\|_{L^{p}_{\wz}(\rn)}$.

\end{defn}

Obviously, for any integers $N_1$ and $N_2$
with $N_1\ge N_2\ge \wt{N}_{p,\,\wz}$,
$$h^{p}_{\rho,\,\wt{N}_{p,\wz}}(\wz)\subset
h^{p}_{\rho,\,N_2}(\wz) \subset h^{p}_{\rho,\,N_1}(\wz),$$
and the inclusions are continuous.

Next, we introduce the weighted local atoms, via which,
we give the definition of
the weighted atomic local Hardy space.

\begin{defn}\label{d3.3}
Let $\wz\in A^{\rz,\fz}_{\fz}(\rn)$, $q_{\wz}$ be as in \eqref{qw}.
A triplet $(p,q,s)_\wz$ is called to be admissible, if $p\in (0,1]$,
$q\in(q_\wz,\fz]$ and $s\in\nn$ with $s\ge [n(q_\wz/p-1)]$.
A function $a$ on $\rn$ is said to be a $(p,q,s)_\wz-atom$ if
\begin{enumerate}
\item[\rm(i)] $\supp \, a\subset Q(x,r)$ and $r\le L_{1}\rz(x)$,
\item[\rm(ii)] $\|a\|_{L^q_\wz(\rn)}\le [\wz( Q)]^{1/q-1/p}$,
\item[\rm(iii)]$\dint_{\rn} a(x)x^\az dx=0$ for all $\az\in \zz_+^n$ with
$|\az|\le s$, when $Q=Q(x,r),\ r<L_{2}\rz(x)$,
\end{enumerate}
where $L_1\equiv4 C_0 (3\sqrt{n})^{k_0}$, $L_2\equiv1/C_0^2(3\sqrt{n})^{k_0+1}$,
and $C_0$, $k_0$ are constant given in Lemma \ref{2.1}.
 Moreover, a function $a(x)$ on $\rn$ is called  a
$(p,\,q)_{\wz}$-single-atom with $q\in(q_{\wz},\fz]$, if
$$\|a\|_{L^q_{\wz}(\rn)}\le[\wz(\rn)]^{1/q-1/p}.$$
\end{defn}

\begin{defn}\label{d3.4}
Let $\wz\in A^{\rz,\fz}_{\fz}(\rn)$, $q_{\wz}$ be as in \eqref{qw},
and $(p,q,s)_\wz$ be admissible, The  weighted atomic local Hardy space
$h^{p,\,q,\,s}_{\rho}(\wz)$ is defined as the set of all
$f\in\cd'(\rn)$ satisfying that
$$f=\sum_{i=0}^{\fz}\lz_i a_i$$
in $\cd'(\rn)$, where $\{a_i\}_{i\in\nn}$ are
$(p,\,q,\,s)_{\wz}$-atoms with $\supp(a_i)\subset Q_i$, $a_0$ is a
$(p,\,q)_{\wz}$-single-atom, $\{\lz_i\}_{i\in\zz_+}\subset\cc$.
Moreover, the quasi-norm of $f\in h^{p,q,s}_{\rho}(\wz)$ is defined by
$$\|f\|_{h^{p,q,s}_{\rho}(\wz)}\equiv
\dinf\l\{\l[\dsum_{i=0}^\fz|\lz_i|^p\r]^{1/p}\r\},$$
where the infimum is taken over all the decompositions of $f$ as
above.
\end{defn}

It is easy to see that if  triplets $(p,q,s)_\wz$ and $(p,\bar
q,\bar s)_\wz$ are admissible and  satisfy $\bar q\le q$ and $\bar
s\le s$, then $(p,q,s)_\wz$-atoms are $(p,\bar q,\bar
s)_\wz$-atoms, which  implies that
$h^{p,q,s}_{\rho}(\wz)\subset h^{p,\bar q,\bar s}_{\rho}(\wz)$ and
the inclusion is continuous.

Next, we introduce some local vertical, tangential and nontangential
maximal functions, and then we establish the characterizations of the
weighted local Hardy space $h^{p}_{\rho,\,N}(\wz)$ by
these local maximal functions.

\begin{defn}\label{d3.5}
Let
\beq \label{3.3}
\psi_0\in\cdd(\rn)\,\, \text{with}\,\,\int_{\rn}\psi_0 (x)\,dx\neq0.
\eeq
For every $x\in\rn$, there exists an integer $j_x\in\zz$ satisfying
$2^{-j_x}< \rz(x)\le 2^{-j_x+1}$,
and then for $j\ge j_x$, $A,\,B\in[0,\fz)$ and $y\in\rn$, let
$m_{j,\,A,\,B,\,x}(y)\equiv(1+2^j |y|)^A 2^{{B|y|}/{\rz(x)}}$.

The  local vertical maximal function $\psi_0^{+}(f)$ of $f$ associated to
$\psi_0$ is defined by setting, for all $x\in\rn$,
\beq \label{3.4}
\psi_0^{+}(f)(x)\equiv \sup_{j\ge j_x}|(\psi_0)_j \ast f(x)|,
\eeq
the local tangential Peetre-type  maximal function
$\psi^{\ast\ast}_{0,\,A,\,B}(f)$ of $f$ associated to $\psi_0$ is
defined by setting, for all $x\in\rn$,
\beq \label{3.5}
\psi^{\ast\ast}_{0,\,A,\,B}(f)(x)\equiv\sup_{j\ge j_x,\,y\in\rn}
\frac{|(\psi_0)_j \ast f(x-y)|}{m_{j,\,A,\,B,\,x}(y)},
\eeq
and the  local nontangential maximal function
$(\psi_0)^{\ast}_{\triangledown}(f)$ of $f$ associated to $\psi_0$ is
defined by setting, for all $x\in\rn$,
\beq \label{3.6}
(\psi_0)^{\ast}_{\triangledown}(f)(x)\equiv
\sup_{|x-y|<2^{-l}<\rz(x)}|(\psi_0)_l \ast f(y)|,
\eeq
where $l\in\zz$.

\end{defn}

Obviously, for any $x\in\rn$, we have
$$\psi_0^{+}(f)(x)\le(\psi_0)^{\ast}_{\triangledown}(f)(x)
\ls \psi^{\ast\ast}_{0,\,A,\,B}(f)(x).$$
We point out that the local
tangential Peetre-type maximal function
$\psi^{\ast\ast}_{0,\,A,\,B}(f)$ was introduced by Rychkov
\cite{Ry}.

In order to characterize $h^p_{\rz,N}(\wz)$ by the local vertical and the
local nontangential maximal function, we need to establish
some relations in the norm of $L^{p}_{\wz}(\rn)$
of the local maximal functions
$\psi^{\ast\ast}_{0,\,A,\,B}(f),\,\psi_0^{+}(f)$ and
$\wt{\cm}_{N,\,R}(f)$, which further imply the desired
characterizations.

We begin with a lemma on local reproducing formula,
 which can be deduced from the Lemma 1.6 in \cite{Ry}, and
 we omit the details of its proof here.

\begin{lem}\label{l3.1}
Let  $\psi_0$ be as in \eqref{3.3} and
$\psi(x)\equiv\psi_0(x)-(1/{2^n})\psi_0 (x/{2})$ for all $x\in\rn$.
Then for any given integers  $j\in\zz$ and $L\in\zz_+$,
there exist $\vz_0,\,\vz\in\cdd(\rn)$
such that $L_{\vz}\ge L$ and
\beq \label{3.7}
f=(\vz_0)_j\ast(\psi_0)_j \ast f+\sum_{k=j+1}^{\fz}
\vz_k\ast\psi_k \ast f \eeq
in $\cd'(\rn)$ for all $f\in\cd'(\rn)$.
\end{lem}

\begin{lem}\label{l3.2}
Let $0<r<\fz$, $\psi_0$ be as in \eqref{3.3} and
$\psi(x)\equiv\psi_0(x)-(1/{2^n})\psi_0 (x/{2})$. Then there exists a
positive constant $A_0$ depending only on the support of $\psi_0$ such that
for any $A\in (A_0,\fz)$ and $B\in [0,\fz)$, there exists a positive constant
$C$ depending only on $n,\,r,\,\psi_0,\,A$ and $B$, such that for all
$f\in\cd'(\rn)$, $x,\,x_0\in\rn$ and $j\ge j_{x_0}$
(where $2^{-j_{x_0}}< \rz(x_0)\le 2^{-j_{x_0}+1}$),
we have
\beq \label{3.8}
|\psi_j *f(x)|^r\le C \sum_{k=j}^{\fz} 2^{(j-k)Ar} 2^{kn}
\int \frac{|\psi_k*f(x-y)|^r}{m_{j,Ar,Br,x_0}(y)}\,dy. \eeq
\end{lem}

\begin{proof}
By Lemma \ref{l3.1}, we can find $\vz_0,\,\vz\in\cdd(\rn)$
so that $L_{\vz}\ge A$
and \eqref{3.7} is true. Hence, we have
\beq \label{3.9}
\psi_j*f=(\vz_0)_j\ast(\psi_0)_j * \psi_j \ast f+\sum_{k=j+1}^{\fz}
\psi_j*\vz_k\ast\psi_k \ast f.
\eeq
The function $\psi_j*\vz_k$ ($k\ge j+1$) have support size $\le C 2^{-j}$
and enjoy the uniform estimate
\beq \label{3.10}
\|\psi_j*\vz_k\|_{L^{\fz}(\rn)}\le C2^{(j-k)A} 2^{jn}, \eeq
which can be easily deduced by the moment condition on $\vz$
(see \cite[(2.13)]{Ry}).
Therefore, we may write
\beq \label{3.11}
|\psi_j*\vz_k(y)|\le C \frac{2^{(j-k)A}2^{kn}}{m_{j,A,B,x_0}(y)}
 \quad (y\in \rn).
\eeq
Putting \eqref{3.11} together with the similar estimate for
$(\vz_0)_j\ast(\psi_0)_j$
into \eqref{3.9} gives \eqref{3.8} for $r=1$, and the case
$r>1$ follows by H\"older's inequality.
To obtain the case $r<1$, we introduce the maximal functions
$$M_{A,B,\,x_0}(x,j)=\sup_{k\ge j,\,y\in\rn}
2^{(j-k)A}\frac{|\psi_k*f(x-y)|}{m_{j,A,B,\,x_0}(y)}.
$$
The \eqref{3.8} with $r=1$ gives
\beq
2^{(j-k)A}|\psi_k*f(x-y)|\le
C \sum_{l=k}^{\fz}2^{(j-l)A}2^{ln}
\int \frac{|\psi_l*f(x-z)|}{m_{k,A,B,\,x_0}(z-y)}\,dz,
\eeq
and the right of (3.12) decreases as $k$ increases.
Hence, to get the estimate for $M_{A,B,\,x_0}(x,j),$
we may only consider (3.12) with $k=j$.
Combing with the elementary inequality
\beq \label{3.13}
m_{j,A,B,\,x_0}(z)\le m_{j,A,B,\,x_0}(y)m_{k,A,B,\,x_0}(z-y).\eeq
we can get
\begin{align} \label{3.14}
M_{A,B,\,x_0}(x,j)
&\le C\dsum_{k=j}^{\fz}2^{(j-k)A}2^{kn}
\dint \dfrac{|\psi_l*f(x-z)|}{m_{j,A,B,\,x_0}(z)}\,dz \nonumber \\
&\le C M_{A,B,\,x_0}(x,j)^{1-r}
\dsum_{k=j}^{\fz}2^{(j-k)Ar}2^{kn}
\dint \dfrac{|\psi_l*f(x-z)|^r}{m_{j,Ar,Br,\,x_0}(z)}\,dz.
\end{align}

Considering $|\psi_j*f(x)|\le M_{A,B,\,x_0}(x,j)$, (3.14) implies (3.8),
if $M_{A,B,\,x_0}(x,j)<\fz.$
By \cite[Proposition\,2.3.4(a)]{Gr}, for any $f\in \cd'(\rn)$,
we have $M_{A,B,\,x_0}(x,j)<\fz$ for all $x\in\rn$ and $j\ge j_{x_0}$,
provided $A>A_0$, where $A_0$ is a positive constant
depending only on the support of $\psi_0$.
This finishes the proof.
\end{proof}

For $f\in L^1_{\loc}(\rn)$, $B\in[0,\fz)$ and $x\in\rn$, let
\beq \label{3.15}
K_Bf(x)=\frac{1}{(\rz(x))^n}\int_{\rn} |f(y)|2^{-B\frac{|x-y|}{\rz(x)}}\,dy,
\eeq
and for the operator $K_B$, we have the following lemma:

\begin{lem}\label{l3.3}
Let $p\in(1,\fz)$ and $\wz\in A_p^{\rz,\tz}(\rn)$, then there exist
constants $C>0$
and $B_0\equiv B_0(\wz,n)>0$
such that for all $B>B_0/p$,
$$\|K_Bf\|_{L^p_{\wz}(\rn)}\le C \|f\|_{L^p_{\wz}(\rn)},$$
for all $f\in L^p_{\wz}(\rn).$
\end{lem}

\begin{proof}
It is suffice to show that there exists a constant $C>0$ such that
for all $B>B_0$,
$$K_Bf(x)\le C M_{V,p'\tz}f(x),$$
then combining with Lemma \ref{l2.3}, we get the boundedness
of the operator $K_B$.

To control $K_Bf(x)$, we argue as follows:
\beqs
\begin{aligned}
K_Bf(x)&=\frac{1}{(\rz(x))^n}\int_{\rn}
|f(y)|2^{-B\frac{|x-y|}{\rz(x)}}\,dy\\
&=\frac{1}{(\rz(x))^n}\int_{|y-x|<\rz(x)}
|f(y)|2^{-B\frac{|x-y|}{\rz(x)}}\,dy
+\frac{1}{(\rz(x))^n}\int_{|y-x|\ge \rz(x)}
|f(y)|2^{-B\frac{|x-y|}{\rz(x)}}\,dy\\
&=\frac{1}{(\rz(x))^n}\int_{|y-x|<\rz(x)}
|f(y)|2^{-B\frac{|x-y|}{\rz(x)}}\,dy\\
&\qquad\qquad +\sum_{k=0}^{\fz}
\frac{1}{(\rz(x))^n}\int_{|y-x|\sim 2^k \rz(x)}
|f(y)|2^{-B\frac{|x-y|}{\rz(x)}}\,dy\\
&\equiv I_1+I_2.
\end{aligned}
\eeqs
For $I_1$, it is easy to get
$$ I_1\le \dfrac{C}{\Psi_{p'\tz }(B_1)|B_1|} \int_{B_1}|f(y)|\,dy
 \le C M_{V,p'\tz} f(x), $$
in which $B_1=B(x,\rz(x))$ is a critical ball.

\noi For $I_2$, we have
\beqs
\begin{aligned}
I_2&\le C\dsum_{k=0}^{\fz}
\dfrac{(1+2^{k+1})^{p'\tz} 2^{kn}}{2^{B2^k}}
\dfrac{1}{\Psi_{p'\tz }(2^{k+1}B_1)|2^{k+1}B_1|}
 \int_{2^{k+1}B_1}|f(y)|\,dy\\
& \le C\l(\dsum_{k=0}^{\fz} \dfrac{(1+2^{k+1})^{p'\tz} 2^{kn}}{2^{B2^k}}\r)
   M_{V,p'\tz} f(x)\\
&\le C  M_{V,p'\tz} f(x),
\end{aligned}
\eeqs
where the sum converges when $B>B_0/p$.
\end{proof}

\begin{lem}\label{l3.4}
Let $\psi_0$ be as in \eqref{3.3} and $r\in(0,\fz)$. Then
for any $A\in(\max\{A_0,n/r\},\fz)$ (where $A_0$ is as in Lemma 3.2)
and
$B\in[0,\fz)$, there exists a positive constant $C$, depending only
on $n,\,r,\,\psi_0,\,A$ and $B$, such that for all
$f\in\cd'(\rn)$, $x\in\rn$ and $j\ge j_x$
(where $2^{-j_{x}}< \rz(x)\le  2^{-j_x+1}$),
\beqs
\begin{aligned}
\left[(\psi_0 )^{\ast}_{j,\,A,\,B}(f)(x)\r]^r
&\le C\dsum_{k=j}^{\fz}
2^{(j-k)(Ar-n)}\left\{M^{\loc}(|(\psi_0 )_k \ast
f|^r)(x)\r.\\
&+K_{Br}(|(\psi_0 )_k \ast f|^r)(x)\Big\},
\end{aligned}
\eeqs
where
$$
(\psi_0 )^{\ast}_{j,\,A,\,B}(f)(x)\equiv \sup_{y\in \rn}
\dfrac{|(\psi_0)_j * f(x-y)|}{m_{j,\,A,\,B,\,x}(y)}
$$
for all $x\in \rn$.
\end{lem}

\begin{proof}
First we can get the stronger version of \eqref{3.8} by virtue of
\eqref{3.13}, that is:
\beqs
\begin{aligned}
\left[(\psi_0 )^{\ast}_{j,\,A,\,B}(f)(x)\r]^r
&\le C\sum_{k=j}^{\fz} 2^{(j-k)Ar}2^{kn}
\int_{\rn} \frac{|(\psi_0)_k*f(y)|^r}{m_{j,Ar,Br,\,x}(x-y)}\,dy\\
&\le C\sum_{k=j}^{\fz} 2^{(j-k)(Ar-n)}
\l\{2^{jn}\int_{|y-x|< 2^{-j_x}}
\frac{|(\psi_0)_k*f(y)|^r}{(1+2^j|x-y|)^{Ar}}\,dy\r.\\
&\l.\qquad\ +\ 2^{jn}\int_{|y-x|\ge 2^{-j_x}}
\frac{|(\psi_0)_k*f(y)|^r}{(2^j|x-y|)^{Ar}2^{Br|x-y|/\rz(x)}}\,dy \r\}\\
&\equiv C\sum_{k=j}^{\fz} 2^{(j-k)(Ar-n)}\{I+II\}.
\end{aligned}
\eeqs
Since $2^{-j_{x}}< \rz(x)\le 2^{-j_x+1}$ and $j\ge j_x$, for $I$ we have
\beqs
\begin{aligned}
I&=2^{jn}\int_{2^{-j}\le |y-x|<2^{-j_x}}
\frac{|(\psi_0)_k*f(y)|^r}{(1+2^j|x-y|)^{Ar}}\,dy
+ 2^{jn}\int_{|y-x|\le 2^{-j}}
\frac{|(\psi_0)_k*f(y)|^r}{(1+2^j|x-y|)^{Ar}}\,dy\\
&\equiv I_1+I_2.
\end{aligned}
\eeqs
According to the definition of $M^{loc}f(x)$ (see \eqref{Mloc}),
for $I_2$ we have
$$
I_2 \le 2^{jn}\int_{|y-x|\le 2^{-j}}
|(\psi_0)_k*f(y)|^r\,dy
\le C M^{loc}\l(|(\psi_0)_k*f|^r\r)(x),
$$
and for $I_1$ we have
\beqs
\begin{aligned}
I_1 &\le 2^{jn} \sum_{l=j_x+1}^j
\int_{2^{-l}\le |y-x|< 2^{-l+1}}
\frac{|(\psi_0)_k*f(y)|^r}{(2^j|x-y|)^{Ar}}\,dy\\
&\le \sum_{l=j_x+1}^j
\frac{2^{jn}(2^{-l+1})^n}{(2^{j-l})^{Ar}}
\frac{1}{(2^{-l+1})^n}
\int_{|y-x|\le 2^{-l+1}} |(\psi_0)_k*f(y)|^r\,dy\\
&\le \sum_{l=j_x+1}^j  \frac{2^n}{2^{(Ar-n)(j-l)}}
M^{loc}\l(|(\psi_0)_k*f|^r\r)(x)\\
&\le C M^{loc}\l(|(\psi_0)_k*f|^r\r)(x),
\end{aligned}
\eeqs
where $Ar>n$. In addition, with regard to $II$, we have the
following estimate,
\beqs
\begin{aligned}
II &\le \frac{2^{jn}(\rz(x))^n}{(2^{j-j_x})^{Ar}}
\frac{1}{(\rz(x))^n}
\int_{\rn} |(\psi_0)_k*f(y)|^r 2^{-Br\frac{|x-y|}{\rz(x)}}\,dy\\
&\le C \frac{2^{jn}(2^{-j_x})^n}{(2^{j-j_x})^{Ar}}
K_{Br}(|(\psi_0)_k*f|)(x)\\
&\le C \frac{(2^{j-j_x})^n}{(2^{j-j_x})^{Ar}}
K_{Br}(|(\psi_0)_k*f|)(x)\\
&\le CK_{Br}(|(\psi_0)_k*f|)(x),
\end{aligned}
\eeqs
where the last inequality is a consequence of the face
that $j\ge j_x$ and $Ar>n$.
This finishes the proof.
\end{proof}

Now we can establish  weighted norm inequalities of $\psi_0^{+}(f)$,
$\psi^{\ast\ast}_{0,\,A,\,B}(f)$ and $\wt{\cm}_{N,\,R} (f)$.

\begin{thm}\label{t3.1}
Let $\wz\in A^{\rz,\fz}_{\fz}(\rn)$, $R\in(0,\fz)$,
$p\in(0,1]$, $\psi_0$ and $q_{\wz}$
be respectively as in \eqref{3.3} and \eqref{qw}, and let $\psi_0^{+}(f)$,
$\psi^{\ast\ast}_{0,\,A,\,B}(f)$ and $\wt{\cm}_{N,\,R} (f)$ be respectively
as in \eqref{3.4}, \eqref{3.5} and \eqref{tgnr}. Let
$A_1\equiv \max\{A_0, nq_{\wz}/p\}$, $B_1\equiv B_0/p$
and $N_0\equiv[2A_1]+1$, where $A_0$ and $B_0$ are respectively as in
Lemmas \ref{l3.2} and \ref{l3.3}. Then for any $A\in (A_1,\fz),\ B\in(B_1,\fz)$
and integer $N\ge N_0$, there exists a positive constant $C$, depending only
on $A,\,B,\,N,\,R,\,\psi_0,\,\wz$ and $n$, such that for all
$f\in\cd'(\rn)$,
\beq \label{3.16}
 \left\|\psi^{\ast\ast}_{0,\,A,\,B}(f)\r\|_{L^{p}_{\wz}(\rn)}\le C
 \left\|\psi_0^{+}(f)\r\|_{L^{p}_{\wz}(\rn)},
\eeq
and
\beq \label{3.17}
 \left\|\wt{\cm}_{N,\,R} (f)\r\|_{L^{p}_{\wz}(\rn)}\le C
 \left\|\psi_0^{+}(f)\r\|_{L^{p}_{\wz}(\rn)},
\eeq
\end{thm}

\begin{proof}
Let $f\in\cd'(\rn)$. First, we prove \eqref{3.16}. Let
$A\in(A_1,\fz)$ and $B\in(B_1,\fz)$.
By $A_1\equiv \max\{A_0,\,nq_{\wz}/p\}$ and $B_1\equiv B_0
/p$, we know that there exists
$r_0\in(0,p/q_{\wz})$ such that $A>n/r_0$
and $Br_0>B_0/q_{\wz}$,  where $A_0$ and $B_0$ are
respectively as in Lemmas \ref{l3.2} and \ref{l3.3}.
Thus, by Lemma 3.4, for all $x\in\rn$ and $j\ge j_x$ we have
\begin{align}\label{3.18}
\left[(\psi_0)^{\ast}_{j,\,A,\,B}(f)(x)\r]^{r_0}
&\ls \sum_{k=j}^{\fz}
2^{(j-k)(Ar_0 -n)}\Big\{M^{ \loc}\left(|(\psi_0)_k\ast f|^{r_0}\r)(x)
\nonumber\\
 &\qquad +K_{Br_0}\left(|(\psi_0)_k\ast f|^{r_0}\r)(x)\Big\}.
\end{align}
Let $\psi^{+}_0 (f)$ and $\psi^{\ast\ast}_{0,\,A,\,B}(f)$ be
respectively as in \eqref{3.4} and \eqref{3.5}. We notice that
for any $x\in\rn$ and $k\ge j_x$,
$$|(\psi_0)_k \ast f (x)|\le\psi^{+}_0(f)(x),$$
which together with \eqref{3.18} implies that for all
$x\in\rn$,
\beq\label{3.19}
\left[\psi^{\ast\ast}_{0,\,A,\,B}(f)(x)\r]^{r_0}\ls
M^{\loc}\left([\psi^{+}_0 (f)]^{r_0})(x)+K_{Br_0}([\psi^{+}_0
(f)]^{r_0}\r)(x).
\eeq
Then by \eqref{3.19} we have
\begin{align}\label{3.20}
\int_{\rn}\left|\psi^{\ast\ast}_{0,\,A,\,B}(f)(x)\r|^p\wz(x)\,dx
& \ls\int_{\rn}\left|\left\{M^{\loc}
\left([\psi^{+}_0 (f)]^{r_0}\r)(x)\r\}\r|^{p/{r_0}}\wz(x)\,dx\nonumber\\
&\qquad+\int_{\rn}\left|\left\{K_{Br_0}\left([\psi^{+}_0
(f)]^{r_0}\r)(x)\r\}\r|^{p/{r_0}}\wz(x)\,dx\nonumber\\
&\equiv \mathrm{I_1}+\mathrm{I_2}.
\end{align}
For $I_1$, as $r_0<p/q_{\wz}$, we have $q\equiv p/r_0>q_{\wz}$
and $\wz\in A_q^{\rz,\fz}(\rn)$, therefore by Lemma \ref{l2.4}(vii)
 we get
\beq \label{3.21}
\int_{\rn}\l|M^{\loc}\left([\psi^{+}_0 (f)]^{r_0}\r)(x)\r|^{p/r_0}\wz(x)\,dx
\ls \int_{\rn} \l|\psi^{+}_0 (f)\r|^p\wz(x)\,dx
\eeq
and for $I_2$ by Lemma \ref{l3.3} we get
\beq \label{3.22}
\int_{\rn}\l|K_{Br_0}\left([\psi^{+}_0(f)]^{r_0}\r)(x)\r|^{p/r_0}\wz(x)\,dx
\ls \int_{\rn} \l|\psi^{+}_0 (f)\r|^p\wz(x)\,dx,
\eeq
which together with \eqref{3.21} implies \eqref{3.16}.

Now we prove \eqref{3.17}. By $N_0 \equiv[2A_1]+1$,
we know
that there exists $A\in(A_1,\fz)$ such that $2A<N_0$. In the rest of
this proof, we fix $A\in(A_1,\fz)$ satisfying $2A<N_0$ and
$B\in(B_1,\fz)$.
Take an integer $N\ge N_0$ and $R\in(0,\fz)$.
For any $\gz\in\cdd_{N,\,R}(\rn)$, $x\in\rn$,
$l\in\zz$ (where $l$ satisfies $2^{-l}\in (0,\,\rz(x))$)
and $j\ge j_x$
(where $2^{-j_x}< \rz(x)\le 2^{-j_x+1}$),
from Lemma \ref{l3.1}, it follows that
\beq\label{3.23}
\gz_l \ast f=\gz_l \ast(\vz_0)_j\ast(\psi_0)_j\ast
f+\sum_{k=j+1}^{\fz}\gz_l \ast\vz_k\ast\psi_k \ast f,
\eeq
where $\vz_0,\,\vz\in\cdd(\rn)$ with $L_{\vz}\ge N$ and $\psi$ is as
in Lemma \ref{l3.1}.

 For any given  $l_0\in\zz$ which satisfies $2^{-l_0}\in (0,\,\rz(x))$,
 and $z\in\rn$ which
 satisfies $|z-x|<2^{-l_0}$,  by (3.23) we have
\begin{align}\label{3.24}
|\gz_{l_0} \ast f(z)|&\le \left|\gz_{l_0} \ast(\vz_0)_{l_0}\ast(\psi_0)_{l_0}
\ast f(z)\r|+\sum^{\fz}_{k=l_0+1}\left|\gz_{l_0} \ast\vz_k\ast\psi_k \ast f(z)\r|\nonumber\\
&\le \int_{\rn}\left|\gz_{l_0} \ast(\vz_0)_{l_0}
(y)\r|\left|(\psi_0)_{l_0}\ast f(z-y)\r|\,dy\nonumber\\
&\quad +\sum_{k=l_0+1}^{\fz}\int_{\rn}\left|\gz_{l_0} \ast\vz_k
(y)\r|\left|\psi_k\ast f(z-y)\r|\,dy \equiv I_3+I_4.
\end{align}
To estimate $I_3$, from
\begin{equation*}
\begin{aligned}
\psi^{\ast\ast}_{0,\,A,\,B}(f)(x)
&=\sup_{j\ge j_x,\,y\in\rn}
\frac{|(\psi_0)_j\ast f(x-y)|}{m_{j,\,A,\,B,\,x}(y)}\\
&= \sup_{j\ge j_x,\,y\in\rn}\frac{|(\psi_0)_j\ast
f(x-(y+x-z))|}{m_{j,\,A,\,B,\,x}(y+x-z)}\\
&= \sup_{j\ge j_x,\,y\in\rn}\frac{|(\psi_0)_j\ast
f(z-y)|}{m_{j,\,A,\,B,\,x}(y+x-z)},
\end{aligned}
\end{equation*}
we infer that
$$\left|(\psi_0)_{l_0}\ast f(z-y)\r|\le\psi^{\ast\ast}_{0,\,A,\,B}(f)(x)
m_{l_0,\,A,\,B,\,x}(y+x-z),$$
which together with the facts that
$$m_{l_0,\,A,\,B,\,x}(y+x-z)\le m_{l_0,\,A,\,B,\,x}(x-z)m_{l_0,\,A,\,B,\,x}(y)$$
and
\begin{equation*}
\begin{aligned}
m_{l_0,\,A,\,B,\,x}(x-z) &= (1+2^{l_0}|x-z|)^A 2^{B\frac{|x-z|}{\rz(x)}}
\ls 2^A,
\end{aligned}
\end{equation*}
implies that
$$|(\psi_0)_{l_0}\ast f(z-y)|\ls2^{A}\psi^{\ast\ast}_{0,\,A,\,B}(f)(x)
m_{l_0,\,A,\,B,\,x}(y).$$
Thus, we have
$$
I_3 \ls2^{A}\left\{\int_{\rn}|\gz_{l_0} \ast(\vz_0)_{l_0}
(y)|m_{l_0,\,A,\,B,\,x}(y)\,dy\r\}\psi^{\ast\ast}_{0,\,A,\,B}(f)(x).
$$

To estimate $\mathrm{I_4}$, by the definition of $\psi$, it is easy
to know that for any $k\in\zz$,
$$\left|\psi_k\ast
f(z-y)\r|\le\left|(\psi_0)_k\ast f(z-y)\r|+\left|(\psi_0)_{k-1}\ast
f(z-y)\r|.$$
By the definition of $\psi^{\ast\ast}_{0,\,A,\,B}(f)$
and the facts that
$$m_{k,\,A,\,B,\,x}(y+x-z)\le
m_{k,\,A,\,B,\,x}(x-z)m_{k,\,A,\,B,\,x}(y),$$
for any $k\in\zz$ and
$m_{k,\,A,\,B,\,x}(x-z)\ls2^{(k-l_0)A}$, we conclude that
\begin{equation*}
\begin{aligned}
|(\psi_0)_k\ast f(z-y)|&\le
\psi^{\ast\ast}_{0,\,A,\,B}(f)(x)m_{k,\,A,\,B,\,x}(y+x-z)\\
&\le \psi^{\ast\ast}_{0,\,A,\,B}(f)(x) m_{k,\,A,\,B,\,x}(x-z)m_{k,\,A,\,B,\,x}(y)\\
&\ls 2^{(k-l_0)A} m_{k,\,A,\,B,\,x}(y)\psi^{\ast\ast}_{0,\,A,\,B}(f)(x).
\end{aligned}
\end{equation*}
Similarly, we also have
$$|(\psi_0)_{k-1}\ast f(z-y)|\ls2^{(k-l_0)A}
m_{k,\,A,\,B,\,x}(y)\psi^{\ast\ast}_{0,\,A,\,B}(f)(x).$$
Thus,
$$
I_4\ls\sum_{k=l_0+1}^{\fz}2^{(k-l_0)A}\left\{\int_{\rn}|\gz_t
\ast\vz_k (y)|m_{k,\,A,\,B,\,x}(y)\,dy\r\}\,\psi^{\ast\ast}_{0,\,A,\,B}(f)(x).
$$

From \eqref{3.24} and the above estimates of $\mathrm{I_3}$ and
$\mathrm{I_4}$, it follows that
\begin{align}\label{3.25}
|\gz_{l_0} \ast f(z)|&\ls\left\{\int_{\rn}|\gz_{l_0} \ast(\vz_0)_{l_0}(y)|m_{l_0,\,A,\,B,\,x}(y)\,dy\r.\nonumber\\
& \quad + \sum_{k=l_0+1}^{\fz}2^{(k-l_0)A}\int_{\rn}|\gz_{l_0} \ast\vz_k
(y)|m_{k,\,A,\,B,\,x}(y)\,dy\bigg\}\psi^{\ast\ast}_{0,\,A,\,B}(f)(x).
\end{align}
Assume that $\supp(\vz_0)\subset B(0,R_0)$. Then
$\supp((\vz_0)_j)\subset B(0,2^{-j}R_0)$ for all $j\ge j_x$.
Moreover, by $\supp(\gz)\subset B(0,R)$,
we see that
$\supp(\gz_{l_0})\subset B(0,2^{-l_0}R).$
From this, we further deduce that
$\supp(\gz_{l_0}\ast(\vz_0)_{l_0})\subset B(0,2^{-l_0}(R_0+R))$ and
$$
|\gz_{l_0}\ast(\vz_0)_{l_0} (y)|\ls\int_{\rn}|\gz_{l_0} (s)||(\vz_0)_{l_0}(y-s)|\,ds
\ls2^{l_0 n}\int_{\rn}|\gz_{l_0} (s)|\,ds\sim2^{l_0 n},
$$
which implies that
\begin{equation}\label{3.26}
\int_{\rn}|\gz_{l_0}\ast(\vz_0)_{l_0} (y)|m_{l_0,\,A,\,B,\,x}(y)\,dy
 \ls2^{l_0 n}\int_{B(0,2^{-l_0}(R_0 +R))}(1+2^{l_0}|y|)^A 2^{\frac{B|y|}{\rz(x)}}\,dy\ls1.
\end{equation}
Moreover, since $\vz$ has vanishing moments up to order $N$, it was
proved in \cite[(2.13)]{Ry} that
$$\|\gz_{l_0}\ast\vz_k\|_{L^{\fz}(\rn)}\ls2^{(l_0-k)N}2^{l_0 n}$$
for all $k\in\zz$ with $k\ge l_0 +1$, which, together with the
facts that $l_0\ge j_x$, $N>2A$ and
$$\supp(\gz_{l_0}\ast\vz_k) \subset B(0,2^{-l_0}R_0+2^{-k}R),$$
implies that
\begin{align}\label{3.27}
\sum_{k=l_0+1}^{\fz}2^{(k-l_0)A}&\int_{\rn}|\gz_{l_0} \ast\vz_k
(y)|m_{k,\,A,\,B,\,x}(y)\,dy\nonumber\\
 & \ls\sum_{k=l_0+1}^{\fz}2^{(k-l_0)A}2^{(l_0-k)N}2^{l_0 n}
 (2^{-l_0}R_0 +2^{-k}R)^n\nonumber\\
&\hspace{1em}
\times\left[1+2^k(2^{-l_0}R_0+2^{-k}R)\r]^A 2^{B(2^{-l_0}R_0+2^{-k}R)/\rz(x)}\nonumber\\
 &\ls\sum_{k=l_0 +1}^{\fz}2^{(l_0-k)(N-2A)}\ls1.
\end{align}
Thus, from \eqref{3.25}, \eqref{3.26} and \eqref{3.27}, we deduce
that $|\gz_{l_0} \ast f(z)|\ls\psi^{\ast\ast}_{0,\,A,\,B}(f)(x)$. Then,
by the arbitrariness of $l_0\ge j_x$ and $z\in B(x,2^{-l_0})$, we know that
\beq
\wt{\cm}_{N,R}(f)(x)\ls\psi^{\ast\ast}_{0,\,A,\,B}(f)(x),
\eeq
which deduces the \eqref{3.17} and completes the proof of this theorem.
\end{proof}

As a corollary of Theorem \ref{t3.1}, we immediately obtain  the
local vertical and the local nontangential maximal function
characterizations of $h^{p}_{\rz,\,N}(\wz)$ with $N\ge N_{p,\,\wz}$ as follows.
Here and in what follows,
\begin{equation}\label{3.29}
N_{p,\,\wz}\equiv\max\left\{\wt{N}_{p,\,\wz},\,N_0\r\},
\end{equation}
where $\wt{N}_{p,\,\wz}$ and $N_0$ are  respectively as in
Definition \ref{d3.2} and Theorem \ref{t3.1}.

\begin{thm}\label{t3.2}
Let  $\wz\in A^{\rz,\,\fz}_{\fz}(\rn)$, $\psi_0$ and $N_{p,\,\wz}$ be respectively
as in \eqref{3.3} and \eqref{3.29}. Then for any integer $N\ge N_{p,\,\wz}$,
the following are equivalent:
\begin{enumerate}
\item[\rm(i)] $f\in h^{p}_{\rz,\,N}(\wz);$
\item[\rm(ii)] $f\in\cd'(\rn)$ and $\psi^{+}_0 (f)\in L^{p}_{\wz}(\rn);$
\item[\rm(iii)] $f\in\cd'(\rn)$ and $(\psi_0)^{\ast}_{\triangledown}
(f)\in L^{p}_{\wz}(\rn);$
\item[\rm(iv)] $f\in\cd'(\rn)$ and $\wt{\cm}_N (f)\in L^{p}_{\wz}(\rn);$
\item[\rm(v)] $f\in\cd'(\rn)$ and $\wt{\cm}^0_N (f)\in L^{p}_{\wz}(\rn);$
\item[\rm(vi)] $f\in\cd'(\rn)$ and $\cm^0_N (f)\in L^{p}_{\wz}(\rn)$.
\end{enumerate}
Moreover, for all $f\in h^p_{\rz,\,N}(\wz)$
\begin{align}\label{3.30}
\|f\|_{h^{p}_{\rho,\,N}(\wz)}&\sim \left\|\psi^{+}_0(f)\r\|_{L^{p}_{\wz}(\rn)}
\sim\left\|(\psi_0)^{\ast}_{\triangledown}(f)\r\|_{L^{p}_{\wz}(\rn)}\nonumber\\
&\sim \left\|\wt{\cm}_N(f)\r\|_{L^{p}_{\wz}(\rn)}
\sim\left\|\wt{\cm}^0_N(f)\r\|_{L^{p}_{\wz}(\rn)}
\sim\left\|\cm^0_N(f)\r\|_{L^{p}_{\wz}(\rn)},
\end{align}
where the implicit constants are independent of $f$.
\end{thm}
\begin{proof}
$\mathrm{(i)\Rightarrow(ii)}$.
 Pick an integer $N\ge N_{p,\,\wz}$ and $f\in h^{p}_{\rho,\,N}(\wz)$.
Let $\wt{\psi}_0$ satisfy \eqref{3.3} and $\wt{\psi}_0\in\cdd_N (\rn)$.
 Then from the definition
of $\cm_N (f)$, we infer that $\wt{\psi}^{+}_0 (f)\le\cm_N (f)$ and
hence $\wt{\psi}^{+}_0 (f)\in L^{p}_{\wz}(\rn)$.
For any $\psi_0$ satisfying \eqref{3.3}, assume that $\supp(\psi_0)\subset B(0,R)$.
Then, by \eqref{3.17} and the above argument, we have
$$\left\|\wt{\cm}_{N,\,R} (f)\r\|_{L^{p}_{\wz}(\rn)}
\ls \left\|\wt{\psi}_0^{+}(f)\r\|_{L^{p}_{\wz}(\rn)}
\ls \|f\|_{h^{p}_{\rz,\,N}(\wz)},$$
which together with $\psi^{+}_0(f)\ls\wt{\cm}_{N,\,R} (f)$
implies that $\psi^{+}_0 (f)\in L^{p}_{\wz}(\rn)$ and
$$\left\|\psi^{+}_0
(f)\r\|_{L^{p}_{\wz}(\rn)}\ls\left\|f\r\|_{h^{p}_{\rho,\,N}(\wz)}.$$

$\mathrm{(ii)\Rightarrow(iii)}$. Let $f\in\cd'(\rn)$ satisfy
$\psi_0^{+}(f) \in L^{p}_{\wz}(\rn)$, where $\psi_0$ is as in \eqref{3.3}.
Then from the fact that
$$\psi_0^{+}(f)
\le(\psi_0)^{\ast}_{\triangledown}(f)\ls\psi^{\ast\ast}_{0,\,A,\,B}(f)$$
and \eqref{3.16}, we deduce that
$(\psi_0)^{\ast}_{\triangledown}(f)\in L^{p}_{\wz}(\rn)$
and
$$\|(\psi_0)^{\ast}_{\triangledown}(f)\|_{L^{p}_{\wz}(\rn)}\ls\|\psi^{+}_0
(f)\|_{L^{p}_{\wz}(\rn)}.$$

$\mathrm{(iii)\Rightarrow(iv)}$. Let $f\in\cd'(\rn)$ satisfy
$(\psi_0)^{\ast}_ {\triangledown}(f)\in L^{p}_{\wz}(\rn)$, where
$\psi_0$ is as in \eqref{3.3}. By \eqref{3.17},
$$\|\wt{\cm}_{N}(f)\|_ {L^{p}_{\wz}(\rn)}\ls
\|\psi_0^{+}(f)\|_{L^{p}_{\wz}(\rn)},$$
which together with the fact that
$$\psi_0^{+}(f) \le(\psi_0)^{\ast}_{\triangledown}(f)$$
and the assumption that $(\psi_0)^{\ast}_ {\triangledown}(f)\in
L^{p}_{\wz}(\rn)$ implies $\wt{\cm}_{N} (f)\in L^{p}_{\wz}(\rn)$ and
$$\left\|\wt{\cm}_N
(f)\r\|_{L^{p}_{\wz}(\rn)}\ls\left\|(\psi_0)^{\ast}_{\triangledown}(f)
\r\|_{L^{p}_{\wz}(\rn)}.$$

$\mathrm{(iv)\Rightarrow(v)\Rightarrow(vi)}$. By the facts that
$\cm^0_N (f)\le\wt{\cm}^0_N (f)\le\wt{\cm}_N (f)$ for any
$f\in\cd'(\rn)$ , we see that all the
conclusions hold. Moreover, it is obvious that
$$\left\|\cm^0_N(f)\r\|_{L^{p}_{\wz}(\rn)}
\le\left\|\wt{\cm}^0_N(f)\r\|_{L^{p}_{\wz}(\rn)}
\le\left\|\wt{\cm}_N(f)\r\|_{L^{p}_{\wz}(\rn)}.$$

$\mathrm{(vi)\Rightarrow(i)}$. Let $f\in\cd'(\rn)$ satisfy
$\cm^0_N(f)\in L^{p}_{\wz}(\rn).$
Let $\psi_1$ satisfy \eqref{3.3} and
$\psi_1\in\cdd^0_N (\rn)$. Then by \eqref{3.17}, we have that
$$\left\|\wt{\cm}_{N} (f)\r\|_ {L^{p}_{\wz}(\rn)}
\ls \|\psi_1^{+}(f)\|_{L^{p}_{\wz}(\rn)},$$
which together with the
facts that $\psi_1^{+}(f)\le\cm^0_{N} (f)$
and $\cm_{N} (f)\le\wt{\cm}_{N} (f)$ implies that
$$\left\|\cm_{N} (f)\r\|_{L^{p}_{\wz}(\rn)}
\ls\left\|\cm^0_{N} (f)\r\|_{L^{p}_{\wz}(\rn)}.$$
Thus, by the definition of
$h^{p}_{\rho,\,N}(\wz)$, we know that $f\in h^{p}_{\rho,\,N}(\wz)$ and
$$\|f\|_ {h^{p}_{\rho,\,N}(\wz)}\ls\|\cm^0_{N}(f)\|_ {L^{p}_{\wz}(\rn)},$$
which completes the proof of Theorem \ref{t3.2}.
\end{proof}

By the Theorems \ref{t3.1} and \ref{t3.2}, we also have the following corollary about
local tangential maximal function characterization of
$h^{p}_{\rho,\,N}(\wz)$, and we omit the details here.

\begin{cor}\label{c3.1}
Let  $\psi_0$ be as in \eqref{3.3},
$\wz\in A^{\rz,\,\fz}_{\fz}(\rn)$, $N_{p,\,\wz}$ be as
in \eqref{3.29}, $A$ and $B$ be as in Theorem \ref{t3.1}. Then for
integer $N\ge N_{p,\,\wz}$,
$f\in h^{p}_{\rho,\,N}(\wz)$
if and only if $f\in\cd'(\rn)$ and
$\psi^{\ast\ast}_{0,\,A,\,B}(f)\in{L^{p}_{\wz}(\rn)}$;
moreover,
$$\|f\|_{h^{p}_{\rho,\,N}(\wz)}\sim
\|\psi^{\ast\ast}_{0,\,A,\,B}(f)\|_{L^{p}_{\wz}(\rn)}.$$
\end{cor}

Next we give some basic properties of $h^p_{\rz,\,N}(\wz)$
and $h^{p,\,q,\,s}_{\rz}(\wz)$.

\begin{prop}\label{p3.2}
Let  $\wz\in A^{\rz,\,\fz}_{\fz}(\rn)$, $p\in(0,1]$ and $N_{p,\,\wz}$
be as in \eqref{3.29}.
For any integer $N\ge N_{p,\,\wz}$,  the inclusion
$h^{p}_{\rho,\,N}(\wz)\hookrightarrow\cd'(\rn)$ is continuous.
\end{prop}

\begin{proof}
First, for any $x\in B(0,\rz(0))$, by Lemma \ref{l2.1}, there exist
$C_0\ge 1$ and $k_0 \geq 1$, such that
$$\rz(0)\le C_0  \l(1+\frac{|x|}{\rz(0)}\r)^{k_0} \rz(x)
\le C_0 2^{k_0} \rz(x).$$
We take $r_1\equiv\rz(0)/C_02^{k_0+1}< \min\{\rz(x),\,\rz(0)\}$, then we have
$B(0,r_1)\subset B(0,\rz(0))$. In addition,
 for any $x\in B(0,r_1)$, we also have $|x|<r_1< \rz(x)$.

Next, let $f\in h^{p}_{\rz,\,N}(\wz)$. For any given $\vz\in\cdd(\rn)$,
assume that $\supp(\vz)\subset B(0,R)$ with $R\in(0,\fz)$.
Then by Theorem \ref{t3.1} and \ref{t3.2},  we have
\begin{eqnarray*}
\begin{aligned}
|\langle f,\vz\rangle|=\left|f\ast\wt{\vz}(0)\r|
&\le\left\|\wt{\vz}\r\|_{\cdd_{N,\,R}(\rn)}\inf_{x\in B(0,r_1)}\wt{\cm}_{N,\,R} (f)(x)\\
&\le \left\|\wt{\vz}\r\|_{\cdd_{N,\,R}(\rn)} [\wz(B(0,r_1))]^{-1/p}
\l\|\wt{\cm}_{N,\,R} (f)\r\|_{L^p_{\wz}(\rn)} \\
&\ls \left\|\wt{\vz}\r\|_{\cdd_{N,\,R}(\rn)} [\wz(B(0,r_1))]^{-1/p}
\l\|f\r\|_{h^p_{\rho,\,N}(\wz)},
\end{aligned}
\end{eqnarray*}
where $\wt{\cm}_{N,\,R} (f)$ is as in \eqref{tgnr} and
$\wt{\vz}(x)\equiv\vz(-x)$ for all $x\in\rn$.
This implies $f\in {\cal D}'(\rn)$
and the inclusion is continuous. The proof is finished.
\end{proof}

\begin{prop}\label{p3.3}
Let  $\wz\in A^{\rz,\,\fz}_{\fz}(\rn)$, $p\in(0,1]$ and $N_{p,\,\wz}$
be as in \eqref{3.29}.
For any integer $N\ge N_{p,\,\wz}$, the space
$h^{p}_{\rho,\,N}(\wz)$ is complete.
\end{prop}

\begin{proof}
For any $\psi\in\cdd_N (\rn)$ and $\{f_i\}_{i\in\nn}\subset\cd'(\rn)$
such that $\{\sum_{i=1}^j f_i\}_{j\in\nn}$ converges in $\cd'(\rn)$
to a distribution $f$ as $j\to\fz$, and the series $\sum_i
 f_i*\psi(x)$ converges pointwise to $f*\psi(x)$ for each $x\in\rn$.
Therefore,
$$\l(\cm_N (f)(x)\r)^p\le\left(\sum_{i=1}^{\fz}\cm_N
(f_i)(x)\r)^p\le\sum_{i=1}^{\fz}\left(\cm_N (f_i)(x)\r)^p \quad
{\rm for \ all}\ x\in\rn,$$
and hence
$\|f\|_{h^p_{\rho,N}(\wz)}\le \dsum_i^{\fz} \|f_i\|_{h^p_{\rho,N}(\wz)}$.

To prove that $h^p_{\rho,N}(\wz)$ is complete, it  suffices to show
that for every sequence $\{f_j\}_{j\in\nn}$ with
$\|f_j\|_{h^p_{\rho,N}(\wz)}<2^{-j}$ for any $j\in\nn$, the series
$\sum_{j\in\nn}f_j$ convergence in $h^p_{\rho,N}(\wz)$.
Since $\{\sum_{i=1}^j f_i\}_{j\in\nn}$ is a Cauchy sequence in
$h^{p}_{\rho,\,N}(\wz)$, by Proposition \ref{p3.2} and the
completeness of $\cd'(\rn)$, $\{\sum_{i=1}^j f_i\}_{j\in\nn}$ is
also a Cauchy sequence in $\cd'(\rn)$ and thus converges to some
$f\in\cd'(\rn)$. Therefore,
$$\bigg\|f-\dsum_{i=1}^j f_i\bigg\|^p_{h^p_{\rho,N}(\wz)}
=\bigg\|\dsum_{i=j+1}^\fz f_i\bigg\|^p_{h^p_{\rho,N}(\wz)}
\le \dsum_{i=j+1}^\fz 2^{-ip}\to 0$$ as $j\to\fz$. This finishes the
proof.
\end{proof}

\begin{thm}\label{t3.3}
Let $\wz\in A^{\rz,\,\fz}_{\fz}(\rn)$ and $N_{p,\,\wz}$ be as in \eqref{3.29}.
If
$(p,\,q,\,s)_{\wz}$ is an admissible triplet (see Definition \ref{d3.3}) and
integer $N\ge N_{p,\,\wz}$, then
$$h^{p,\,q,\,s}_{\rho}(\wz)\subset
h^{p}_{\rho,\,N_{p,\,\wz}}(\wz)\subset
h^{p}_{\rho,\,N}(\wz),$$
and moreover, there exists a positive
constant $C$ such that for all $f\in h^{p,\,q,\,s}_{\rho}(\wz)$,
$$\|f\|_{h^{p}_{\rho,\,N}(\wz)}
\le\|f\|_{h^{p}_{\rho,\,N_{p,\,\wz}}(\wz)}
\le C\|f\|_{h^{p,\,q,\,s}_{\rho}(\wz)}.$$
\end{thm}

\begin{proof}
Obviously, we only need to prove $h^{p,\,q,\,s}_{\rho}(\wz)\subset
h^p_{\rho,\,N_{p,\wz}}(\wz)$. For  all $f\in h^{p,\,q,\,s}_{\rho}(\wz)$,
$$\|f\|_{h^p_{\rho,\,N_{p,\wz}}(\wz)}\ls \|f\|_{h^{p,\,q,\,s}_{\rho}(\wz)}.$$
By Definition \ref{d3.4} and Theorem \ref{t3.2}, it suffices to prove that
there exists a positive constant $C$ such that
\beq \label{3.31}
\l\|\cm^0_{N_{p,\,\wz}}(a)\r\|_{L^p_{\wz}(\rn)}\le C,
\quad{\rm for\ all}\ (p,q)_\wz-{single-atoms}\ a,
\eeq
and
\beq \label{3.32}
\l\|\cm^0_{N_{p,\,\wz}}(a)\r\|_{L^p_{\wz}(\rn)}\le C,
\quad{\rm for\ all}\ (p,q,s)_\wz-{atoms}\ a.
\eeq
We  first prove (3.31).
Since $q\in (q_\wz,\fz]$, so $\wz\in A_q^{\rz,\,\fz}(\rn)$.
Let $a$ be a $(p,\,q)_{\wz}$-single-atom. When
$\wz(\rn)=\fz$, by the definition of the single atom, we know that
$a=0$ for almost every $x\in\rn$. In this case, it is easy to see
that \eqref{3.31} holds. When $\wz(\rn)<\fz$,
from H\"older's inequality, $\wz\in A_q^{\rz,\,\fz}(\rn)$ and
Proposition \ref{p3.1}(i), we deduce that
\begin{equation*}
\begin{aligned}
\l\|\cm^0_{N_{p,\,\wz}}(a)\r\|_{L^p_{\wz}(\rn)}^p
&=\int_{\rn}\left|\cm^0_{N_{p,\,\wz}}(a)(x)\r|^p\wz(x)\,dx\\
&\le \l(\int_{\rn}\left|\cm^0_{N_{p,\,\wz}}(a)(x)\r|^q\wz(x)\,dx\r)^{p/q}
\l(\int_{\rn}\wz(x)\,dx\r)^{1-p/q}\\
&\le C\|a\|_{L^q_\wz(\rn)}^p[\wz(\rn)]^{1-p/q}\le C.
\end{aligned}
\end{equation*}

Next, we prove (3.32). Let $a$ be a
$(p,\,q,\,s)_{\wz}$-atom supported in the  cube $Q\equiv Q(x_0,r)$.
We consider the following two cases for $Q$.

The first case is when $r< L_{2}\rz(x_0)$.
Let $\wt{Q}\equiv2\sqrt{n}Q$, then we have
\begin{align}\label{3.33}
\int_{\rn}\left|\cm^0_{N_{p,\,\wz}}(a)(x)\r|^p\wz(x)\,dx
&=\int_{\wt{Q}}\left|\cm^0_{N_{p,\,\wz}}(a)(x)\r|^p\wz(x)\,dx
+\int_{{\wt{Q}}^{\complement} }\left|\cm^0_{N_{p,\,\wz}}(a)(x)\r|^p\wz(x)\,dx
\nonumber\\
&\equiv\mathrm{I_1}+\mathrm{I_2}.
\end{align}
For $I_1$, by H\"older's inequality and the properties of $A_q^{\rz,\,\tz}(\rn)$
(see Lemma \ref{l2.4}(vi)), we have
\begin{align}\label{3.34}
I_1
&\le \l(\int_{\rn}\left|\cm^0_{N_{p,\,\wz}}(a)(x)\r|^q\wz(x)\,dx\r)^{p/q}
\l(\int_{\wt{Q}}\wz(x)\,dx\r)^{1-p/q}\nonumber\\
&\le C\|a\|_{L^q_\wz(\rn)}^p[\wz(\wt{Q})]^{1-p/q}\le C.
\end{align}
To estimate $I_2$, we claim that for  $x\in \wt{Q}^\complement$
\begin{equation}\label{3.35}
\cm^0_{N_{p,\,\wz}}(a)(x)\le C |Q|^{(s_0+n+1)/n}[\wz(Q)]^{-1/p}
|x-x_0|^{-(s_0+n+1)} \chi_{B(x_0,c_1\rz(x_0))}(x),
\end{equation}
where $s_0\equiv[n(q_{\wz}/p-1)]$ and $c_1>2\sqrt{n}$ is an constant independent of the atom $a$.
Indeed, for any $\psi\in\cdd^0_{N}(\rn)$ and $2^{-l}\in(0,\rz(x))$, let $P$ be the Taylor
expansion of $\psi$ about $(x-x_0)/2^{-l}$ with degree $s_0$.
By Taylor's remainder theorem, for any $y\in\rn$, we have
\begin{equation*}
\begin{aligned}
&\left|\psi\left(\frac{x-y}{2^{-l}}\r)-P\left(\frac{x-x_0}{2^{-l}}\r)\r| \\
&\le C\sum_{\genfrac{}{}{0pt}{}{\az\in\zz^n_+}{|\az|=s_0+1}}
\left|\left(\partial^{\az}\psi\r)
\left(\frac{\theta(x-y)+(1-\theta)(x-x_0)}{2^{-l}}\r)\r|\left|
\frac{x_0-y}{2^{-l}}\r|^{s_0 +1},
\end{aligned}
\end{equation*}
where $\theta\in(0,1)$.
By $2^{-l}\in(0,\rz(x))$ and $x\in\wt{Q}^{\complement}$,
we see that $\supp (a\ast\psi_l)\subset B(x_0,c_1\rz(x_0))$,
and  $a\ast\psi_l (x)\neq0$ implies that $2^{-l}>|x-x_0|/2$.
Thus, from the above facts and Definition
\ref{d3.3}, we get that for all
$x\in\wt{Q}^{\complement}$,
\begin{equation*}
\begin{aligned}
|a\ast\psi_l (x)|&\le \frac{1}{2^{-ln}}\left\{\int_{Q}|a(y)|
\left|\psi\left(\frac{x-y}{2^{-l}}\r)-P\left(\frac{x-x_0}{2^{-l}}\r)\r|\,dy\r\}\,
\chi_{B(x_0,c_1\rz(x_0))}(x)\\
&\le C |x-x_0|^{-(s_0 +n+1)}\left\{\int_{Q}|a(y)||x_0 -y|^{s_0+1}\,dy\r\} \chi_{B(x_0,c_1\rz(x_0))}(x)\\
&\le C |Q|^{(s_0 +1)/{n}}
\|a\|_{L^q_{\wz}(\rn)}\left(\int_{Q}[\wz(y)]^{-q'/q}\,dy\r)^{1/q'}\!
|x-x_0|^{-(s_0 +n+1)}\chi_{B(x_0,c_1\rz(x_0))}(x)\\
&\le C |Q|^{(s_0 +n+1)/{n}}[\wz(Q)]^{-1/p}|x-x_0|^{-(s_0 +n+1)}
\chi_{B(x_0,c_1\rz(x_0))}(x),
\end{aligned}
\end{equation*}
which together with the arbitrariness of $\psi\in\cdd^0_{N}(\rn)$
implies \eqref{3.35}. Thus, the claim holds.

Let $Q_i\equiv2^i\sqrt{n}Q$ for all $i\in\nn$ and $i_0 \in\nn$
satisfying $2^{i_0}r\le c_1\rz(x_0)<2^{i_0 +1}r$. As
$s_0=\left[ n\left(q_{\wz}/p-1\r)\r]$,
we know that there exists $q_0\in(q_{\wz},\fz)$ such that
$p(s_0 +n+1)>nq_0$.
Then from the Lemma \ref{l2.4}, we conclude that
\begin{equation*}
\begin{aligned}
I_2&\le \int_{\sqrt{n}r\le|x-x_0|<c_1\rz(x_0)}
\left|\cm^0_{N_{p,\,\wz}}(a)(x)\r|^p\wz(x)\,dx\\
&\le C |Q|^{p(s_0+n+1)/n}[\wz(Q)]^{-1}
\int_{\sqrt{n}r\le|x-x_0|<c_1\rz(x_0)}
|x-x_0|^{-p(s_0+n+1)}\wz(x)\,dx\\
&\le C r^{p(s_0+n+1)}[\wz(Q)]^{-1}
\sum_{i=0}^{i_0}\int_{Q_{i+1}\setminus Q_{i}}
|x-x_0|^{-p(s_0+n+1)}\wz(x)\,dx\\
&\le C [\wz(Q)]^{-1}
\sum_{i=0}^{i_0} 2^{-ip(s_0+n+1)} \wz(Q_{i+1})\\
&\le C [\wz(Q)]^{-1}
\sum_{i=0}^{i_0} 2^{-i[p(s_0+n+1)-nq_0]} \wz(Q)
\le C,
\end{aligned}
\end{equation*}
which together with \eqref{3.33} and \eqref{3.34} implies
\eqref{3.32} in the first   case.

Now we consider the case $L_{2}\rz(x_0)\le r\le L_{1}\rz(x_0)$,
let $Q^{\ast}\equiv Q(x_0,c_2r)$, in which $c_2>1$ is an constant independent of  atom $a$.
Thus, from $\supp(\cm^0_{N_{p,\,\wz}}(a))\subset Q^{\ast}$,  H\"older's inequality
and Lemma \ref{l2.4}, we get
\begin{equation*}
\begin{aligned}
\int_{\rn}\left|\cm^0_{N_{p,\,\wz}}(a)(x)\r|^p\wz(x)\,dx
&=\int_{Q^*} \left|\cm^0_{N_{p,\,\wz}}(a)(x)\r|^p\wz(x)\,dx\\
&\le C\|a\|_{L^q_\wz(\rn)}^p[\wz(Q^*)]^{1-p/q}\\
&\le C\|a\|_{L^q_\wz(\rn)}^p[\wz(Q)]^{1-p/q}\\
&\le C.
\end{aligned}
\end{equation*}
This finishes the proof of Theorem 3.3.
\end{proof}

\section{Calder\'on-Zygmund decompositions}\label{s4}

\noi In this section, we establish the Calder\'on-Zygmund
decompositions associated with local grand maximal functions on weighted
Euclidean space $\rn$. We  follow the constructions in \cite {St}, \cite {Bo} and
\cite{BLYZ}.

Let $\wz\in A_{\fz}^{\rz,\,\fz}(\rn)$ and $q_{\wz}$ be as in \eqref{qw}.
For integer $N\ge 2$, let $\cm_N (f)$ and $\cm_N^0 (f)$ be as in \eqref{gnr}.
Throughout  this section, we consider a distribution $f$ so that for all $\lz>0$,
$$\wz\left(\{x\in\rn:\,\cm_N (f)(x)>\lz\}\r)<\fz.$$
For a given $\lz>\inf_{x\in\rn}\cm_{N}(f)(x)$, we set
$$\oz_{\lz}\equiv\{x\in\rn: {\cal M}_N(f)(x)>\lz\}.$$
It is obvious that $\oz_{\lz}$ is a proper open subset of $\rn$.
As in \cite{St}, we give the usual Whitney decomposition of $\oz_{\lz}$.
Thus we can find closed cubes $Q_k$ whose interiors distance from
$\oz_{\lz}^\complement$, with $\oz_{\lz}=\bigcup_k Q_k$ and
$$diam(Q_k)\le 2^{-(6+n)} dist(Q_k,\oz_{\lz}^\complement)\le 4diam(Q_k).$$
In what follows, fix $a\equiv 1+2^{-(11+n)}$  and $b\equiv1+2^{-(10+n)}$,
and if we denote $\bar Q_k=aQ_k, Q^*_k=bQ_k$, we have
$Q_k\subset\bar Q_k\subset Q_k^*$.
Moveover, $\oz_{\lz}=\bigcup_{k}Q^{\ast}_k$, and $\{Q^{\ast}_k\}_k$
have the bounded interior property, namely, every point in
$\oz_{\lz}$ is contained in at most a fixed number of
$\{Q^{\ast}_k\}_k$.

Now we take a function $\xi\in \cdd(\rn)$ such that $0\le\xi\le1$,
$\supp(\xi)\subset aQ(0,1)$ and $\xi\equiv1$ on $Q(0,1)$.
For $x\in\rn$, set $\xi_k (x)\equiv\xi((x-x_k)/l_k)$, where and in what
follows, $x_k$ is the center of the cube $Q_k$ and $l_k$ is its sidelength.
Obviously, by the construction of
$\{Q_k^{\ast}\}_k$ and $\{\xi_k\}_k$, for any $x\in\rn$, we have
$1\le\sum_{k}\xi_k (x)\le M$, where $M$ is a fixed positive integer
independent of $x$.
Let $\eta_k\equiv \xi_k/(\sum_j\xi_j).$
Then $\{\eta_k\}_k$ form a smooth partition of unity for
$\oz_{\lz}$ subordinate to the locally
finite covering $\{Q_k^{\ast}\}_k$ of $\oz_{\lz}$, namely,
$\chi_{\oz_{\lz}}=\sum_k \eta_k$ with each $\eta_k\in \cdd(\rn)$
supported in $\bar Q_k$.

Let $s\in\zz_{+}$ be some fixed integer and $\cp_s (\rn)$ denote the
linear space of polynomials in $n$ variables of degrees no more than
$s$. For each $i\in\nn$ and $P\in\cp_s (\rn)$, set
\begin{equation}\label{4.1}
\|P\|_i\equiv\left[\frac{1}{\int_{\rn}\eta_i
(y)\,dy}\int_{\rn}|P(x)|^2\eta_i (x)\,dx\r]^{1/2}.
\end{equation}
Then it is easy to see that $(\cp_s (\rn),\,\|\cdot\|_i)$ is a
finite dimensional Hilbert space.
Let $f\in\cd'(\rn)$. Since $f$ induces a linear functional on $\cp_s (\rn)$ via
$$P\mapsto\frac{1}{\int_{\rn}\eta_i (y)\,dy}\langle f, P\eta_i\rangle,$$
by the Riesz represent theorem, there exists a
unique polynomial $P_i \in\cp_s (\rn)$ for each $i$ such that
for all $Q\in\cp_s (\rn)$,
$$\langle f, Q\eta_i\rangle=\langle P_i, Q\eta_i\rangle
 =\int_{\rn} P_i(x) Q(x)\eta_i(x)\,dx.$$
For each $i$, define the distribution
$b_i\equiv(f-P_i)\eta_i$ when $l_i\in(0,L_3\rz(x_i))$
(where $L_3=2^{k_0}C_0$, $x_i$ is the center of the cube $Q_i$)
and
$b_i\equiv f\eta_i $ when $l_i\in[L_3\rz(x_i),\fz)$.

We will show that for suitable choices of $s$ and $N$, the series
$\sum_i b_i$ converge in $\cd'(\rn)$, and in this case, we define
$g\equiv f-\sum_i b_i$ in $\cd'(\rn)$. We point out that the represent
$f=g+\sum_i b_i$, where $g$ and $b_i$ are as above, is called a
Calder\'on-Zygmund decomposition of $f$ of degree $s$ and height
$\lz$ associated with $\cm_N (f)$.

The rest of this section consists of a series of lemmas.
In Lemma 4.1 and Lemma 4.2, we give some properties of the smooth partition
of unity $\{\eta_i\}_i$. In Lemmas 4.3 through 4.6, we derive some
estimates for the bad parts $\{b_i\}_i$. Lemma 4.7 and Lemma 4.8
give controls over the good part $g$. Finally, Corollary 4.1 shows
the density of $L_\wz^q(\rn)\bigcap h^p_{\rho,N}(\wz)$ in
$h^p_{\rho,N}(\wz)$, where $q\in (q_\wz,\fz)$.

\begin{lem}\label{l4.1}
There exists a positive constant $C_1$ depending only on $N$,
 such that for all $i$ and $l\le l_i$,
$$\dsup_{|\az|\le N}\sup_{x\in\rn}|\pz^\az\eta_i (lx)|\le C_1.$$
\end{lem}

Lemma 4.1 is essentially Lemma 5.2 in \cite{Bo}.

\begin{lem}\label{l4.2}
If  $l_i<L_3\rz(x_i)$, then there exists a constant $C_2>0$
independent of $f\in {\cal D}'(\rn)$, $l_i$ and $\lz>0$ so that
$$\dsup_{y\in\rn}|P_i(y)\eta_i(y)|\le C_2\lz.$$
\end{lem}
\begin{proof}
As in the proof of Lemma 5.3 in \cite{Bo}.
Let $\pi_1,\cdots, \pi_m(m=dim{\cal P}_s)$ be an orthonormal basis of
${\cal P}_s$ with respect to the norm (4.1).
we have
\beq\label{4.2}
P_i=\dsum_{k=1}^m\l(\dfrac {1}{\int\eta_i}
\dint f(x)\pi_k(x)\eta_i(x)\,dx\r)\bar\pi_k,
\eeq
where the
integral is understood as $\langle f,\pi_k\eta_i \rangle $.
Therefore,
\begin{align}
1&=\frac{1}{\int\eta_i}
\int_{\bar Q_i} |\pi_k(x)|^2\eta_i(x)\,dx
\ge \dfrac{2^{-n}}{|Q_i|} \dint_{\bar Q_i}|\pi_k(x)|^2\eta_i(x)\,dx \nonumber \\
&\ge \dfrac {2^{-n}}{| Q_i|}
\dint_{Q_i}|\pi_k(x)|^2\,dx=2^{-n}\dint_{Q^0}|\wt\pi_k(x)|^2\,dx,
\end{align}
where $\wt\pi_k(x)=\pi_k(x_i+l_i x)$ and $Q^0$ denotes the cube of
side length $1$ centered at the origin.

Since ${\cal P}_s$ is finite dimensional, all norms on ${\cal P}_s$
are equivalent, then there exists $A_1>0$ such that for all $P \in {\cal P}_s$
$$\dsup_{|\az|\le s}\dsup_{z\in bQ^0}|\pz^\az P(z)|\le
A_1\l(\dint_{Q^0}|P(z)|^2\,dz\r)^{1/2}.$$
From this and (4.3), for $k=1,\cdots, m$, we have
\beq \label{4.4}
\dsup_{|\az|\le s}\dsup_{z\in bQ^0}|\pz^\az \wt\pi_k(z)|\le A_1 2^{n/2}.
\eeq
If $z\in 2^{8+n}nQ_i\cap\oz^{\complement}$, by Lemma \ref{l2.1}, we have
$\rz(x_i)\le C_0(1+2^{8+n}n^2L_3)^{k_0}\rz(z)$, then we let
$\widetilde{L}\equiv1/C_0(1+2^{8+n}n^2L_3)^{k_0}L_3$.
For $k=1,\cdots, m$, we define
$$\Phi_k(y)=\dfrac{2^{-k_in}}{\int\eta_i}\pi_k(z-2^{-k_i}y)\eta_i(z-2^{-k_i}y),$$
where $z\in 2^{8+n}nQ_i\cap\oz^{\complement}$ and $2^{-k_i}\le \widetilde{L} l_i< 2^{-k_i+1} $.
It is easy to see that $\supp \Phi_k\subset B(0, R_1)$
where $R_1\equiv2^{9+n}n^2/\widetilde{L}$,
and $\|\Phi_k\|_{{\cal D}_N}\le A_2$ by Lemma \ref{4.1}.

Note that
$$\dfrac 1{\int\eta_i}\dint f(x)\pi_k(x)\eta_i(x)\,dx=(f*(\Phi_k)_{k_i})(z),$$
since $2^{-k_i}\le \widetilde{L} l_i<\widetilde{L} L_3\rz(x_i)\le \rz(z)$,  then we have
$$\l|\dfrac 1{\int\eta_i}\dint
f(x)\pi_k(x)\eta_i(x)\,dx\r|\le {\cal M}_Nf(z)\|\Phi_k\|_{{\cal D}_N}\le A_2\lz. $$
By (4.2), (4.4) and above estimate, we have
$$\dsup_{z\in Q^*_i}| P_i(z)|\le m 2^{n/2} A_1A_2\lz.$$
Thus,
$$\dsup_{z\in\rn}| P_i(z)\eta_i(z)|\le C_2\lz.$$
The proof is complete.
\end{proof}

By the same method, we can get the following lemma as the Lemma 4.3 in \cite{Ta1}, and
we omit the details here.

\begin{lem}\label{l4.3}
There exists a constant $C_3>0$ such that
\beq \label{4.5}
{\cal M}_N^0 b_i(x)\le C_3 {\cal M}_Nf(x)\quad {\rm for}\ x\in Q_i^*.
\eeq
\end{lem}

\begin{lem}\label{l4.4}
Suppose $Q\subset \rn$ is bounded, convex, and $0\in Q$, and $N$
is a positive integer. Then there is a constant $C$ depending only
on $Q$ and $N$ such that for every $\phi\in{\cal D}(\rn)$ and
every integer $s,\ 0\le s<N$ we have
$$\dsup_{z\in Q}\dsup_{|\az|\le N}|\pz^\az R_y(z)|\le C
\dsup_{z\in y+Q}\dsup_{s+1\le |\az|\le N}|\pz^\az \phi(z)|,$$
where
$R_y$ is the remainder of the Taylor expansion of $\phi$ of order
$s$ at the point $y\in\rn$.
\end{lem}
\noi Lemma 4.4 is Lemma 5.5 in \cite{Bo}.

\begin{lem}\label{l4.5}
Suppose $0\le s<N$. Then there exist positive constants $C_4, C_5$
so that for $i\in\nn$,
\beq \label{4.6}
{\cal M}_N^0(b_i)(x)\le
C\dfrac{\lambda l_i^{n+s+1}}{(l_i+|x-x_i|)^{n+s+1}}\chi_{\{|x-x_i|<C_4 \rz(x)\}}(x)\quad
{\rm  if}\ x\not\in Q_i^*.
\eeq
Moreover,
$${\cal M}_N^0(b_i)(x)=0,\ {\rm if}\  x\not\in Q_i^*\ {\rm and}\
l_i\ge C_5\rz(x).$$
\end{lem}
\begin{proof}
Take $\vz\in {\cal D}^0_N(\rn)$. Recall that $\eta_i$ is
supported in the cube $\bar Q_i$, and we have taken $\bar Q_i$ to
be strictly contained in $Q_i^*$.
Thus if $x\not\in Q_i^*$ and
$\eta_i(y)\not=0$, then there exists a positive constant $C_4$
such that  $|x-y|\le|x-x_i|\le C_4|x-y|$. On the other hand, the support property
of $\vz$ requires that $\rz(x)>2^{-l}\ge |x-y|\ge 2^{-11-n}l_i$. Hence,
$|x-x_i|\le C_4 2^{-l}$, $l_i< 2^{11+n}\rz(x)\equiv C_5\rz(x)$ and $l_i<C_5 2^{-l}$.
Pick some $w\in (2^{8+n}n Q_i)\cap\oz^{\complement}$, and we discuss the following two cases.

\noi Case I. If $L_3 \rz(x_i)\le l_i< C_5 2^{-l}<C_5 \rz(x)$, then  	
according to the Lemma  \ref{l2.1} we have
$l_i<C_5C_0(1+C_4)^{k_0}\rz(x_i)$ and
$$\rz(\wz)\ge C_0^{-1}\l(1+\frac{|\wz-x_i|}{\rz(x_i)}\r)^{-k_0}\rz(x_i)
\ge C_0^{-1}\l(1+2^{8+n}n\sqrt{n}L_2\r)^{-k_0}\rz(x_i),$$
 therefore,
$l_i<a_1\rz(w)$, where  $a_1>1$ is a constant.

Now we define $\bar l_i=l_i/a_1< \rz(w)$ and take $k_i\in\zz$ such that
$2^{-k_i}\le \bar l_i<2^{-k_i+1}$, then for $\vz\in {\cal D}^0_N(\rn)$,
$\phi(z)\equiv\vz(2^{-k_i}z/2^{-l})$
and $2^{-l}<\rz(x)$ we have
\beqs
\begin{aligned}
(b_i*\vz_l)(x)&=2^{ln}\dint b_i(z)\vz(2^l(x-z))dz\\
&=2^{ln}\dint b_i(z)\phi(2^{k_i}(x-z))dz\\
&=2^{ln}\dint b_i(z)\phi_{2^{k_i}(x-w)}(2^{k_i}(w-z))dz\\
&=\dfrac{2^{ln}}{2^{k_in}}(f*\Phi_{k_i})(w),
\end{aligned}
\eeqs
where
$$\Phi(z)\equiv\phi_{2^{k_i}(x-w)}(z)\eta_i(w-2^{-k_i}z), \ \phi_{2^{k_i}(x-w)}(z)\equiv\phi(z+2^{k_i}(x-w)).$$
Obviously, $\supp\Phi\subset B(0,R_2)$, where $R_2\equiv2^{9+n}n^2a_1$.
Notice that $l_i<C_52^{-l}$ and
$|x-x_i|\le C_42^{-l}$, we obtain
\beq \label{4.7}
|(b_i*\vz_l)(x)|\le C\dfrac{2^{ln}}{2^{k_in}}{\cal M}_Nf(w)
\le C\lz \dfrac{2^{ln}}{2^{k_in}}
\le C\lz\dfrac{l_i^{n+s+1}}{(l_i+|x-x_i|)^{n+s+1}}.
\eeq
\noi Case II. If $ l_i<L_3\rz(x_i)$ and $\vz\in {\cal D}^0_N(\rn)$,
taking $j_i\in\zz$ such that
$2^{-j_i}\le  l_i<2^{-j_i+1}$,
then we define
 $\phi(z)=\vz(2^{-j_i}z/2^{-l})$ and consider the Taylor expansion of $\phi$ of order $s$
at the point $y=2^{j_i}(x-w)$:
$$\phi(y+z)=\dsum_{|\az|\le
s}\dfrac{\pz^\az\phi(y)}{\az!}z^\az+R_y(z),$$ where $R_y$ denotes
the remainder.
Thus we get
\begin{align} \label{4.8}
(b_i*\vz_l)(x)&=2^{ln}\dint b_i(z)\vz(2^{ln}(x-z))dz \nonumber\\
 &=2^{ln}\dint b_i(z)\phi(2^{j_in}(x-z))dz  \nonumber\\
&=2^{ln}\dint b_i(z)R_{2^{j_i}(x-w)}(2^{j_i}(w-z))dz  \nonumber\\
&=\dfrac{2^{ln}}{2^{j_in}}(f*\Phi_{j_i})(w)
-2^{ln}\dint P_i(z)\eta_i(z)R_{2^{j_i}(x-w)}(2^{j_i}(w-z))dz,
\end{align}
where
$$\Phi(z)\equiv R_{2^{j_i}(x-w)}(z)\eta_i(\wz-2^{-j_i}z).$$
Obviously, $\supp\Phi\subset B_n\equiv B(0,R_2)$. Applying Lemma \ref{l4.4} to
$\phi(z)=\vz(2^{-j_i}z/2^{-l}),\ y=2^{j_i}(x-w)$ and $B_n$, we have
\beqs
\begin{aligned}
\dsup_{z\in B_n}\dsup_{|\az|\le N}|\pz^\az R_y(z)|
&\le C\dsup_{z\in y+B_n}\dsup_{s+1\le |\az|\le N}|\pz^\az \phi(z)|\\
&\le C\dsup_{z\in y+B_n}\l(\dfrac {2^{-j_i}}{2^{-l}}\r)^{s+1}
\dsup_{s+1\le |\az|\le N}|\pz^\az \vz(2^{-j_i}z/2^{-l})|\\
&\le C\l(\dfrac {2^{-j_i}}{2^{-l}}\r)^{s+1}.
\end{aligned}
\eeqs
Notice that $l_i<C_5 2^{-l}$ and
$|x-x_i|\le C_4 2^{-l}$,  therefore by (4.8), we obtain
\begin{align}
(b_i*\vz_l)(x)&\le\dfrac{2^{ln}}{2^{j_in}}|(f*\Phi_{j_i})(w)|+2^{ln}
\dint |P_i(z)\eta_i(z)R_{2^{j_i}(x-w)}(2^{j_i}(w-z))|dz \nonumber\\
&\le C\dfrac{2^{ln}}{2^{j_in}}\l({\cal M}_Nf(w)\|\Phi\|_{{\cal D}_N}+
\lz\dsup_{z\in B_n}\dsup_{|\az|\le N}|\pz^\az R_y(z)|\r) \nonumber\\
&\le C\lz\dfrac{l_i^{n+s+1}}{(l_i+|x-x_i|)^{n+s+1}}.
\end{align}
By combining both cases, we obtain (4.6).
\end{proof}

\begin{lem}\label{l4.6}
Let $\wz\in A_\fz^{\rz,\,\fz}(\rn)$ and $q_\wz$ be as in \eqref{qw}. If
$p\in(0,1]$, $s\ge [n(q_\wz/p-1)]$, $N>s$ and $N\ge N_{p,\,\wz}$,
where $N_{p,\,\wz}$ is as in (3.29),
then there exists a positive
constant $C_6$ such that for all $f\in h^{p}_{\rho,N}(\wz)$,
$\lz>\dinf_{x\in\rn}{\cal M}_Nf(x)$ and $i\in\nn$,
\beq \label{4.10}
\dint_{\rn}\l({\cal M}^0_N(b_i)(x)\r)^p\wz(x)dx
\le C_6\dint_{Q_i^*} \l({\cal M}_N(f)(x)\r)^p\wz(x)dx.
\eeq
Moreover the series
$\sum_i b_i$ converges in $h^p_{\rho,N}(\wz)$ and
\beq \label{4.11}
\dint_{\rn}\l({\cal M}^0_N\l(\sum_i b_i\r)(x)\r)^p\wz(x)dx
\le  C_6\dint_{\oz} \l({\cal M}_N(f)(x)\r)^p\wz(x)dx.
\eeq
\end{lem}

\begin{proof}
By the proof of Lemma \ref{l4.5}, we know $|x-x_i|<C_4\rz(x)$, $l_i<C_5\rz(x)$
and $\rz(x)\le C_0(1+C_4)^{k_0}\rz(x_i)$, thus $Q_i^*\subset a_2\rz(x_i)Q_i^0$,
where $a_2\equiv2C_0(1+C_4)^{k_0}\max\{C_4,\,C_5\}$ and $Q_i^0\equiv Q(x_i,1)$.
Furthermore, we have
\begin{align} \label{4.12}
\dint_{\rn}({\cal M}^0_N(b_i)(x))^p\wz(x)dx
&\le \dint_{Q^*_i}({\cal M}^0_N(b_i)(x))^p\wz(x)dx   \nonumber\\
&\qquad+\dint_{a_2\rz(x_i)Q^0_i\setminus Q^*_i}({\cal M}^0_N(b_i)(x))^p\wz(x)dx.
\end{align}
Notice that  $s\ge[n(q_\wz/p-1)]$ implies
$2^{-n(q_\wz+\eta)}2^{(s+n+1)p}>1$ for sufficient small $\eta>0$.
By using Lemma 2.1 (iii) with $\wz\in A_{q_\wz+\eta}^{\rz,\,\fz}(\rn)$, Lemma
\ref{l4.5} and  the fact that ${\cal M}_N(f)(x)>\lz$ for all $x\in
Q^*_i$, we have
\beq
\begin{aligned} \label{4.13}
\dint_{a_2\rz(x_i)Q^0_i\setminus Q^*_i}({\cal M}^0_N(b_i)(x))^p\wz(x)dx
&\le \dsum_{k=0}^{k_0}\dint_{2^kQ_i^*\setminus 2^{k-1}Q_i^*}
({\cal M}^0_N(b_i)(x))^p\wz(x)dx\\
&\le \lz^p\wz(Q^*_i)\dsum_{k=0}^{k_0}[2^{-n(q_\wz+\eta)+(s+n+1)p}]^{-k}\\
&\le C\dint_{Q_i^*}({\cal M}_Nf(x))^p\wz(x)dx,
\end{aligned}
\eeq
where $b=1+2^{-(10+n)}$, $k_0\in\zz$ such that $2^{k_0-1}bl_i\le a_2\rz(x_i)< 2^{k_0}bl_i$.

Combining the last two estimates we obtain (4.10),
furthermore, we have
 $$\sum_i\dint_{\rn}({\cal M}^0_N(b_i)(x))^p\wz(x)dx
 \le C \dsum_i\dint_{Q_i^*}({\cal M}_Nf(x))^p\wz(x)dx
 \le C\dint_\oz({\cal M}_N(f)(x))^p\wz(x)dx,$$
which together with complete of
$h^p_{\rho,N}(\wz)$ (see Proposition 3.3)implies that $\sum_i b_i$
converges in $h^p_{\rho,N}(\wz)$. Therefore,  the series
$\sum_i b_i$ converges in ${\cal D}'(\rn)$ and
${\cal M}_N^0(\sum_i b_i)(x)\le \sum_i {\cal M}_N^0( b_i)(x)$ by Proposition 3.2,
which gives (4.11). This finishes the proof of Lemma 4.6.
\end{proof}

\begin{lem}\label{l4.7.}
Let $\wz\in A_\fz^{\rz,\,\fz}(\rn)$ and $q_\wz$ be as in (2.4), $s\in\zz_+$,
and integer $N\ge 2$. If $q\in(q_\wz,\fz]$ and $f\in L_\wz^q(\rn)$, then
the series $\sum_i b_i$ converges in $L_\wz^q(\rn)$ and there
exists a positive constant $C_7$, independent of $f$ and $\lz$,
such that
$$\Big\|\sum_i |b_i|\Big\|_{L^q_\wz(\rn)}
\le C_7\|f\|_{L^q_\wz(\rn)}.$$
\end{lem}
\begin{proof}

 The proof for $q=\fz$ is similar to that for $q\in
(q_\wz,\fz)$. So we only give the proof for $q\in(q_\wz,\fz)$. Set
$F_1=\{i\in\nn: l_i\ge L_3\rz(x_i)\}$ and $F_2=\{i\in\nn: l_i< L_3\rz(x_i)\}$.
By Lemma \ref{l4.2}, for $i\in F_2$, we have
\beqs
\begin{aligned}
\dint_{\rn}|b_i(x)|^q\wz(x)dx
&\le\dint_{Q_i^*}|f(x)|^q\wz(x)dx+\dint_{Q_i^*}|P_i(x)\eta_i(x)|^q\wz(x)dx\\
&\le \dint_{Q_i^*}|f(x)|^q\wz(x)dx + \lz^q\wz(Q^*_i).
\end{aligned}
\eeqs
For $i\in F_1$, we have
$$\dint_{\rn}|b_i(x)|^q\wz(x)dx\le\dint_{Q_i^*}|f(x)|^q\wz(x)dx.$$
From these, we obtain
\beqs
\begin{aligned}
\dsum_i\dint_{\rn}|b_i(x)|^q\wz(x)dx
&=\dsum_{i\in F_1}\dint_{\rn}|b_i(x)|^q\wz(x)dx
+\dsum_{i\in F_2}\dint_{\rn}|b_i(x)|^q\wz(x)dx  \\
&\le \dsum_i\dint_{Q_i^*}|f(x)|^q\wz(x)dx
+ C\dsum_{i\in F_2}\lz^q\wz(Q^*_i)\\
&\le \dsum_i\dint_{Q_i^*}|f(x)|^q\wz(x)dx+ C\lz^q\wz(\oz)\\
&\le C \dint_{\rn}|f(x)|^q\wz(x)dx.
\end{aligned}
\eeqs
Combining above estimates with the fact that $\{b_i\}_i$ have finite covering, we obtain
$$\Big\|\sum_i |b_i|\Big\|_{L^q_\wz(\rn)}\le C_7\|f\|_{L^q_\wz(\rn)}.$$
This finishes the proof.
\end{proof}

\begin{lem}\label{l4.8}
 If $N>s\ge 0$ and $\sum_i b_i$ converges in ${\cal D}'(\rn)$,
 then there exists a  positive constant $C_8$, independent of $f$
 and $\lz$, such that for all $x\in\rn$,
 $${\cal M}_N^0(g)(x)\le {\cal M}_N^0(f)(x)\chi_{\oz^\complement}(x)+
C_8\lz \sum_i\dfrac{l_i^{n+s+1}}{(l_i+|x-x_i|)^{n+s+1}}\chi_{\{|x-x_i|<C_4\rz(x)\}}(x)
+C_8\lz \chi_{\oz}(x),$$
where $x_i$ is the center of $Q_i$ and $C_4$ is as in Lemma 4.5.
\end{lem}
\begin{proof}
If $x\not\in\oz$, since
$${\cal M}_N^0(g)(x)\le {\cal M}_N^0(f)(x)+\dsum_i {\cal M}_N^0(b_i)(x),$$
by Lemma 4.5, we obtain
$${\cal M}_N^0(g)(x)\le {\cal M}_N^0(f)(x)\chi_{\oz^\complement}(x)
+ C\lz\dsum_i\dfrac{l_i^{n+s+1}}{(l_i+|x-x_i|)^{n+s+1}}\chi_{\{|x-x_i|<C_4\rz(x)\}}(x).$$
If $x\in\oz$, take $k\in\nn$ such that $x\in Q_k^*$.
Let $J\equiv\{i\in\nn: Q_i^*\bigcap Q_k^*\neq \emptyset\}$. Then the cardinality
of $J$ is bounded by $L$. By Lemma 4.5, we have
$$\dsum_{i\not\in J} {\cal M}_N^0(b_i)(x)\le C\lz \dsum_{i\not\in J}
\dfrac{l_i^{n+s+1}}{(l_i+|x-x_i|)^{n+s+1}}\chi_{\{|x-x_i|<C_4\rz(x)\}}(x).$$
It suffices to estimate the grand maximal function of
$g+\sum_{i\not\in J}b_i=f-\sum_{i\in J}b_i$.
Take $\vz\in {\cal D}^0_N(\rn)$ and $l\in\zz$ such that $0<2^{-l}<\rz(x)$,
then we write
\beq
\begin{aligned}
\Big(f-\sum_{i\in J}b_i\Big)*\vz_i(x)
&=(f\xi)*\vz_l(x)+\Big(\sum_{i\in J}P_i\eta_i\Big)*\vz_l(x)\\
&=f*\widetilde{\Phi}_l(w)+\Big(\sum_{i\in J}P_i\eta_i\Big)*\vz_l(x),
\end{aligned}
\eeq
where $w\in (2^{8+n}nQ_k)\bigcap\oz^{\complement}$, $\xi=1-\sum_{i\in J}\eta_i$ and
$$\widetilde{\Phi}(z)\equiv\vz(z+2^l(x-w))\xi(w-2^{-l}z).$$
Since for $N\ge 2$ there is a constant $C>0$ such that $\|\vz\|_{L^1(\rn)}\le C$
for all $\vz\in {\cal D}^0_N(\rn)$ and Lemma 4.1, we have
 $$\l|\l(\dsum_{i\in J}P_i\eta_i\r)*\vz_l(x)\r|\le C\lz.$$
Finally, we estimate $f*\Phi_l(w)$. There are two cases:
If $2^{-l}\le 2^{-(11+n)}l_k$,  then $f*\Phi_l(w)=0$,
 because $\xi$ vanishes in $Q_k^*$
 and $\vz_l$ is supported in $B(0,2^{-l})$. On the other hand, if $2^{-l}\ge 2^{-(11+n)}l_k$,
 then there exists a positive constant $a_3>1$ such that $2^{-l}<\rz(x)<a_3\rz(w)$.
Take $\Phi(x)\equiv\widetilde{\Phi}(x/2^{m_1})$ and
$m_1\in\nn$ satisfying $2^{m_1-1}\le a_3<2^{m_1}$,
then
$\supp \Phi \subset B(0,R_3)$ where $R_3\equiv2^{3(11+n)}a_3$,
and $\|\Phi\|_{{\cal D}_N}\le C$.
Therefore, $2^{-l-m_1}<\rz(x)/a_3<\rz(w)$ and
 $$\l|(f*\widetilde{\Phi}_l)(w)\r|=2^{-m_1n}|(f*\Phi_{l+m_1})(w)|
 \le C{\cal M}_N f(w)\|\Phi\|_{{\cal D}_N}\le C\lz.$$
According to  above estimates, we have
$$\Big|(f-\sum_{i\in J}b_i)*\vz_l\Big|\le C\lz,$$
then we can get
$${\cal M}^0_N\Big((f-\sum_{i\in J}b_i)\Big)(x)\le C\lz.$$
This finishes the proof of the lemma.
\end{proof}

\begin{lem}\label{l4.9}
Let $\wz\in A_\fz^{\rz\,\fz}(\rn)$, $q_\wz$ be as in (2.4), $q\in(q_{\wz},\,\fz)$,
$p\in (0,1]$ and $N\ge N_{p,\,\wz}$, where $N_{p,\,\wz}$ is as in (3.29).
\begin{enumerate}
 \item[(i)] If $N>s\ge [n(q_\wz/p-1)]$ and $f\in h_{\rho,\,N}^p(\wz)$,
 then ${\cal M}_N^0(g)\in  L_\wz^q(\rn)$ and there
 exists a positive constant $C_9$, independent of $f$ and $\lz$,
 such that
 $$\dint_{\rn}[{\cal M}_N^0(g)(x)]^q\wz(x)dx
 \le C_9\lz^{q-p}\dint_{\rn} [{\cal M}_N(f)(x)]^p\wz(x)dx.$$
 \item[(ii)] If $N\ge 2$ and $f\in L_\wz^q(\rn)$, then $g\in L_\wz^\fz(\rn)$
 and there exists a positive constant $C_{10}$,
independent of $f$ and $\lz$,
 such that $\|g\|_{L_\wz^\fz}\le C_{10}\lz$.
\end{enumerate}
\end{lem}

\begin{proof}
We first prove (i).  Since $f\in h^p_{\rz,\,N}(\wz)$, by Lemma 4.6 and Proposition \ref{p3.2},
$\sum_i b_i$ converges in
both $h^{p}_{\rz,\,N}(\wz)$ and ${\cal D}'(\rn)$.
Notice that $s\ge[n(q_\wz/p-1)]$, by Lemma 4.8 and the proof of Lemma 4.6, we get
\beqs
\begin{aligned}
\dint_{\rn} ({\cal M}^0_N(g)(x))^q\wz(x)dx
&\le  C\lz^q\dsum_i\dint_{\rn}
\l[\dfrac{l_i^{(n+s+1)}}{(l_i+|x-x_i|)^{(n+s+1)}}\chi_{B(x_i,a_2\rz(x_i))}(x)\r]^q\wz(x)dx\\
&\qquad+C\lz^q\int_{\rn}\chi_{\oz}(x)\wz(x)\,dx
+\dint_{\oz^\complement}({\cal M}_N(f)(x))^q\wz(x)dx\\
&\le C\lz^q\dsum_i\wz(Q_i^*)+C\lz^q\wz(\oz)+\dint_{\oz^\complement}({\cal M}_N(f)(x))^q\wz(x)dx\\
&\le C\lz^q\wz(\oz)+C\lz^{q-p}\dint_{\oz^\complement}({\cal M}_N(f)(x))^p\wz(x)dx\\
&\le C_9\lz^{q-p}\dint_{\rn}({\cal M}_N(f)(x))^p\wz(x)dx.
\end{aligned}
\eeqs
Thus, (i) holds.

Next, we prove (ii).
If $f\in L_\wz^q(\rn)$, then $g$ and $\{b_i\}_i$ are
functions. By Lemma 4.7, we know that $\sum_i b_i$ converges in $L^q_\wz(\rn)$
and hence in ${\cal D}'(\rn)$ by Lemma \ref{l2.5}(ii).
Write
$$g=f-\dsum_i b_i=f\Big(1-\dsum_i\eta_i\Big)+\dsum_{i\in
F_2}P_i\eta_i=f\chi_{\oz^\complement}+\dsum_{i\in F_2}P_i\eta_i.$$
By Lemma
4.3, we have $|g(x)|\le C\lz$ for all $x\in\oz$, and by
Proposition 3.1(i), we also have $|g(x)|=|f(x)|\le {\cal M}_N f(x)\le \lz$ for
almost everywhere $x\in \oz^\complement$. Therefore,
$\|g\|_{L_\wz^\fz(\rn)}\le C_{10}\lz$ which yields (ii).
\end{proof}

\begin{cor}\label{c4.1}
Let $\wz\in A_\fz^{\rz,\,\fz}(\rn)$ and $q_\wz$ be as in (2.4). If
$q\in (q_\wz,\fz)$, $p\in (0,1]$ and $N\ge N_{p,\,\wz}$,
where $N_{p,\,\wz}$ is as in (3.29),  then
$h^p_{\rho,N}(\wz)\cap L_\wz^q(\rn)$ is dense in
$h^p_{\rho,N}(\wz)$.
\end{cor}

\begin{proof}
 Let $f\in h^p_{\rho,N}(\wz)$. For any
$\lz>\dinf_{x\in\rn}{\cal M}_Nf(x)$, let $f=g^\lz+\sum_i b_i^\lz$
be the Calder\'on-Zygmund decomposition of $f$ of degree $s$ with
$[n(q_\wz/p-1)]\le s<N$ and height $\lz$ associated to ${\cal M}_N f$.
By Lemma 4.6, we have
$$\Big\|\dsum_i b_i^\lz\Big\|_{h^p_{\rho,N}(\wz)}\le C\dint_{\{x\in\rn: {\cal
M}_Nf(x)>\lz\}}({\cal M}_Nf(x))^p\wz(x)dx.$$
Therefore, $g^\lz\to f$ in $h^p_{\rho, N}(\wz)$ as $\lz\to\fz$.
Moreover, by Lemma 4.9,
we have ${\cal M}_N^0 (g^\lz)\in L_\wz^q(\rn)$, which combined with Proposition \ref{p3.1}(ii)
implies $g^\lz\in L_\wz^q(\rn)$.  Thus, Corollary 4.1 is proved.
\end{proof}

\section{Weighted atomic decompositions of $h^p_{\rho,\,N}(\wz)$ \label{s5}}

\noi In this section, we establish the
equivalence between $h^{p}_{\rho,\,N}(\wz)$ and
$h^{p,\,q,\,s}_{\rho}(\wz)$ by using the Calder\'on-Zygmund
decomposition associated to the local grand maximal function stated
in Section \ref{s4}.

Let $\wz\in A_\fz^{\rz,\,\fz}(\rn), q_\wz$ be as in  (2.4), $p\in (0,1]$,
$N\ge N_{p,\,\wz}$, $s\equiv [n(q_\wz/p-1)]$ and $f\in h^p_{\rho,\,N}(\wz)$.
Take $m_0\in\zz$ such that
$2^{m_0-1}\le\inf_{x\in\rn}{\cal M}_Nf(x)<2^{m_0}$, if
$\inf_{x\in\rn}{\cal M}_Nf(x)=0$, write $m_0=-\fz$.
For each integer $m\ge m_0$ consider the Calder\'on-Zygmund decomposition of $f$ of
degree $s$ and height $\lz=2^m$ associated to ${\cal M}_Nf$, namely
\begin{equation} \label{5.1}
f=g^m+\dsum_{i\in\nn} b_i^m,
\end{equation}
and
$$\oz_m\equiv\{x\in\rn: {\cal M}_Nf(x)>2^m\}, \ Q_i^m\equiv Q_{l_i^m}.$$
In this section, we write $\{Q_i\}_i$, $\{\eta_i\}_i$, $\{P_i\}_i$
 and $\{b_i\}_i$, respectively, as
$\{Q^m_i\}_i$, $\{\eta^m_i\}_i$, $\{P^m_i\}_i$ and $\{b^m_i\}_i$.
The  center and the  sidelength of $Q^m_i$ are
respectively denoted by $x^m_i$ and $l^m_i$. Recall that for all $i$
and $m$,
\begin{equation}\label{5.2}
\sum_i \eta^m_i =\chi_{\oz_{m}},\quad
\supp(b^m_i)\subset\supp (\eta^m_i)\subset Q^{m\ast}_i,
\end{equation}
$\{Q^{m\ast}_i\}_i$ has the bounded interior property, and for all
$P\in\cp_s (\rn)$,
\begin{equation}\label{5.3}
\langle f, P\eta^m_i\rangle=\langle P^m_i,P\eta^m_i\rangle.
\end{equation}

\noi For each integer $m\ge m_0$ and $i,\,j\in\nn$, let $P^{m+1}_{i,\,j}$
be the orthogonal projection of $(f-P^{m+1}_j)\eta^m_i$ on $\cp_s(\rn)$
with respect to the norm
$$\|P\|_j^2\equiv\frac{1}{\int_{\rn}\eta^{m+1}_j (y)\,dy}
\int_{\rn}|P(x)|^2 \eta^{m+1}_j (x)\,dx,$$
namely, $P^{m+1}_{i,\,j}$
is the unique polynomial of $\cp_s (\rn)$ such that for any
$P\in\cp_s (\rn)$,
\begin{equation}\label{5.4}
\langle(f-P^{m+1}_j)\eta^k_i,P\eta^{m+1}_j\rangle
= \int_{\rn}P^{m+1}_{i,\,j} (x)P(x)\eta^{m+1}_j (x)\,dx.
\end{equation}
In what follows,
we denote
$Q_i^{m*}=(1+2^{-(10+n)})Q_i^m$,
$$E^m_1\equiv\left\{i\in\nn:\ l^m_i\ge \rz(x_i^m)/(2^5 n)\r\},\qquad
E^k_2\equiv\left\{i\in\nn:\ l^m_i<\rz(x_i^m)/(2^5 n)\r\},$$
$$F^k_1\equiv\left\{i\in\nn:\ l^m_i\ge L_3 \rz(x_i^m)\r\},\qquad
F^k_2\equiv\left\{i\in\nn:\ l^m_i< L_3\rz(x_i^m)\r\},$$
where $L_3=2^{k_0}C_0$ is as in Section 4.

Observe that
\begin{eqnarray}\label{5.5}
P^{m+1}_{i,\,j}\neq0 \quad \text{if and only if} \quad Q_i^{m\ast}\cap
Q_j^{(m+1)\ast}\neq\emptyset.
\end{eqnarray} Indeed, this follows directly from
the definition of $P^{m+1}_{i,\,j}$.  The following Lemmas \ref{l5.1}-\ref{l5.3}
can be proved by similar methods of Lemmas 5.1-5.3 in \cite{Ta1}.

\begin{lem}\label{l5.1}
Notice that $\oz_{m+1}\subset \oz_m$, then
\begin{enumerate}
\item[\rm(i)] If $Q_i^{m\ast}\cap Q_j^{(m+1)\ast}\neq\emptyset$,
then $l^{m+1}_j\le2^4\sqrt{n}l^m_i$ and
$ Q_j^{(m+1)\ast}\subset2^6nQ_i^{k\ast}\subset\oz_{m}$.

\item[\rm(ii)] There exists a positive integer $L$ such that for
each $i\in\nn$, the cardinality of $\{j\in\nn:\,Q_i^{m\ast}\cap
Q_j^{(m+1)\ast}\neq\emptyset\}$ is bounded by $L$.
\end{enumerate}
\end{lem}

\begin{lem}\label{l5.2}
 There exists a positive constant $C$
 such that for all $i,\,j\in\nn$ and integer $m\ge m_0$ with
$l^{m+1}_j< L_3 \rz(x_j^{m+1})$,
\begin{equation}\label{5.6}
\sup_{y\in\rn}\left|P^{m+1}_{i,\,j}(y)\eta^{m+1}_j (y)\r|\le C2^{m+1}.
\end{equation}
\end{lem}

\begin{lem}\label{l5.3}
For any $k\in\zz$ with $m\ge m_0$,
$$\sum_{i\in\nn}\Big(\sum_{j\in F^{m+1}_2}P^{m+1}_{i,\,j}\eta^{m+1}_j\Big)=0,$$
where the series converges both in $\cd'(\rn)$ and pointwise.
\end{lem}

The following lemma gives the weighted atomic decomposition for a
dense subspace of $h^{p}_{\rz,\,N}(\wz)$.

\begin{lem}\label{l5.4}
Let $\wz\in A_\fz^{\rz,\,\fz}(\rn)$, $q_\wz$ and $N_{p,\,\wz}$
be respectively as in (2.4) and (3.29). If
$ p\in(0,1]$, $s\ge [n(q_\wz/p-1)]$, $N>s$ and $N\ge N_{p,\,\wz}$, then
for any $f\in (L^q_\wz(\rn)\cap h^p_{\rho,\, N}(\wz))$, there
exists numbers $\lz_0\in \cc$ and $\{\lz_i^m\}_{m\ge k_0,i}\subset \cc$,
$(p,\fz,s)_\wz$-atoms $\{a_i^m\}_{m\ge m_0,i}$ and a single atom
$a_0$ such that
\beq \label{5.7}
f=\dsum_{m\ge m_0}\dsum_{i}\lz_i^ma_i^m+\lz_0a_0,
\eeq
where the series converges almost everywhere and in ${\cal D}'(\rn)$.
Moreover, there exists a positive constant $C$, independent of $f$, such that
\beq \label{5.8}
\dsum_{m\ge m_0,i}|\lz_i^m|^p+|\lz_0|^p\le C\|f\|_{h^p_{\rz,\,N}(\wz)}.
\eeq
\end{lem}

\begin{proof}
Let $f\in (L^q_{\wz}(\rn)\cap h^{p}_{\rho,\,N}(\wz))$. We first
consider the case  $m_0 =-\fz$.
As above, for each $m\in\zz$,
$f$ has a Calder\'on-Zygmund decomposition of degree $s$ and height
$\lz=2^m$ associated to ${\cal M}_N (f)$ as in \eqref{5.1}, namely,
$f=g^m +\sum_{i}b_i^m.$
By Corollary 4.1 and Proposition 3.1, $g^m\to f$ in both
$h^p_{\rz,N}(\wz)$ and ${\cal D}'(\rn)$ as $m\to\fz$. By Lemma 4.9 (i),
$\|g^m\|_{L_\wz^q(\rn)}\to 0$ as $m\to-\fz$, and moreover, by Lemma 2.5 (ii),
$g^m\to 0$ in ${\cal D}'(\rn)$ as $m\to-\fz$.
Therefore,
\begin{equation} \label{5.9}
f=\dsum_{m=-\fz}^\fz(g^{m+1}-g^m)
\end{equation}
in ${\cal D}'(\rn)$.
Moreover, since $\supp(\sum_i b_i^m)\subset\oz_m$ and
$\wz(\oz_m)\to 0$ as $m\to \fz$, then $g^m\to f$ almost everywhere
as $m\to\fz$. Thus, (5.9) also holds almost everywhere.
By Lemma 5.3 and
$\sum_i\eta_i^m b_j^{m+1}=\chi_{\oz_m}b_j^{m+1}=b_j^{m+1}$ for all $j$, then we have
\begin{align}\label{5.10}
g^{m+1}-g^m
&= \Big(f-\sum_j b^{m+1}_j \Big)-\Big(f-\sum_i b_i^m\Big) \nonumber\\
&= \sum_i b_i^m-\sum_j b^{m+1}_j
+\sum_i\Big(\sum_{j\in F^{m+1}_2}P^{m+1}_{i,\,j}\eta^{m+1}_j \Big)\nonumber\\
&= \sum_i \Big[b_i^m-\sum_j b^{m+1}_j \eta^m_i
+\sum_{j\in F^{m+1}_2}P^{m+1}_{i,\,j}\eta^{m+1}_j\Big] \equiv\sum_i h^m_i,
\end{align}
where all the series converge in both  ${\cal D}'(\rn)$ and almost
everywhere.
Furthermore, from the definitions of $b_j^m$ and
$b_j^{m+1}$, we infer that when $l^m_i<L_3\rz(x_i^m)$,
\begin{equation}\label{5.11}
h^m_i=f\chi_{\oz_{{m+1}}^\complement}\eta^m_i-P^m_i \eta^m_i
+\sum_{j\in F^{m+1}_2} P^{m+1}_j \eta^m_i \eta^{m+1}_j
+\sum_{j\in F^{m+1}_2}P^{m+1}_{i,\,j}\eta^{m+1}_j,
\end{equation}
and when $l_i^m\ge L_3\rz(x_i^m)$,
\begin{equation}\label{5.12}
h^m_i=f\chi_{\oz_{m+1}^\complement}\eta^m_i
+\sum_{j\in F^{m+1}_2} P^{m+1}_j \eta^m_i \eta^{m+1}_j
+\sum_{j\in F^{m+1}_2}P^{m+1}_{i,\,j}\eta^{m+1}_j.
\end{equation}
By Proposition \ref{p3.1}(i), we know that  for almost every
$x\in\oz_{{m+1}}^\complement$,
$$|f(x)|\le\cm_N (f)(x)\le2^{m+1},$$
which combined with Lemma \ref{l4.2}, Lemma \ref{l5.1}(ii),
 Lemma \ref{5.2}, \eqref{5.11} and \eqref{5.12} implies
that there exists a positive constant $C_{11}$ such that for all
$i\in\nn$,
\begin{equation}\label{5.13}
\|h^m_i\|_{L^{\fz}_{\wz}(\rn)}\le C_{11} 2^m.
\end{equation}
Next, we show that for each $i$ and $m$, $h^m_i$ is either a multiple of a
$(p,\,\fz,\,s)_{\wz}$-atom
or a finite linear combination of
$(p,\,\fz,\,s)_{\wz}$-atom by considering the following two cases
for $i$.

\noi Case I. For $i\in E^m_1$, $l_i^m\ge \rz(x_i^m)/2^5n$. Clearly,
$h_i^m$ is supported in a cube $\wt Q_i^m$ that contains
$Q_i^{m*}$ as well as all the $Q_j^{(m+1)*}$ that intersect $Q_i^{m*}$.
In fact, observe that if
$Q_i^{m*}\cap Q_j^{(m+1)*}\not=\emptyset$, by Lemma 5.1, we have
 $Q_j^{(m+1)*}\subset 2^6nQ_i^{m*}\subset \oz_m,$ thus, we set
$\wt Q_i^m\equiv2^6nQ_i^{m*}.$
Since $l(\wt Q_i^m)\ge 2\rz(x_i^m)$,
by the same method of Lemma 3.1 in \cite{Ta3}, $\wt Q_i^m$
can be decomposed into finite disjoint cubes $\{Q_{i,\,k}^m\}_k$
such that $\wt Q_i^m=\bigcup_{k=1}^{n_i}Q_{i,\,k}^m$
and
$l_{i,k}^m/4<\rz(x)\le C_0(3\sqrt{n})^{k_0}l_{i,k}^m$
for some $x\in Q_{i,\,k}^m=Q(x_{i,k}^m,l_{i,k}^m)$,
where $C_0,\ k_0$ are  constants given in Lemma \ref{l2.1}.
Moreover, by Lemma 2.1, we also have $l_{i,k}^m\le L_1 \rz(x_{i,k}^m)$ and
$l_{i,k}^m> L_2 \rz(x_{i,k}^m).$
 Therefore, let
$$\lz_{i,k}^m\equiv C_{11}2^m[\wz(Q_{i,k}^m)]^{1/p}
\quad \text{and}
\quad a_{i,k}^m\equiv(\lz_{i,k}^m)^{-1}
\frac{h_i^m\chi_{Q_{i,\,k}^m}}{\sum_{k=1}^{n_i}\chi_{Q_{i,\,k}^m}},
$$
then $\supp\, a_{i,k}^m\subset Q_{i,\,k}^m$ and
$\|a_{i,k}^m\|_{L^{\fz}_\wz(\rn)}\le [\wz(Q_{i,k}^m)]^{-1/p}$,
hence each $a_{i,k}^m$ is a $(p,\,\fz,\,s)_{\wz}$-atom
and $h_i^m=\sum_{k=1}^{n_i}\lz_{i,k}^m a_{i,k}^m$.

\noi Case II.
For $i\in E^m_2$, if $j\in F^{m+1}_1$,
we claim that $Q^{m\ast}_i \cap Q^{(m+1)\ast}_j=\emptyset$.
In fact, if  $Q^{m\ast}_i \cap Q^{(m+1)\ast}_j\neq\emptyset$,
by Lemma \ref{l5.1} (i), we know $l^{m+1}_j\le2^4\sqrt{n}l^m_i$
then we can deduce that
$l_i^m<l_i^m/2\sqrt{n}$ which is a contradiction, hence the claim is true.
Thus, we have
\begin{align} \label{5.14}
h^m_i&=\left(f-P^m_i\r)\eta^m_i-\sum_{j\in F^{m+1}_1}f\eta^{m+1}_j\eta^m_i
 -\sum_{j\in F^{m+1}_2}\left(f-P^{m+1}_{j}\r)\eta^{m+1}_j \eta^m_i\nonumber \\
&\qquad \qquad+\sum_{j\in F^{m+1}_2}P^{m+1}_{i,\,j}\eta^{m+1}_j \nonumber \\
&=\left(f-P^m_i\r)\eta^m_i -\sum_{j\in F^{m+1}_2}
\left\{\left(f-P^{m+1}_{j}\r)\eta^{m+1}_j \eta^m_i-P^{m+1}_{i,\,j} \eta^{m+1}_j\r\}.
\end{align}
Let $\wt{Q}^m_i\equiv2^6 n Q^{m\ast}_i$, then $l(\wt{Q}^m_i)<L_1\rz(x_i^m)$
and $\supp\,h^m_i\subset\wt{Q}^m_i$.
Moveover, $h^m_i$ satisfies the desired moment conditions, which are
deduced from the moment conditions of $(f-P^m_i)\eta^m_i$ (see \eqref{5.3})
and $(f-P^{m+1}_{j})\eta^{m+1}_j\eta^m_i-P^{m+1}_{i,\,j}\eta^{m+1}_j$ (see \eqref{5.4}).
Let $\lz_i^m\equiv C_{11}2^m[\wz(\wt Q_i^m)]^{1/p}$
and $a_i^m\equiv(\lz_i^m)^{-1}h_i^m$, then $a_i^m$ is a $(p,\fz,s)_\wz$-atom.

Thus, from \eqref{5.9}, \eqref{5.10}, Case I and Case II, we
infer that
$$f=\sum_{m\in \zz}\bigg(\sum_{i\in E_1^m}\bigg(\sum_{k=1}^{n_i}\lz^m_{i,k} a^m_{i,k}\bigg)
+\sum_{i\in E_2^m}\lz^m_i a^m_i\bigg)$$
holds in both ${\cal D}'(\rn)$ and almost everywhere.
Moreover, by Lemma \ref{l2.4}, we get
\beqs
\begin{aligned}
\dsum_{k\in\zz}\bigg[\sum_{i\in E_1^m}\bigg[\sum_{k=1}^{n_i}|\lz^m_{i,k}|^p \bigg]
+\sum_{i\in E_2^m}|\lz^m_i|^p \bigg]
&\le C\dsum_{k\in\zz} 2^{mp}\bigg[\sum_{i\in E_1^m}\bigg[\sum_{k=1}^{n_i}\wz(Q_{i,k}^m) \bigg]
+\sum_{i\in E_2^m}\wz(\wt Q_i^m)\bigg] \\
&\le C\dsum_{k\in\zz}2^{mp}\bigg[\sum_{i\in E_1^m}\wz(\wt Q_i^m)
+\sum_{i\in E_2^m}\wz(\wt Q_i^m)\bigg] \\
&\le C\dsum_{m\in\zz}\sum_{i\in\nn}2^{mp}\wz(\wt Q_i^m)\\
\end{aligned}
\eeqs
\beqs
\begin{aligned}
&\qquad\qquad\qquad\qquad\qquad \le C\dsum_{m\in\zz}\sum_{i\in\nn}2^{mp}\wz( Q_i^{m*}) \\
&\qquad\qquad\qquad\qquad\qquad \le C \dsum_{m\in\zz} 2^{mp} \wz(\oz_m)\\
&\qquad\qquad\qquad\qquad\qquad \le C\|{\cal M}_N(f)\|_{L^p_\wz(\rn)}^p
= C\|f\|_{h^p_{\rho,N}(\wz)}^p,
\end{aligned}
\eeqs
which implies \eqref{5.8} in the case that $m_0 =-\fz$.

Finally, we consider the case that $k_0>-\fz$. In this case, by
$f\in h^{p}_{\rho,\,N}(\wz)$, we see that $\wz(\rn)<\fz$.
Adapting the previous arguments, we conclude that
\begin{equation}\label{5.15}
f=\sum_{m=m_0}^{\fz}\left(g^{m+1}-g^m\r)+g^{m_0} \equiv\wt{f}+g^{m_0}.
\end{equation}
For the function $\wt f$, we have the same $(p,\fz,s)_\wz$
atomic decomposition as above
\begin{equation}\label{5.16}
\wt{f}= \sum_{m\ge m_0,\,i} \lz^m_i a^m_i,
\end{equation}
and
\begin{equation}\label{5.17}
\dsum_{m\ge m_0}\dsum_{i\in\nn}|\lz_i^m|^p\le
C\| f\|_{h^p_{\rho,N}(\wz)}^p.
\end{equation}
For the function $g^{m_0}$, by Lemma \ref{l4.9} (ii), we have
\begin{equation}\label{5.18}
\|g^{m_0}\|_{L^\fz_\wz(\rn)}\le C_{10} 2^{m_0}
\le 2C_{10}\inf_{x\in\rn}{\cal M}_Nf(x),
\end{equation}
where $C_{10}$ is the same as in Lemma \ref{l4.9} (ii).
Let $\lz_0\equiv C_{10}2^{m_0}[\wz(\rn)]^{1/p}$
and $a_0\equiv \lz_0^{-1}g^{m_0}$, then
\begin{equation}\label{5.19}
\|a_0\|_{L^\fz_\wz(\rn)}\le [\wz(\rn)]^{-1/p} \quad \text{and} \quad
|\lz_0|^p\le (2C_{10})^p\| f\|_{h^p_{\rho,N}(\wz)}^p.
\end{equation}
Thus, $a_0$ is a $(p,\,\fz)_{\wz}$-single-atom and
$g^{m_0}=\lz_0 a_0$, which together with \eqref{5.15} and
\eqref{5.16} implies \eqref{5.7} in the case that $k_0>-\fz$.
Furthermore, by combining (5.17) with (5.19), we obtain
$$\dsum_{m\ge m_0}\dsum_{i\in\nn}|\lz_i^m|^p
+|\lz_0|^p\le C\| f\|_{h^p_{\rho,N}(\wz)}^p.$$
The proof of the lemma is complete.
\end{proof}

Now we state the weighted atomic decompositions of $h^{p}_{\rho,\,N}(\wz)$.

\begin{thm}\label{t5.1}
Let $\wz\in A_\fz^{\rz,\,\fz}(\rn)$, $q_\wz$ and $N_{p,\,\wz}$
be respectively as in (2.4) and (3.29).
If $q\in (q_\wz,\fz],\ p\in (0,1]$,  and
integers $s$ and $N$ satisfy $N\ge N_{p,\wz}$ and
$N>s\ge[n(q_\wz/p-1)]$, then $h^{p,q,s}_{\rho}(\wz)=h^p_{\rho, N}(\wz)=
h^p_{\rho, N_{p,\wz}}(\wz)$ with equivalent norms.
\end{thm}
\begin{proof}
It is easy to see that
$$h^{p,\fz,\bar s}_{\rho}(\wz)\subset h^{p,\,q,\,s}_{\rho}(\wz)\subset
h^p_{\rho, N_{p,\wz}}(\wz)\subset h^p_{\rho, N}(\wz) \subset
h^p_{\rho, \bar N}(\wz),$$
where $\bar s$ is an integer no less
than $s$ and $\bar N$ is an integer larger than $N$, and the
inclusions are continuous.
Thus, to prove Theorem 5.1, it suffices
to prove that for any $N>s\ge[n(q_\wz/p-1)]$,
$h^p_{\rho, N}(\wz)\subset h^{p,\fz,s}_{\rho}(\wz)$,
and for all $f\in h^p_{\rho, N}(\wz)$,
$\|f\|_{h^{p,\fz,s}_{\rho}(\wz)}\le C\|f\|_{h^p_{\rho, N}(\wz)}$.

Let $f\in h^p_{\rho,N}(\wz)$. By Corollary 4.1, there
exists a sequence of functions
$\{f_m\}_{m\in\nn}\subset (h^p_{\rho,N}(\wz)\cap L_\wz^q(\rn))$
such that for all $m\in\nn$,
\begin{equation}\label{5.20}
\|f_m\|_{h^p_{\rho, N}(\wz)}\le 2^{-m} \|f\|_{h^p_{\rho,N}(\wz)}
\end{equation}
and $f=\sum_{m\in\nn}f_m$ in $h^p_{\rho, N}(\wz)$.
By Lemma 5.4, for each $m\in\nn$, $f_m$ has an atomic decomposition
$$f_m=\sum_{i\in\zz_+}\lz_i^ma_i^m$$
in ${\cal D}'(\rn)$ with
$$\sum_{i\in\zz_+}|\lz_i^m|^p\le C\|f_m\|_{h^p_{\rho,N}(\wz)}^p,$$
where $\{\lz^m_i\}_{i\in\zz_+}\subset\cc$, $\{a^m_i\}_{i\in\nn}$ are
$(p,\,\fz,\,s)_{\wz}$-atoms and $a^m_0$ is a
$(p,\,\fz)_{\wz}$-single-atom.

Let
$$\wt{\lz}_0\equiv[\wz(\rn)]^{1/p}\sum^{\fz}_{m=1}|\lz^m_0|\|a^m_0\|_{L^{\fz}_{\wz}(\rn)}
\quad \text{and} \quad
\wt{a}_0\equiv(\wt{\lz}_0)^{-1}\sum^{\fz}_{m=1}\lz^m_0 a^m_0.$$
Then
$$\wt{\lz}_0\wt{a}_0=\sum^{\fz}_{m=1}\lz^m_0 a^m_0.$$
It is easy to see that
$$\|\wt{a}_0\|_{L^{\fz}_{\wz}(\rn)}\le [\wz(\rn)]^{-1/p},$$
which implies that $\wt{a}_0$ is a
$(\rz,\fz)_{\wz}$-single-atom.
Since $\|a_0^m\|_{L_{\wz}^{\fz}(\rn)}\le (\wz(\rn))^{-1/p}$ and
$$|\lz_0^m|\le C\|f_m\|_{h^p_{\rho,N}(\wz)}\le C 2^{-m}\|f\|_{h^p_{\rho,N}(\wz)},$$
we have
$$|\wt{\lz}_0|\le C\Big(\sum_{m=1}^{\fz}2^{-m}\Big)\|f\|_{h^p_{\rho,N}(\wz)}
\le C\|f\|_{h^p_{\rho,N}(\wz)},$$
moreover, we also have
$$\dsum_{m\in\nn}\dsum_{i\in\nn}|\lz_i^m|^p +|\wt{\lz}_0|^p
\le C\Big(\dsum_{m\in\nn}\|f_m\|_{h^p_{\rho,N}(\wz)}^p+\|f\|^p_{h^p_{\rho,N}(\wz)}\Big)
\le C\|f\|_{h^p_{\rho,N}(\wz)}^p.$$
Finally, we obtain
$$f=\sum_{m\in\nn}\sum_{i\in\nn}\lz^m_i a^m_i
+\wt{\lz}_0\wt{a}_0\in h^{p,\,\fz,\,s}_{\rho}(\wz)$$
and
$$\|f\|_{h^{p,\,\fz,\,s}_{\rho}(\wz)}
\le C \|f\|_{h^{p}_{\rho,\,N}(\wz)}.$$
The theorem is  proved.
\end{proof}

For simplicity, from now on, we denote by $h^p_{\rho}(\wz)$ the
weighted local Hardy space $h^p_{\rho, N}(\wz)$  when $N\ge N_{p,\,\wz}$.

\section{Finite atomic decompositions\label{s6}}

\noi In this section, we prove that for any given finite
linear combination of weighted atoms when $q <\fz$, its norm in
$h^{p}_{\rz,\,N}(\wz)$ can be achieved via all its finite
weighted atomic decompositions. This extends the main results in \cite{MSV} to the
setting of weighted local Hardy spaces.

\begin{defn}\label{d6.1}
Let  $\wz\in A^{\rz,\,\fz}_{\fz}(\rn)$ and
$(p,\,q,\,s)_{\wz}$ be admissible as in Definition \ref{d3.3}.
Then  $h^{p,\,q,\,s}_{\rho,\,\fin}(\wz)$ is defined
to be the vector space of all finite linear combinations
of $(p,\,q,\,s)_{\wz}$-atoms and
a $(p,\,q)_{\wz}$-single-atom, and the norm of $f$
in $h^{p,\,q,\,s}_{\rho,\,\fin}(\wz)$ is defined by
\beqs
\begin{aligned}
\|f\|_{h^{p,q,s}_{\rho,\fin}(\wz)}
&\equiv\dinf\Big\{\Big[\dsum_{i=0}^k|\lz_i|^p\Big]^{1/p}:
f=\dsum_{i=0}^k\lz_ia_i,\ k\in\zz_+,\ \{\lz_i\}_{i=0}^k\subset\cc,\
\{a_i\}_{i=1}^k\ {\rm are}\ \Big. \\
&\qquad\qquad (p,q,s)_\wz\ {\rm atoms},\ {\rm and}\ a_0\ {\rm is\ a}\ (p,q)_\wz\ {\rm single\ atom}\Big\}.
\end{aligned}
\eeqs
\end{defn}

Obviously, for any admissible triplet
$(p,\,q,\,s)_\wz$ atom and $(p,\,q)_\wz$ single atom,
$h_{\rho,\fin}^{p,q,s}(\wz)$ is dense in $h^{p,q,s}_{\rho}(\wz)$ with
respect to the quasi-norm $\|\cdot\|_{h_{\rho}^{p,q,s}(\wz)}$.

\begin{thm}\label{t6.1}
Let $\wz\in A^{\rz,\,\fz}_{\fz}(\rn)$, $q_{\wz}$ be as in (2.4) and
$(p,\,q,\,s)_{\wz}$ be admissible as in Definition \ref{d3.3}.
If $q\in(q_{\wz},\fz)$, then
$\|\cdot\|_{h^{p,\,q,\,s}_{\rho,\,\fin}(\wz)}$ and
$\|\cdot\|_{h^{p}_{\rho}(\wz)}$ are equivalent quasi-norms on
$h^{p,\,q,\,s}_{\rho,\,\fin}(\wz)$.

\end{thm}

\begin{proof}
Obviously, by Theorem \ref{t5.1}, we infer that
$h^{p,\,q,\,s}_{\rho,\,\fin}(\wz)\subset
h^{p,\,q,\,s}_{\rho}(\wz)=h^{p}_{\rho}(\wz)$
and for all $f\in h^{p,\,q,\,s}_{\rho,\,\fin}(\wz)$, we have
$$\|f\|_{h^{p}_{\rho}(\wz)}\le C \|f\|_{h^{p,\,q,\,s}_{\rho,\,\fin}(\wz)}.$$
Thus, it suffices to show that for every $q\in(q_\wz,\fz)$ there exists a
constant $C$ such that for all $f\in h^{p,\,q,\,s}_{\rho,\,\fin}(\wz)$,
\begin{equation}\label{6.1}
\|f\|_{h^{p,\,q,\,s}_{\rho,\,\fin}(\wz)}\le C\|f\|_{h^{p}_{\rho}(\wz)}.
\end{equation}
Suppose  that $f$ is in
$h^{p,q,s}_{\rho,\fin}(\wz)$ with $\|f\|_{h^p_{\rho}(\wz)}=1$.
In this section, as in Section 5, we take $m_0\in\zz$ such that
$2^{m_0-1}\le\inf_{x\in\rn}{\cal M}_Nf(x)<2^{m_0}$, and for
$\inf_{x\in\rn}{\cal M}_Nf(x)=0$, write $m_0=-\fz$.
For each integer $m\ge m_0$, set
$$\oz_m\equiv\{x\in\rn: {\cal M}_Nf(x)>2^m\},$$
where and in what follows $N=N_{p,\wz}$.
Since $f\in(h^{p}_{\rho,\,N}(\wz)\cap L^{q}_{\wz}(\rn))$,
by Lemma \ref{l5.4}, there exist $\lz_0\in\cc$,
$\{\lz^m_i\}_{m\ge k_0,\,i}\subset\cc$, a
$(p,\,\fz)_{\wz}$-single-atom $a_0$ and
$(p,\,\fz,\,s)_{\wz}$-atoms $\{a^m_i\}_{m\ge m_0,\,i}$, such that
\begin{equation}\label{6.2}
f=\sum_{m\ge m_0}\sum_{i}\lz^m_i a^m_i+\lz_0 a_0
\end{equation}
holds both in ${\cal D}'(\rn)$ and almost everywhere.
First, we claim that \eqref{6.2} also holds in $L^q_{\wz}(\rn)$.
For any $x\in\rn$,
by $\rn=\cup_{m\ge m_0}(\oz_{2^m}\setminus \oz_{2^{k+1}})$,
we see that there exists $j\in\zz$ such that
$x\in(\oz_{2^j}\setminus \oz_{2^{j+1}})$.
By the proof of Lemma \ref{l5.4}, we know that for
all $m>j$, $\supp(a^m_i)\subset \wt{Q}^m_i\subset\oz_m\subset\oz_{j+1}$;
then from \eqref{5.13} and \eqref{5.18}, we conclude that
$$\bigg|\sum_{m\ge m_0}\sum_{i}\lz^m_i a^m_i(x)\bigg|+|\lz_0 a_0 (x)|
\le C \sum_{k_0\le k\le j}2^k +2^{k_0}
\le C 2^j\le C\cm_N (f)(x).$$
Since $f\in L^q_{\wz}(\rn)$, from Proposition
\ref{p3.1}(ii), we infer that $\cm_N (f)(x)\in L^q_{\wz}(\rn)$. This
combined with the Lebesgue dominated convergence theorem implies that
$$\sum_{m\ge m_0}\sum_{i}\lz^m_i a^m_i+\lz_0 a_0$$
converges to $f$ in $L^q_{\wz}(\rn)$, which deduce the claim.

Next, we show \eqref{6.1} by considering the following two cases for $\wz$.

\noi Case I: $\wz(\rn)=\fz$. In this case, as $f\in L^{q}_{\wz}(\rn)$,
we know that $m_0=-\fz$ and $a_0 (x)=0$ for
almost every $x\in\rn$ in \eqref{6.2}. Thus, in this case,
\eqref{6.2} can be written as
$$f=\sum_{m\in\zz} \sum_{i}\lz^m_ia^m_i.$$
Since, when $\wz(\rn)=\fz$, all
$(\rz,\,q)_{\wz}$-single-atoms are 0,
which implies that $f$ has compact support for $f\in h^{p,\,q,\,s}_{\rz,\,\fin}(\wz)$.
Assume that $\supp(f)\subset Q_0 \equiv Q(x_0,r_0)$ and
$\wt{Q}_0\equiv Q(x_0,r_1)$, in which
$r_1=\sqrt{n}r_0+C_0^2(1+R)^{k_0+1}(1+\sqrt{n}r_0/\rz(x_0))\rz(x_0)$.
Then for any $\psi\in{\cal D }_N (\rn)$, $x\in\rn\setminus\wt{Q}_0$ and
$2^{-l}\in(0,\rz(x))$,
we have
\beqs
\psi_l \ast f(x)=\int_{Q(x_0,r_0)}\psi_l(x-y)f(y)\,dy
=\int_{B(x,R\rz(x))\cap Q(x_0,r_0)}\psi_l(x-y)f(y)\,dy=0.
\eeqs
Thus, for any $m\in\zz$, $\oz_m\subset\wt{Q}_0$, which implies
that  $\supp(\sum_{m\in\zz} \sum_{i}\lz^m_i a^m_i)\subset \wt{Q}_0$.
For each positive integer $K$, let
$$F_K\equiv\{(m,i):\,m\in\zz,\,m\ge m_0, i\in\nn,\,|m|+i\le K\},$$
and
$$f_K\equiv\sum_{(m,i)\in F_K}\lz^m_i a^m_i.$$
Then, by the above claim, we know that $f_K$ converges
to $f$ in $L^q_{\wz}(\rn)$.
Hence, for any given $\vez\in(0,1)$, there exists a $K_0 \in\nn$ large
enough such that $\supp(f-f_{K_0})/\varepsilon\subset\wt{Q}_0$ and
$$\|(f-f_{K_0})/\varepsilon\|_{L^q_{\wz}(\rn)}
\le [\wz(\wt{Q}_0)]^{1/q-1/p}.$$
For $\wt{Q}_0$, since $l(\wt{Q}_0)=r_1>2\rz(x_0)$,
we can decompose it into finite disjoint cubes $\{Q_j\}_j$
such that $\wt Q_0=\bigcup_{j=1}^{N_0}Q_{j}$
and
$l_j/4<\rz(x)\le C_0(3\sqrt{n})^{k_0}l_j$
for some $x\in Q_j=Q(x_j,l_j)$.
Moreover, each $l_j$ satisfies $L_2\rz(x_j)<l_j<L_1\rz(x_j)$.
It is clear that for $q\in(q_{\wz},\,\fz)$ and $p\in(0,\,1]$
we have
$$\|(f-f_{K_0})\chi_{Q_i}/\vez\|_{L^q_{\wz}(\rn)}
\le[\wz(\wt{Q}_0)]^{1/q-1/p}
\le [\wz(Q_j)]^{1/q-1/p},$$
which together with $\supp((f-f_{K_0})\chi_{Q_j}/\vez)\subset Q_j$
implies that $(f-f_{K_0})\chi_{Q_j}/\vez$ is a
$(p,\,q,\,s)_{\wz}$-atom for $j=1,\,2,\cdots,\,N_0$.
Therefore,
$$f=f_{K_0}+\sum_{j=1}^{N_0}\vez \frac{(f-f_{K_0})\chi_{Q_j}}{\vez}$$
 is a finite  weighted atom linear combination of $f$ almost everywhere.
Then take $\vez\equiv N_0^{-1/p}$,
we obtain
$$\|f\|^p_{h^{p,\,q,\,s}_{\rz,\,\fin}(\wz)}
\le \dsum_{(m,i)\in F_K}|\lz_i^m|^p+N_0\vez^p\le C,
$$
which implies the Case I.

\noi Case II: $\wz(\rn)<\fz$. In this case, $f$ may not have compact
support. Similarly to Case I, for any positive integer $K$, let
$$f_K\equiv\sum_{(m,i)\in F_K}\lz^m_i a^m_i+\lz_0 a_0$$
and $b_K\equiv f-f_K$, where $F_K$ is as in Case I.
From the above claim, we
deduce that $f_K$ converges to $f$ in $L^q_{\wz}(\rn)$.
Thus, there exists a positive integer $K_1\in\nn$ large enough such that
$$\l\|b_{K_1}\r\|_{L^q_{\wz}(\rn)}\le [\wz(\rn)]^{1/q-1/p}.$$
Thus, $b_{K_1}$ is a
$(p,\,q)_{\wz}$-single-atom and $f=f_{K_1}+b_{K_1}$ is a finite
 weighted atom linear combination of $f$.
 By Lemma 5.4, we have
$$\|f\|^p_{h^{p,q,s}_{\rho,\fin}(\wz)}
\le C\Big(\dsum_{(m,i)\in F_K}|\lz_i^m|^p+\lz_0^p\Big)\le C.$$
Thus, (6.1) holds, and the theorem is now proved.

\end{proof}

 As an application of finite atomic
decompositions, we establish boundedness in $h^p_\rz(\wz)$ of
quasi- Banach-valued sublinear operators.

As in \cite{BLYZ}, a  quasi-Banach space space ${\cal B}$
is a vector space endowed with a quasi-norm $\|\cdot\|_{{\cal B}}$
which is nonnegative, non-degenerate (i.e., $\|f\|_{\cal B}=0$ if
and only if $f=0$), homogeneous, and obeys the quasi-triangle
inequality, i.e., there exists a positive constant $K$ no less
than $1$ such that for all $f,g\in{\cal B}$, $\|f+g\|_{{\cal
B}}\le K(\|f\|_{\cal B}+\|g\|_{\cal B})$.

Let $\bz\in (0,1]$. A quasi-Banach space ${\cal B}_\bz$ with the
quasi-norm $\|\cdot\|_{{\cal B}_\bz}$ is called a
$\bz$-quasi-Banach space if $\|f+g\|^\bz_{{\cal B}_\bz}\le
\|f\|^\bz_{{\cal B}_\bz}+\|g\|^\bz_{{\cal B}_\bz}$ for all $f,g\in
{\cal B}_\bz$.

Notice that any Banach space is a $1$-quasi-Banach space, and the
quasi-Banach space $l^\bz,\ L^\bz_\wz(\rn)$ and $h^\bz_\wz(\rn)$
with $\bz\in (0,1)$ are typical $\bz$-quasi-Banach spaces.

For any given $\bz$-quasi-Banach space ${\cal B}_\bz$ with $\bz\in
(0,1]$ and  a linear space ${\cal Y}$, an operator $T$ from
 ${\cal Y}$ to ${\cal B}_\bz$ is said to be ${\cal B}_\bz$-sublinear
 if for any $f, g\in {\cal B}_\bz$  and $\lz,\ \nu\in\cc$,
 $$\|T(\lz f+\nu g)\|_{{\cal B}_\bz}\le\l(|\lz|^\bz\|T(f)\|_{{\cal
 B}_\bz}^\bz +|\nu|^\bz\|T(g)\|_{{\cal B}_\bz}^\bz\r)^{1/\bz}$$
and
$\|T(f)-T(g)\|_{{\cal B}_\bz}\le \|T(f-g)\|_{{\cal B}_\bz}$.

We remark that if $T$ is linear, then it is ${\cal B}_\bz$-sublinear.
Moreover, if ${\cal B}_\bz$ is a space of functions, and
$T$ is nonnegative and sublinear in the classical sense, then $T$
is also ${\cal B}_\bz$-sublinear.

\begin{thm}\label{t6.2}
Let $\wz\in A_\fz^{\rz,\,\fz}(\rn), 0<p\le\bz\le 1$, and ${\cal B}_\bz$ be a
$\bz$-quasi-Banach space. Suppose $q\in (q_\wz,\fz)$ and
$T: h^{p,q,s}_{\rho,\,\fin}(\wz)\to {\cal B}_\bz$
is a ${\cal B}_\bz$-sublinear operator  such that
$$
S\equiv\sup\{\|T(a)\|_{{\cal B}_\bz}:
\ a \ {\rm is\ a\ }
(p,q,s)_\wz{\ \rm atom}
\ {\rm or}\ (p,q)_\wz\ {\rm  single\ atom}\}<\fz.
$$
Then there exists a unique bounded ${\cal B}_\bz$-sublinear
operator $\wt T$ from $h^p_\rho(\wz)$ to ${\cal B}_\bz$ which extends $T$.
\end{thm}

\begin{proof}
For any $f\in h^{p,q,s}_{\rho,\fin}(\wz)$, by Theorem 6.1,
there exist a set of numbers $\{\lz_j\}_{j=0}^l\subset\cc$,
$(p,\,q,\,s)_\wz$-atoms $\{a_j\}_{j=1}^l$ and  a $(p,\,q)_\wz$ single
atom $a_0$ such that $f=\sum_{j=0}^l\lz_ja_j$ pointwise and
$$\sum_{j=0}^l|\lz_j|^p\le C\|f\|_{h^p_\rho(\wz)}^p.$$
Then by the assumption, we have
$$\|T(f)\|_{{\cal B}_\bz}\le C\bigg[\dsum_{j=0}^l|\lz_j|^p\bigg]^{1/p}
\le C\|f\|_{h^p_\rho(\wz)}.$$
Since $h^{p,\,q,\,s}_{\rho,\,\fin}(\wz)$ is dense
in $h^p_\rho(\wz)$, a density argument gives the desired results.
\end{proof}

\section{Atomic characterization of \wha}

\noi In this section, we
apply the  atomic characterization of the  weighted local
Hardy spaces $h_{\rho}^1(\wz)$ with  $A_1^{\rz,\tz}(\rn)$ weights  to establish atomic characterization of weighted Hardy space \wha\
associated to Schr\"{o}dinger operator  with $A_1^{\rz,\tz}(\rn)$ weights.

Let $\L=-\Delta+V$ be a  Schr\"{o}dinger operator on $\rn,\ n\geq 3$,
where $V \in  RH_{n/2}$ is a fixed non-negative potential.

Let $\{T_t\}_{t>0}$ be the semigroup of linear operators generated by $\L$ and
$T_t(x,y)$ be their kernels, that is,
\beq
T_t f(x)=e^{-t\L}f(x)=\int_{\rn}T_t(x,y)f(y)\,dy,
\qquad {\text{for}}\ t>0\ {\text{and}}\ f\in L^{2}(\rn).
\eeq
Since \v \ is non-negative the Feynman-Kac formula implies that
\beq
0\leq T_t(x,y)\leq \wt{T}_t(x,y)\equiv(4\pi t)^{-\frac{n}{2}}\exp\l(-\frac{|x-y|^{2}}{4t}\r).
\eeq
Obviously, by (1.2) the maximal operator
\beqs
\T^{*}f(x)=\dsup_{t>0}\l|T_tf(x)\r|
\eeqs
is of weak-type (1,1).
A weighted Hardy-type space related to $\L$ with $A_1^{\rz,\tz}(\rn)$ weights
is naturally defined by:
\beq
H_{\L}^{1}(\omega)\equiv\{f\in L^{1}_\wz(\rn): \T^{*}f(x)\in L^{1}_\wz(\rn)\},
\qquad\text{with}
\quad \whaf \equiv\|\T^{*}f\|_{L^{1}_\wz(\rn)}.
\eeq
The $H_{\L}^{1}(\omega)$ with $\wz\in A_1(\rn)$ has been studied in \cite{LTZ,ZL}

Now let us recall some basic properties of kernels $T_t(x,y)$ and the operator $\T^*$

\begin{lem}\label{l7.1}{\rm (see \cite{DZ1})}
For every $l>0$ there is a constant $C_l$ such that
\beq T_t(x,y)\leq C_l (1+|x-y|/\rz(x))^{-l}|x-y|^{-n},\eeq
for $x,y\in \rn$. Moreover, there is an $\vez>0$ such that for every $C'>0$,
there exists $C$ so that
\beq |T_t(x,y)-\wt{T}_t(x,y)|\leq C\frac{(|x-y|/\rz(x))^{\vez}}{|x-y|^n}, \eeq
for $|x-y|\leq C'\rz(x).$
\end{lem}

Since $T_t(x,y)$ is a symmetric function, we also have
\beq T_t(x,y)\leq C_l (1+|x-y|/\rz(y))^{-l}|x-y|^{-n},\qquad \text{for}\  x,y\in \rn. \eeq

\begin{lem}\label{l7.2}{\rm (see \cite{DZ2})}
There exist a rapidly decaying function $w\ge 0$ and a $\delta>0$
such that
\beq
|T_t(x,y)-\wt{T}_t(x,y)|\leq \l(\frac{\sqrt{t}}{\rz(x)}\r)^{\delta}w_{\sqrt{t}}(x-y).
\eeq
\end{lem}

\begin{lem}\label{l7.3}{\rm (see \cite{DZ3})}
 If $V\in RH_s(\rn),\ s>n/2$, then there exist $\dz=\dz(s)>0$ and $c>0$ such that for
 every $N>0$, there is a constant $C_N$ so that, for all $|h|<\sqrt{t}$
\beq |T_t(x+h,y)-T_t(x,y)|\leq C_N \l(\frac{|h|}{\sqrt{t}}\r)^{\dz} t^{-\frac{n}{2}}
 \l(1+\frac{\sqrt{t}}{\rz(x)}+\frac{\sqrt{t}}{\rz(y)}\r)^{-N}
  \exp\l(-\frac{c|x-y|^2}{t}\r).\eeq
\end{lem}

\begin{lem}\label{l7.4}{\rm (see \cite{BHS1})}
For $1<p<\fz$ the operator $\T^*$ is bounded on $L^p(\wz)$
when $\wz\in A_p^{\rz,\,\fz}(\rn)$, and of weak type (1,1) when
$\wz\in A_1^{\rz,\,\fz}(\rn)$.
\end{lem}

In order to achieve the desired conclusions,  we need the following estimates.

\begin{lem}\label{l7.5}
Let $\wz\in A_1^{\rz,\fz}(\rn)$, then
there exists a positive constant $C$ such that
for all $f\in h^1_{\rho}(\wz)$,
\beq
\|f\|_{h^1_{\rho}(\wz)}\le C \|\wt T^+_{\rz}(f)\|_{L^1_{\wz}(\rn)},
\eeq
where
$$\wt T^+_{\rz}(f)(x)\equiv \sup_{0<t<\rz(x)}|\wt T_{t^2}(f)(x)|$$
and
$$\wt T_{t}(f)(x)\equiv \int_{\rn}\wt T_{t}(x,y)f(y)\,dy.$$
\end{lem}
\begin{proof}
Let $h(x)=(4\pi)^{-n/2}e^{-|x|^2/4}$, then it is easy to find that
$h_t(x-y)=\wt T_{t^2}(x,y)$. Now we take a nonnegative
function $\varphi\in \cdd(\rn)$ such that $\vz(x)=h(x)$
on $B(0,2)$, and we define $\vz^+_{\rz}(f)(x)$ as follows:
$$\vz^+_{\rz}(f)(x)\equiv \sup_{0<t<\rz(x)}|\vz_t*f(x)|.$$
Clearly, for any $x\in\rn$, we have
\beq
\vz^+(f)(x)\le\vz^+_{\rz}(f)(x),
\eeq
see (3.4) for the definition of $\vz^+(f)(x)$.

Let $f\in h^1_{\rz}(\wz)$, for every $N>0$ we have:
\beqs
\begin{aligned}
&\l\|\vz^+_{\rz}(f)-\wt T^+_{\rz}(f)\r\|_{L^1_{\wz}(\rn)}\\
&\le \int_{\rn} \sup_{0<t<\rz(x)}
|\vz_t*f(x)-h_t*f(x)|\,\wz(x)\,dx\\
\end{aligned}
\eeqs
\beqs
\begin{aligned}
&\le \int_{\rn} \l(\sup_{0<t<\rz(x)}
t^{-n} \int_{\rn}
|f(y)|\l|\vz\l(\frac{x-y}{t}\r)-h\l(\frac{x-y}{t}\r)\r|\,dy\r)\wz(x)\,dx\\
&\le \int_{\rn} \l(\sup_{0<t<\rz(x)}
t^{-n} \int_{\rn}
|f(y)|\l|\vz\l(\frac{x-y}{t}\r)-h\l(\frac{x-y}{t}\r)\r| \chi_{\{|y-x|>t\}}(y)\,dy\r)\wz(x)\,dx\\
&\le C \int_{\rn}\l( \int_{\rn}|f(y)|
\sup_{0<t<\rz(x)} t^{-n}  \l(1+\frac{|x-y|}{t}\r)^{-N}\chi_{\{|y-x|>t\}}(y)\,dy\r)\wz(x)\,dx\\
&\le C\int_{\rn}|f(y)|\l( \int_{\rn}
(\rz(x))^{-n} \l(1+\frac{|x-y|}{\rz(x)}\r)^{-N}
 \wz(x)\,dx\r)dy.
\end{aligned}
\eeqs
In the last inequality, we used the following facts that
$$\sup_{0<t<\rz(x)} t^{-n}  \l(1+\frac{|x-y|}{t}\r)^{-N}\le (\rz(x))^{-n} \l(1+\frac{|x-y|}{\rz(x)}\r)^{-N},$$  provided that
$|x-y|>t$ and $N>2n$.

We now estimate the inner integral in the last inequality. In fact,
\beqs
\begin{aligned}
&\int_{\rn}
(\rz(x))^{-n} \l(1+\frac{|x-y|}{\rz(x)}\r)^{-N}  \wz(x)\,dx\\
&=\int_{|x-y|<\rz(y)} (\rz(x))^{-n} \l(1+\frac{|x-y|}{\rz(x)}\r)^{-N} \wz(x)\,dx\\
&\qquad+\int_{|x-y|\ge \rz(y)} (\rz(x))^{-n} \l(1+\frac{|x-y|}{\rz(x)}\r)^{-N} \wz(x)\,dx\\
&\equiv I+II.
\end{aligned}
\eeqs
For $I$,
since $N$ is large enough and (2.2), we have
\beqs
\begin{aligned}
I&\le \frac{C}{(\rz(y))^n}
\int_{|x-y|<\rz(y)} \wz(x)\,dx
\le C  \Psi_\tz(\wt B_0) M_{V,\tz}(\wz)(y)\le C \wz(y),
\end{aligned}
\eeqs
where $\wt B_0=B(y,\rz(y))$.

\noi For $II$, by the same reason as above, we have
\beqs
\begin{aligned}
II&\le C\sum_{i=1}^{\fz}
\int_{|x-y|\sim 2^i\rz(y)} (\rz(x))^{N-n}|x-y|^{-N}\wz(x)\,dx\\
&\le C \sum_{i=1}^{\fz}
\int_{|x-y|\sim 2^i\rz(y)} (\rz(y))^{N-n}
\l(1+ \frac{|x-y|}{\rz(y)}\r)^{\frac{k_0(N-n)}{k_0+1}} |x-y|^{-N}\wz(x)\,dx\\
&\le C \sum_{i=1}^{\fz}
\int_{|x-y|\sim 2^i\rz(y)} (\rz(y))^{N-n}
\l(1+ 2^i\r)^{\frac{k_0(N-n)}{k_0+1}} (2^i\rz(y))^{-N}\wz(x)\,dx\\
\end{aligned}
\eeqs
\beqs
\begin{aligned}
&\le C \sum_{i=1}^{\fz}
(2^{-i})^{\frac{N+nk_0}{k_0+1}}
\frac{1}{(\rz(y))^n} \int_{|x-y|<2^i\rz(y)} \wz(x)\,dx\\
&\le C \sum_{i=1}^{\fz} (2^{-i})^{\frac{N+nk_0}{k_0+1}}
\l(1+2^i\r)^{\tz} M_{V,\tz}(\wz)(y)\\
&\le C \sum_{i=1}^{\fz} (2^{-i})^{\frac{N+nk_0}{k_0+1}-\tz} \wz(y)
\le C\wz(y),
\end{aligned}
\eeqs
and the last inequality holds because the real number $N$ is large enough.

\noi Combining the above two estimates, we get
\beq
\l\|\vz^+_{\rz}(f)-\wt T^+_{\rz}(f)\r\|_{L^1_{\wz}(\rn)}
\le C\int_{\rn}|f(y)|\wz(y)\,dy
=C\|f\|_{L^1_{\wz}(\rn)}.
\eeq
In addition, it is easy to get
$\|f\|_{L^1_{\wz}(\rn)}\le\|\wt T^+_{\rz}f\|_{L^1_{\wz}(\rn)}$.
Therefore, we obtain
\beq
\l\|\vz^+_{\rz}(f)\r\|_{L^1_{\wz}(\rn)}
\le \l\|\wt T^+_{\rz}(f)\r\|_{L^1_{\wz}(\rn)}+C \|f\|_{L^1_{\wz}(\rn)}
\le C\l\|\wt T^+_{\rz}(f)\r\|_{L^1_{\wz}(\rn)}.
\eeq

Finally, from Theorem 3.2, (7.10) and (7.12), it follows that
$$\|f\|_{h^1_{\rz}(\wz)}
\le C \l\|\vz^+(f)\r\|_{L^1_{\wz}(\rn)}
\le C \l\|\vz^+_{\rz}(f)\r\|_{L^1_{\wz}(\rn)}
\le C \|\wt T^+_{\rz}(f)\|_{L^1_{\wz}(\rn)},$$
which finishes the proof.
\end{proof}

For $x,y\in\rn$, set $E_t(x,y)=T_{t^2}(x,y)-\wt T_{t^2}(x,y)$,
$$T^+_{\rz}(f)(x)\equiv \sup_{0<t<\rz(x)}| T_{t^2}(f)(x)|
\quad\text{and} \quad E^+_{\rz}(f)(x)\equiv \sup_{0<t<\rz(x)}| E_{t}(f)(x)|.$$

\begin{lem}\label{l7.6}
Let $\wz\in A_1^{\rz,\fz}(\rn)$, then
there exists a positive constant $C$ such that
for all $f\in L^1_{\wz}(\rn)$,
\beqs
\|E^+_{\rz}(f)\|_{L^1_{\wz}(\rn)}\le C \|f\|_{L^1_{\wz}(\rn)}.
\eeqs
\end{lem}
\begin{proof}
By Lemma 2.2, it suffices to prove that for all $j$,
\beq
\|E^+_{\rz}(\chi_{B^*_j}f)\|_{L^1_{\wz}(\rn)}\le C \|\chi_{B^*_j}f\|_{L^1_{\wz}(\rn)},
\eeq
in which $B_j=B(x_j,\rz(x_j))$.
For any $x\in B^{**}_j$ and $y\in B^{*}_j$, since
$\rz(y)\sim\rz(x_j)\sim\rz(x)$ via Lemma 2.1,
by (7.5) we have
$$|E_t(x,y)|\le C \frac{(|x-y|/\rz(x))^{\vez}}{|x-y|^n}
\le \frac{C}{|x-y|^{n-\vez}(\rz(x_j))^{\vez}},$$
which implies that
\beqs
\begin{aligned}
\int_{B_j^{**}}&\sup_{0<t<\rz(x)} |E_{t}(\chi_{B^*_j}f)|\wz(x)\,dx\\
&\le C \int_{B_j^{**}}
\l(\int_{B_j^{*}} \frac{|f(y)|}{|x-y|^{n-\vez}(\rz(x_j))^{\vez}}\,dy\r)\wz(x)\,dx\\
&\le C \int_{B_j^{*}}
\l(\int_{B_j^{**}}  \frac{\wz(x)}{|x-y|^{n-\vez}(\rz(x_j))^{\vez}}\,dx\r)|f(y)|\,dy\\
& \le C \int_{B_j^{*}}
\l(\sum_{k=-2}^{\fz} \int_{|x-y|\sim 2^{-k}\rz(x_j)}
\frac{\wz(x)}{|x-y|^{n-\vez}(\rz(x_j))^{\vez}}\,dx\r)|f(y)|\,dy\\
& \le C \int_{B_j^{*}}
\l(\sum_{k=-2}^{\fz}
\frac{\wz(B(y,2^{-k}\rz(x_j)))}{(2^{-k}\rz(x_j))^{n-\vez}(\rz(x_j))^{\vez}}\,dx\r)
|f(y)|\,dy\\
& \le C \int_{B_j^{*}}
\l(\sum_{k=-2}^{\fz} \frac{1}{2^{k\vez}}\l(1+C_02^{k_0-k}\r)^{\tz}
\wz(y)\r) |f(y)|\,dy\\
& \le C \int_{B_j^{*}} |f(y)| \wz(y) \,dy
= C \|\chi_{B^*_j}f\|_{L^1_{\wz}(\rn)}.
\end{aligned}
\eeqs
For any $x\in (B^{**}_j)^{\complement}$ and $y\in {B^*_j}$,
it is easy to see that $\rz(x_j)\lesssim|x-x_j|\sim|x-y|$;
in addition,  by (2.2) and (7.7), we have
$0<t<\rz(x)\lesssim |x-x_j|^{k_0/(k_0+1)}(\rz(x_j))^{1/(k_0+1)}$
and $E_t(x,y)\lesssim t^N/|x-y|^{N+n}\sim t^N/|x-x_j|^{N+n}$
for any $N>0$. Therefore, taking $N>(k_0+1)\tz$, we have
\beqs
\begin{aligned}
\int_{(B_j^{**})^{\complement}}
&\sup_{0<t<\rz(x)} |E_{t}(\chi_{B^*_j}f)|\,\wz(x)\,dx\\
&\le C \int_{(B_j^{**})^{\complement}}
\l(\int_{B_j^{*}} \frac{(\rz(x_j))^{\frac{N}{k_0+1}}|f(y)|}
{|x-x_j|^{n+\frac{N}{k_0+1}}}\,dy\r)
\wz(x)\,dx\\
&\le C \int_{B_j^{*}}
\l(\int_{(B_j^{**})^{\complement}} \frac{(\rz(x_j))^{\frac{N}{k_0+1}}\wz(x)}
{|x-x_j|^{n+\frac{N}{k_0+1}}}\,dx\r)
|f(y)|\,dy\\
&\le C \int_{B_j^{*}}
\l(\sum_{i=2}^{\fz}\int_{|x-x_j|\sim 2^i\rz(x_j)}
\frac{(\rz(x_j))^{\frac{N}{k_0+1}}\wz(x)}
{|x-x_j|^{n+\frac{N}{k_0+1}}}\,dx\r)
|f(y)|\,dy\\
&\le C \int_{B_j^{*}}
\l(\sum_{i=2}^{\fz}
\frac{(\rz(x_j))^{\frac{N}{k_0+1}}\wz(B(x_j,2^i\rz(x_j)))}
{(2^i\rz(x_j))^{n+\frac{N}{k_0+1}}}\,dx\r)
|f(y)|\,dy\\
&\le C \int_{B_j^{*}}
\l(\sum_{i=2}^{\fz}
\frac{(1+2^i)^{\tz}}{(2^i)^{\frac{N}{k_0+1}}}\wz(y)\r)
|f(y)|\,dy\\
&\le C \int_{B_j^{*}} |f(y)| \wz(y)\,dy
= C \|\chi_{B^*_j}f\|_{L^1_{\wz}(\rn)},
\end{aligned}
\eeqs
which completes the proof of (7.13) and hence the proof of lemma.
\end{proof}

Next we give several estimates about
$(p,\,q,\,s)_{\wz}$-atoms and
$(p,\,q)_{\wz}$-single-atom,
which are important for our conclusion.

\begin{lem}\label{l7.7}
Let $a$ be a $(p,q,s)_\wz$-atom, and $\supp \, a\subset Q(x_0,r)$,
then for any $x\in(4Q)^{\complement}$, we have following estimates:
\begin{enumerate}
\item[\rm(i)] If $L_2\rz(x_0)\le r\le L_1\rz(x_0)$, then for any $M>0$,
$$\T^{*}a(x)\lesssim \|a\|_{L^1(\rn)}\frac{r^M}{|x-x_0|^{n+M}},$$
\item[\rm(ii)] If $r< L_2\rz(x_0)$ and $|x-x_0|\le 2\rz(x_0)$, then
there exists $\dz>0$  such that
$$\T^{*}a(x)\lesssim \|a\|_{L^1(\rn)}\frac{r^{\dz}}{|x-x_0|^{n+\dz}},$$
\item[\rm(iii)] If $r< L_2\rz(x_0)$ and $|x-x_0|\ge \rz(x_0)/\sqrt{n}$,
then there exists $\dz>0$  such that for any $M>0$,
$$\T^{*}a(x)\lesssim \|a\|_{L^1(\rn)}\frac{r^{\dz}}{|x-x_0|^{n+\dz}}
\l(\frac{\rz(x_0)}{|x-x_0|}\r)^M.$$
\end{enumerate}
\end{lem}
\begin{proof}
If $L_2\rz(x_0)\le r\le L_1\rz(x_0)$,
since $|x-y|\sim|x-x_0|$ and $\rz(y)\sim\rz(x_0)$
for  $x\in(4Q)^{\complement}$ and $y\in Q$,
by Lemma 7.1, for any $M>0$, we have
\beqs
\begin{aligned}
T_ta(x)&\le \int_{\rn}|T_t(x,y)||a(y)|\,dy\\
&\lesssim \int_{Q} \l(1+\frac{|x-y|}{\rz(y)}\r)^{-M}
|x-y|^{-n} |a(y)|\,dy\\
&\lesssim \int_{Q} \l(1+\frac{|x-x_0|}{\rz(x_0)}\r)^{-M}
|x-x_0|^{-n} |a(y)|\,dy\\
&\lesssim \|a\|_{L^1(\rn)}\frac{\rz(x_0)^{M}}{|x-x_0|^{n+M}}
\lesssim \|a\|_{L^1(\rn)}\frac{r^{M}}{|x-x_0|^{n+M}},
\end{aligned}
\eeqs
and then we obtain (i).

If $r< L_2\rz(x_0)$,
by the moment condition of $a$ and Lemma 7.3,
for any $M>0$ and $y'\in Q$ which satisfies $|y-y'|<\sqrt{t}$,
we have
\beqs
\begin{aligned}
T_ta(x)&= \int_{\rn}T_t(x,y)a(y)\,dy\\
&= \int_{Q} \l(T_t(x,y)-T_t(x,y')\r) a(y)\,dy\\
&\lesssim \int_{Q} \l(\frac{|y-y'|}{\sqrt{t}}\r)^{\dz} t^{-\frac{n}{2}}
 \l(1+\frac{\sqrt{t}}{\rz(y)}\r)^{-M}
  \exp\l(-\frac{c|x-y|^2}{t}\r) |a(y)|\,dy\\
&\lesssim \int_{Q} \l(\frac{r}{\sqrt{t}}\r)^{\dz} t^{-\frac{n}{2}}
\l(1+\frac{\sqrt{t}}{\rz(x_0)}\r)^{-M}
\l(\frac{t}{|x-x_0|^2}\r)^K |a(y)|\,dy,
\end{aligned}
\eeqs
where $K>0$ is any real number.

\noi For $|x-x_0|\le 2\rz(x_0)$, taking $K=(n+\dz)/2$,
we obtain
\beqs
\begin{aligned}
T_ta(x)
&\lesssim \int_{Q} \l(\frac{r}{\sqrt{t}}\r)^{\dz} t^{-\frac{n}{2}}
\l(1+\frac{\sqrt{t}}{\rz(x_0)}\r)^{-M}
\l(\frac{t}{|x-x_0|^2}\r)^K |a(y)|\,dy,\\
&\lesssim \|a\|_{L^1(\rn)} \l(\frac{r}{\sqrt{t}}\r)^{\dz} t^{-\frac{n}{2}}
\l(\frac{t}{|x-x_0|^2}\r)^K\\
&= \|a\|_{L^1(\rn)}\frac{r^{\dz}}{|x-x_0|^{n+\dz}},
\end{aligned}
\eeqs
which implies (ii).

\noi For $|x-x_0|\ge \rz(x_0)/\sqrt{n}$,
taking $K=(n+M+\dz)/2$, we obtain
\beqs
\begin{aligned}
T_ta(x)
&\lesssim \int_{Q} \l(\frac{r}{\sqrt{t}}\r)^{\dz} t^{-\frac{n}{2}}
\l(1+\frac{\sqrt{t}}{\rz(x_0)}\r)^{-M}
\l(\frac{t}{|x-x_0|^2}\r)^K |a(y)|\,dy,\\
&\lesssim \|a\|_{L^1(\rn)} \l(\frac{r}{\sqrt{t}}\r)^{\dz} t^{-\frac{n}{2}}
\l(\frac{\rz(x_0)}{\sqrt{t}}\r)^M
\l(\frac{t}{|x-x_0|^2}\r)^K\\
&= \|a\|_{L^1(\rn)}\frac{r^{\dz}}{|x-x_0|^{n+\dz}}\l(\frac{\rz(x_0)}{|x-x_0|}\r)^M,
\end{aligned}
\eeqs
which finishes the proof of lemma.
\end{proof}

\begin{lem}\label{l7.8}
Let $\wz\in A_q^{\rz,\,\tz}(\rn)$ and $a$ be a $(p,q,s)_\wz$-atom,
which satisfies $\supp\, a\subset Q(x_0,r)$,
then there exists a constant $C$ such that:
$$\|a\|_{L^1(\rn)}\le C |Q| \wz(Q)^{-1/p}\Psi_{\tz}(Q).$$
\end{lem}
\begin{proof}
If $q>1$, by H\"older inequality and the definition of
$A_q^{\rz,\,\tz}(\rn)$ weights,
we have
\beqs
\begin{aligned}
\|a\|_{L^1(\rn)}
&=\int_{Q}|a(x)| \wz(x)^{1/q} \wz(x)^{-1/q}\,dx\\
&\le \|a\|_{L^q_{\wz}(\rn)}
\l(\int_{Q}\wz(x)^{-q'/q}\,dx\r)^{1/q'}\\
&\le \wz(Q)^{1/q-1/p} \l(\int_{Q}\wz(x)^{-q'/q}\,dx\r)^{1/q'}
\l(\int_{Q}\wz(x)\,dx\r)^{1/q} \wz(Q)^{-1/q}\\
&\le  C |Q| \wz(Q)^{-1/p}\Psi_{\tz}(Q).
\end{aligned}
\eeqs
If $q=1$, we have
$$\wz(Q)\le C |Q| \Psi_{\tz}(Q) \inf_{x\in Q} \wz(x),$$
which implies
$$\|\wz^{-1}\|_{L^{\fz}(Q)}\le C |Q| \wz(Q)^{-1}\Psi_{\tz}(Q).$$
Therefore, we get
$$\|a\|_{L^1(\rn)}\le \|a\|_{L^1_{\wz}(\rn)}\|\wz^{-1}\|_{L^{\fz}(Q)}
\le C |Q| \wz(Q)^{-1/p}\Psi_{\tz}(Q),$$
which finishes the proof.
\end{proof}

Combining above two lemmas with $\Psi_{\tz}(Q)\lesssim 1$,
we can get the following corollary.

\begin{cor}\label{c7.1}
Let $a$ be a $(p,q,s)_\wz$-atom, and $\supp \, a\subset Q(x_0,r)$,
then for any $x\in(4Q)^{\complement}$, we have following estimates:
\begin{enumerate}
\item[\rm(i)] If $L_2\rz(x_0)\le r\le L_1\rz(x_0)$, then for any $M>0$,
$$\T^{*}a(x)\lesssim \wz(Q)^{-1/p} \l(\frac{r}{|x-x_0|}\r)^{n+M},$$
\item[\rm(ii)] If $r< L_2\rz(x_0)$ and $|x-x_0|\le 2\rz(x_0)$, then
there exists $\dz>0$  such that
$$\T^{*}a(x)\lesssim \wz(Q)^{-1/p} \l(\frac{r}{|x-x_0|}\r)^{n+\dz},$$
\item[\rm(iii)] If $r< L_2\rz(x_0)$ and $|x-x_0|\ge \rz(x_0)/\sqrt{n}$,
then there exists $\dz>0$  such that for any $M>0$,
$$\T^{*}a(x)\lesssim \wz(Q)^{-1/p} \l(\frac{r}{|x-x_0|}\r)^{n+\dz}
\l(\frac{\rz(x_0)}{|x-x_0|}\r)^M.$$
\end{enumerate}
\end{cor}

Next we give the main theorem of this section.

\begin{thm}\label{t7.1}
Let $0\not\equiv V\in RH_{n/2}$ and $\wz\in A_1^{\rz,\,\fz}(\rn)$,
then $h^1_{\rho}(\wz)=H^1_{\L}(\wz)$ with equivalent norms,
that is
$$\|f\|_{h^1_{\rho}(\wz)}\sim\|f\|_{H^1_{\L}(\wz)}.$$
\end{thm}
\begin{proof}
Assume that $f\in H^1_{\L}(\wz)$, by (7.7), we have
\beq
|f(x)|=\lim_{t<\rz(x),\,t\to 0}|\wt T_t(f)(x)|
\le T^+_{\rz}(f)(x)+C\lim_{t\to 0}\l(\frac{t}{\rz(x)}\r)^{\dz}M(f)(x)
\le T^+_{\rz}(f)(x).
\eeq
Then according to (7.14), Lemma 7.5 and Lemma 7.6, we get
$f\in h^1_{\rho}(\wz)$ and
\beqs
\begin{aligned}
\|f\|_{h^1_{\rho}(\wz)}
&\lesssim \|\wt T^+_{\rz}(f)\|_{L^1_{\wz}(\rn)}
\lesssim \| T^+_{\rz}(f)\|_{L^1_{\wz}(\rn)} +\|E^+_{\rz}(f)\|_{L^1_{\wz}(\rn)}\\
& \lesssim \| T^+_{\rz}(f)\|_{L^1_{\wz}(\rn)} +\|f\|_{L^1_{\wz}(\rn)}
\lesssim \| T^+_{\rz}(f)\|_{L^1_{\wz}(\rn)}\\
& \lesssim \|\T^* (f)\|_{L^1_{\wz}(\rn)}=\|f\|_{H^1_{\L}(\wz)}.
\end{aligned}
\eeqs
Conversely, we need to prove that $\T^*$ is bounded from $h^1_{\rho}(\wz)$
to $L^1_{\wz}(\rn)$. To end this, by Lemma 2.4 and Theorem 5.1, it suffices to prove that
for any $(1,q,s)_{\wz}$-atom or $(1,q)_{\wz}$-single-atom $a$,
\beq
\|\T^*(a)\|_{L^1_{\wz}(\rn)}\lesssim 1,
\eeq
where $1< q\le 1+\dz/n$.

\noi If $a$ is a $(1,q)_{\wz}$-single-atom, by H\"older inequality and Lemma 7.4, we have
$$\|\T^*(a)\|_{L^1_{\wz}(\rn)}\le \|\T^*(a)\|_{L^q_{\wz}(\rn)}\,\wz(\rn)^{1-1/q}
\le C \|a\|_{L^q_{\wz}(\rn)}\,\wz(\rn)^{1-1/q}\lesssim 1.$$

\noi If $a$ is a $(1,q,s)_{\wz}$-atom and $\supp \, a\subset Q(x_0,r)$ with
$r\le L_{1}\rz(x_0)$,
then we have
$$\|\T^*(a)\|_{L^1_{\wz}(\rn)}\le \|\T^*(a)\|_{L^1_{\wz}(4Q)}
+\|\T^*(a)\|_{L^1_{\wz}((4Q)^{\complement})}\equiv I+II.$$
For $I$, by H\"older inequality, Lemma 2.4 and Lemma 7.4, we get
\beqs
\begin{aligned}
\|\T^*(a)\|_{L^1_{\wz}(4Q)}
&\le \|\T^*(a)\|_{L^q_{\wz}(4Q)}\,\wz(4Q)^{1-1/q}
\le C \|a\|_{L^q_{\wz}(\rn)}\,\wz(4Q)^{1-1/q}\\
&\le C \l(\wz(4Q)/\wz(Q)\r)^{1-1/q}
\lesssim 1.
\end{aligned}
\eeqs
For $II$, if $L_2\rz(x_0)\le r\le L_1\rz(x_0)$, by Lemma 2.4 and Corollary 7.1,
taking  $M>q(n+\tz)-n$, we have
\beqs
\begin{aligned}
\|\T^*(a)\|_{L^1_{\wz}((4Q)^{\complement})}
&=\sum_{j=3}^{\fz} \int_{2^j Q \backslash 2^{j-1}Q} \T^*(a)(x) \wz(x)\,dx\\
&\lesssim \frac{1}{\wz(Q)}
\sum_{j=3}^{\fz} \int_{2^j Q \backslash 2^{j-1}Q}
\l(\frac{r}{|x-x_0|}\r)^{n+M}  \wz(x)\,dx\\
&\lesssim \frac{1}{\wz(Q)}
\sum_{j=3}^{\fz} 2^{-j(n+M)}\wz(2^j Q)\\
&\lesssim \sum_{j=3}^{\fz}
2^{-j(n+M)}2^{jnq}\l(1+\frac{2^jr}{\rz(x_0)}\r)^{q\tz}\\
&\lesssim \sum_{j=3}^{\fz}
2^{-j[n+M-nq-q\tz]}\lesssim 1;
\end{aligned}
\eeqs
if $r< L_2\rz(x_0)$, then there exists $N_0\in\zz$ such that
$2^{N_0-1}\sqrt{n}r\le \rz(x_0)<2^{N_0}\sqrt{n}r$.
Let us assume that $N_0\ge 3$, otherwise, we just need to consider
the  $I_2$ in the following decomposition:
$$\|\T^*(a)\|_{L^1_{\wz}((4Q)^{\complement})}=
\l(\sum_{j=3}^{N_0}+\sum_{j=N_0+1}^{\fz}\r)
\int_{2^j Q \backslash 2^{j-1}Q} \T^*(a)(x) \wz(x)\,dx \equiv I_1+I_2,$$
for $I_1$, since $|x-x_0|<2^j\sqrt{n}r\le 2^{N_0}\sqrt{n}r\le 2\rz(x_0)$,
$\Psi_{\tz}(2^jQ)\le 3^{\tz}$ and $q<1+\dz/n$,
by Lemma 2.4 and Corollary 7.1, we get
\beqs
\begin{aligned}
I_1
&=\sum_{j=3}^{N_0} \int_{2^j Q \backslash 2^{j-1}Q} \T^*(a)(x) \wz(x)\,dx\\
&\lesssim \frac{1}{\wz(Q)}
\sum_{j=3}^{N_0} \int_{2^j Q \backslash 2^{j-1}Q}
\l(\frac{r}{|x-x_0|}\r)^{n+\dz}  \wz(x)\,dx\\
&\lesssim \frac{1}{\wz(Q)}
\sum_{j=3}^{N_0} 2^{-j(n+\dz)}\wz(2^j Q)\\
&\lesssim \sum_{j=3}^{N_0}
2^{-j[n+\dz-nq]}\lesssim 1,
\end{aligned}
\eeqs
for $I_2$, since $|x-x_0|\ge 2^{j-1}r\ge 2^{N_0}r\ge \rz(x_0)/\sqrt{n}$,
then $\Psi_{\tz}(2^jQ)\le (2^{j+1}\sqrt{n}r/\rz(x_0))^{\tz}$,
thus, taking $M=q\tz$, by $q<1+\dz/n$, Lemma 2.4 and Corollary 7.1, we obtain
\beqs
\begin{aligned}
I_2
&=\sum_{j=N_0+1}^{\fz} \int_{2^j Q \backslash 2^{j-1}Q} \T^*(a)(x) \wz(x)\,dx\\
&\lesssim \frac{1}{\wz(Q)}
\sum_{j=N_0+1}^{\fz} \int_{2^j Q \backslash 2^{j-1}Q}
\l(\frac{r}{|x-x_0|}\r)^{n+\dz}
\l(\frac{\rz(x_0)}{|x-x_0|}\r)^M
\wz(x)\,dx\\
&\lesssim \frac{1}{\wz(Q)}
\sum_{j=N_0+1}^{\fz} 2^{-j(n+\dz)}\wz(2^j Q)
\l(\frac{\rz(x_0)}{2^jr}\r)^M\\
&\lesssim \sum_{j=N_0+1}^{\fz}
2^{-j[n+\dz-nq]}
(\Psi_{\tz}(2^jQ))^{q} \l(\frac{\rz(x_0)}{2^jr}\r)^M
\lesssim 1,
\end{aligned}
\eeqs
which finally implies the (7.15)
and finishes the proof.
\end{proof}

\bigskip

\noi{\bfseries{Acknowledgements}}\quad
The research is supported by National Natural Science Foundation of China under Grant
\#11271024.

\bigskip
\bigskip

\indent Beijing International Studies University,\\
\indent Beijing, 100024, \\
\indent People¡¯s Republic of China

\textsl{E-mail address: zhuhua@pku.edu.cn}

\bigskip
\bigskip

 LMAM, School of Mathematical  Science,\\
\indent Peking University,\\
\indent Beijing, 100871,\\
\indent People¡¯s Republic of China

\textsl{E-mail address:  tanglin@math.pku.edu.cn}

\end{document}